\documentclass[11pt]{amsart}

\usepackage{latexsym}
\usepackage{amssymb}
\usepackage{amsmath}
\usepackage{color}
\usepackage{bbm}

\usepackage{graphicx}
\usepackage{tcolorbox}
\usepackage{tabularx}

\usepackage{tikz}
\usepackage{xparse}

\NewDocumentCommand\DownArrow{O{2.0ex} O{black}}{%
   \mathrel{\tikz[baseline] \draw [<-, line width=0.5pt, #2] (0,0) -- ++(0,#1);}
}
\NewDocumentCommand\UpArrow{O{2.0ex} O{black}}{%
   \mathrel{\tikz[baseline] \draw [->, line width=0.5pt, #2] (0,0) -- ++(0,#1);}
}

\newtheorem{theorem}{Theorem}[section]
\newtheorem{lemma}[theorem]{Lemma}
\newtheorem{proposition}[theorem]{Proposition}
\newtheorem{corollary}[theorem]{Corollary}
\newtheorem{definition}[theorem]{Definition}

\newtheorem{remark}[theorem]{Remark}

\newtheorem{question}[theorem]{Question}

\newcommand\supp{\mathop{\rm supp}}

\newcommand\id{\mathop{\rm id}}

\newcommand\nph{\varphi}

\newcommand{\cl}[1]{\mathcal{#1}}
\newcommand{\bb}[1]{\mathbb{#1}}

\newcommand\Tr{\mathop{\rm Tr}}

\begin{document}

\title[Quantum hypergraph homomorphisms]
{Quantum hypergraph homomorphisms and non-local games}


\author[G. Hoefer]{Gage Hoefer}
\address{School of Mathematical Sciences\\ University of Delaware\\ 501 Ewing Hall\\ Newark\\ DE 19716\\ USA}\email{ghoefer@udel.edu}

\author[I.G. Todorov]{Ivan G. Todorov}
\address{School of Mathematical Sciences\\ University of Delaware\\ 501 Ewing Hall\\ Newark\\ DE 19716\\ USA} 
\email{todorov@udel.edu}


\date{3 November 2022}


\maketitle

\begin{abstract}
Using the simulation paradigm in information theory, we define notions of quantum hypergraph homomorphisms and quantum hypergraph isomorphisms, 
and show that they constitute partial orders and equivalence relations, respectively. 
Specialising to the case where the underlying hypergraphs arise from non-local games, 
we define notions of quantum non-local game homomorphisms and quantum non-local game isomorphisms, 
and show that games, isomorphic with respect to a given correlation type, 
have equal values and asymptotic values relative to this type. 
We examine a new class of no-signalling correlations, which witness the existence of 
non-local game homomorphisms,
and characterise them in terms of states 
on tensor products of canonical operator systems. We define jointly synchronous correlations
and show that they correspond to traces on the tensor product of the canonical C*-algebras
associated with the game parties. 
\end{abstract}

\tableofcontents


\section{Introduction}\label{s_intro}

The connections between operator algebra theory and quantum information theory are  
undergoing at present a phase of intensive development. 
One of the chief catalysts for this trend was the equivalence between  
the Connes Embedding Problem in von Neumann algebra theory 
and the Tsirelson Problem in quantum physics, established in \cite{fritz, jnvwy, oz}. 
Its culminations are arguably the 
refutal of the weak Tsirelson Problem in \cite{slofstra_JAMS},  
the demonstration of the non-closedness of the set of quantum correlations 
presented in \cite{slofstra} (see also \cite{dpp})
and, more recently,
the announced in \cite{jnvwy} resolution of the aforementioned Connes Embedding Problem. 

Behind the developments in \cite{slofstra_JAMS}, \cite{dpp} and \cite{jnvwy} one discerns the role of 
non-local games and their optimal winning probabilities. These objects were first studied from the perspective of 
quantum information theory (see e.g. \cite{chtw, cleve-mittal, mermin}), where they can effectively
witness entanglement, leading to proofs of Bell's Theorem \cite{bell}. Important combinatorial 
ramifications were established in \cite{cmnsw} and later in \cite{mr, amrssv},  
where quantum graph homomorphisms and quantum graph isomorphisms were defined and studied. 
A wealth of mathematical developments in non-local game theory 
took place within the past decade \cite{dpp, lmprsstw, lmr, mptw, pr, pt, rs}, relying on operator system theory, 
quantum group theory, quantum information theory and combinatorics, among others.

A non-local game is a cooperative game, played by two players, Alice and Bob, 
against a verifier. In each round of the game, 
the verifier selects an input pair $(x,y)$ from the cartesian product of two finite sets $X$ and $Y$, 
following a probability distribution $\pi$ on $X\times Y$, and 
sends input $x$ to Alice and input $y$ to Bob. 
Alice produces an output $a$ lying in a specified set $A$, and Bob -- an 
output $b$ lying in a specified set $B$; 
the combination $(x,y,a,b)$ yields a win for the players if it satisfies a previously fixed predicate 
$\lambda : X\times Y \times A\times B \to \{0,1\}$, representing the rules of the game. 

During the course of the game, the players are not allowed to communicate; mathematically this 
is expressed by saying that the
probabilistic strategies $p = \{(p(a,b|x,y))_{(a,b)\in A\times B} : (x,y)\in X\times Y\}$ that they are 
allowed to use are no-signalling correlations, that is, correlations $p(a,b|x,y)$ with well-defined 
conditional marginals $p(a|x)$ and $p(b|y)$. 
Several types of strategies are usually used:
local (corresponding to classical resources), quantum 
(corresponding to finite dimensional entanglement), quantum approximate (corresponding to liminal 
entanglement) and quantum commuting (arising from the commuting model of quantum mechanics). 
Each correlation type gives rise to a corresponding game value and asymptotic game value: 
these are, respectively, the optimal winning probability in one round, and the 
optimal winning probability in the limit when independent rounds, forming an infinite sequence, are 
conducted.

One of the main motivations behind the present work is to identify conditions, upon which 
seemingly different games may have the same value with respect to a given strategy class. 
We propose notions of game homomorphisms and game isomorphisms, associated with a fixed 
correlation type. The existence of a homomorphism from a game $G_1$ into a game $G_2$ of type ${\rm t}$
leads to an inequality between the two ${\rm t}$-values, while the existence of an isomorphism 
of type ${\rm t}$ -- to an equality of these values. 

In order to define game homomorphisms (resp. game isomorphisms) of a given type, we 
embed the two games into a larger game; 
we can think of the new game as a non-local \lq\lq super-game'', played by the 
verifiers of the original games and controlled by a \lq\lq super-verifier''. 
More specifically, suppose that game $G_i$ has inputs from $X_i\times Y_i$
and outputs from $A_i\times B_i$, $i = 1,2$. 
The super-verifier sends a pair $x_2y_2$ from $X_2\times Y_2$ to the verifier of $G_1$, and
a pair $a_1b_1$ from $A_1\times B_1$ to the verifier of $G_2$. The two verifiers return pairs 
$x_1y_1$ from $X_1\times Y_1$ and $a_2b_2$ from $A_2\times B_2$, respectively.
The rules of the super-game are appropriately determined by the rules of the individual games, requiring that 
$(x_2y_2,a_1b_1, x_1y_1,a_2b_2)$ yields a win for the super-game if 
$(x_1,y_1,a_1,b_1)$ and $(x_2,y_2,a_2,b_2)$ either simultaneously yield a win or simultaneously yield a lose 
for the games $G_1$ and $G_2$. 
Heuristically, in this setup, the goal of the two verifiers is to convince the super-verifier that 
the games $G_1$ and $G_2$ are equivalent. 

Non-local game homomorphisms of a given type ${\rm t}$ are defined using
no-signalling correlations of the same type, satisfying certain stronger conditions, which 
allow us to transport the perfect strategies of type ${\rm t}$ for the game $G_1$ 
to perfect strategies of the same type for the game 
$G_2$; the strategy transport thus achieved allows the comparison between the values of $G_1$ and $G_2$. 
In order to define non-local game isomorphisms of a given type, we employ a special kind of 
no-signalling correlations, defined in \cite{bhtt2} and called therein bicorrelations, that allow reversibility.
A game isomorphism of local type amounts to reshuffling the question-answer sets 
that transforms the rule predicates into each other. Much as in the case of 
quantum graph isomorphisms \cite{amrssv}, the quantum identification of 
non-local games is a weaker equivalence relation, designed to take into account the possible presence of 
entanglement between the participating verifiers.

The construction described in the previous paragraphs is obtained as a special case of a more 
general setting hosting hypergraph homomorphisms (resp. isomorphisms);  
this is the second main motivation behind the present paper. 
The latter is achieved by employing the simulation paradigm in information theory \cite{clmw},
according to which, starting with a classical information channel
from an alphabet $X_1$ to an alphabet $A_1$, 
using assistance with no-signalling resources over a quadruple $(X_2,A_1,X_1,A_2)$, 
one can simulate an information channel from alphabet $X_2$ to alphabet $A_2$.
Placing extra restrictions on the support of the input and the output channels, 
we define hypergraph homomorphism (resp. isomorphism) games. 
The perfect local strategies of the latter class of games correspond to 
a type of classical homomorphisms (resp. classical isomorphisms) of the underlying hypergraphs. 
Allowing non-classical resources, this leads to notions of quantum hypergraph homomorphisms (resp. 
isomorphisms). 

We now describe the content of the paper in more detail. 
After collecting some general notation at the end of the present section, 
in Section \ref{ss_gensetup} we recall
the definition of the main no-signalling correlation types and 
introduce the simulation setup, showing that simulators can be composed with preservation of their types. 
In Section \ref{s_hyper} we define an intermediate game, which we call the hypergraph quasi-homomorphism game, 
and show that quasi-homomorphism of a fixed type is a partial quasi-order on the set of all hypergraphs. 
We exhibit an example of hypergraphs that are quantum quasi-homomorphic but not locally quasi-homomorphic. 

The hypergraph quasi-homomorphism game provides
the base for defining, in Section \ref{s_hypis}, the hypergraph homomorphism game, which 
allows us to specify, for every
correlation type ${\rm t}$, a notion of ${\rm t}$-homomorphic hypergraphs. 
In order to define ${\rm t}$-isomorphic hypergraphs, we employ the notion of a classical bicorrelation of type ${\rm t}$.
The latter concept was introduced in \cite{bhtt2} by specialising the notion of a 
quantum bicorrelation studied therein and using the correlation types defined in \cite{tt}
based, in their own turn, on the setup of quantum no-signalling correlations 
of Duan and Winter \cite{dw}.
Using the examples of the separation between quantum isomorphic and locally isomorphic graphs 
\cite{amrssv}, we provide examples that separate quantum isomorphic hypergraphs from locally isomorphic ones. 
We further establish characterisations of hypergraph isomorphisms of
quantum approximate, quantum commuting and no-signalling type
in terms of states on operator system tensor products. The latter charaterisations are obtained as 
consequences of the characterisations of correlation types in \cite{lmprsstw} and \cite{bhtt2}.  

Section \ref{s_strongly} 
of the paper contains the definitions of the different types for a class of no-signalling correlations over 
a quadruple of the form $(X_2\times Y_2)\times (A_1\times B_1)\times (X_1\times Y_1)\times (A_2\times B_2)$, 
which we call strongly no-signalling, and collects some of their properties. 
In Section \ref{s_nonlocal}, we use the strongly no-signalling correlations as strategies for the 
homomorphism game between two given non-local games. This is achieved by applying the results of Section 4 in 
the case of hypergraphs that arise from non-local games. Theorem 6.1 and Theorem 6.4, in particular, establish the 
strategy transport for correlations and bicorrelations, respectively, allowing the comparisons
of the values and the asymptotic values established in Theorem 6.12. 
We also observe that homomorphism (resp. isomorphism) between non-local games, of a given type, 
is a partial quasi-order (resp. an equivalence relation). 

Section \ref{s_repSNS} is dedicated to the operator system representation of strongly no-signalling correlations. 
En route, we develop some basic multivariate tensor product theory in the operator system category, 
extending part of the work on bivariate operator systems tensor products in \cite{kptt}. 
Our results can be seen as a continuation of the characterisations of general correlation types 
in \cite{psstw, lmprsstw}.

In Section \ref{s_special}, we restrict our attention to synchronous games, a class of non-local games
first studied in \cite{psstw} and having gained prominence through a number of recent 
developments (see e.g. \cite{amrssv, dp, dpp1, dpp, hmps, jnvwy, kps, pr}). 
The usual synchronicity condition for a correlation \cite{psstw} needs to be adapted 
for the case of the super-game under consideration; this leads to the definition of 
jointly synchronous correlations.
The main result in this section is Theorem 8.3, which contains a tracial representation of 
jointly synchronous correlations, continuing the tracial characterisation thread from
\cite{psstw} and \cite{kps}. 
The strategy transport from Section \ref{s_nonlocal} specialises in the synchronous case to 
transport of traces, leading to a necessary condition on the tracial state spaces of the game 
algebras (see \cite{hmps}) whose games are quasi-homomorphic.

We point out that the work \cite{bkm} establishes a simulation scheme between 
contextuality scenarios in the sense of \cite{afls}, using the simulation paradigm in an identical way to the 
one utilised here. 
Noticing that contextuality scenarios have as their base ingredients some underlying hypergraphs, 
we see that \cite[Definition 17]{bkm} coincides with our definition of 
local hypergraph quasi-homomorphism. 
While the authors of \cite{bkm} adopt a general categorical perspective and are mainly interested in 
the consequences for generalised probabilistic theories, 
our emphasis is placed on the  
hierarchy of the different concepts one obtains by varying the available correlation resources, 
and their operator theoretic characterisation. 
As a result, we make use of, and augment as necessary, 
the existing operator algebraic techniques in the area.

We use various concepts and results from operator space theory, and
refer the reader to the monographs \cite{Pa} and \cite{Pi} for the general background. 

\smallskip

In the remainder of this section we set notation, to be used throughout the paper. 
For a finite set $X$, we let $\bb{C}^X = \oplus_{x\in X}\bb{C}$ and write $(e_x)_{x\in X}$ for the canonical 
orthonormal basis of $\bb{C}^X$. 
We denote by 
$M_X$ the algebra of all complex matrices over $X\times X$, 
and by $\cl D_X$ its subalgebra of all diagonal matrices. 
We write $\epsilon_{x,x'}$, $x,x'\in X$, for the canonical matrix units in $M_X$, denote by 
$\Tr$ the trace functional on $M_X$, and set $\langle S,T\rangle = \Tr(ST)$. 
We denote by $\cl V\otimes \cl W$ the algebraic tensor product of vector spaces $\cl V$ and $\cl W$, 
except when $\cl V$ and $\cl W$ are Hilbert spaces, in which case the notation is used for their Hilbertian 
tensor product. 
Given sets $X_i$, $i = 1,\dots,n$, we abbreviate 
$X_1\cdots X_n = X_1 \times \cdots \times X_n$, and 
write $M_{X_1\cdots X_n} = \otimes_{i=1}^n M_{X_i}$ and $\cl D_{X_1\cdots X_n} = \otimes_{i=1}^n \cl D_{X_i}$. 
We let $L_{\omega} : M_{X_1X_2}\to M_{X_1}$ be 
the slice map with respect to a given element $\omega\in M_{X_2}\equiv \cl L(\bb{C}^X)^*$; thus, 
$$L_{\omega}(T_1\otimes T_2) = \langle \omega,T_2\rangle T_1, \ \ T_i\in M_{X_i}, i = 1,2.$$
The partial trace ${\rm Tr}_{X_2} : M_{X_1X_2}\to M_{X_1}$ is the slice map with respect to the identity operator $I_{X_2}$
of $M_{X_2}$. 
Finally, for a Hilbert space $H$, we write $\cl B(H)$ for the 
C*-algebra of all bounded linear operators on $H$, and denote by 
$I_H$ the identity operator on $H$.


\section{General setup}\label{ss_gensetup}

A \emph{hypergraph} is a subset $E\subseteq V\times W$, where $V$ and $W$ are finite sets.
For $w\in W$, let $E(w) = \{v\in V : (v,w)\in E\}$. We refer to $V$ as the set of \emph{vertices} of $E$, 
and to $\{E(w) : w\in W\}$ as the set of its \emph{edges}. 
The hypergraph $E$ will be called {\it full} if for every $v\in V$
there exists $w\in W$ such that $(v,w)\in E$. 
The \emph{dual} of the hypergraph $E\subseteq V\times W$ is the hypergraph 
$$E^* := \{(w,v) : (v,w)\in E\}.$$

Let $V$ and $W$ be finite sets. 
A  \emph{(classical) information channel} from $V$ to $W$ is a positive trace preserving linear map 
$\cl E : \cl D_V \to \cl D_W$.
We write $V\stackrel{\cl E}{\mapsto} W$ and set 
$$\cl E(w|v) = \left\langle \cl E(\epsilon_{v,v}),\epsilon_{w,w} \right\rangle, \ \ \ v\in V, w\in W.$$
A channel $V\stackrel{\cl E}{\mapsto} W$ defines a hypergraph
$$E_{\cl E} = \{(v,w) \in V\times W : \cl E(w|v) > 0\}.$$
Given a hypergraph $E\subseteq V\times W$, we set 
$$\cl C(E) = \left\{V\stackrel{\cl E}{\mapsto} W \ : \ \mbox{ a channel  with } E_{\cl E}\subseteq E\right\};$$
if $\cl E\in \cl C(E)$, we say that $\cl E$ \emph{fits} $E$.  
We note that if $\cl C(E)\neq \emptyset$ then $E$ is full. 
The set $\cl C(V\times W)$ coincides with the set of all channels $V\stackrel{\cl E}{\mapsto} W$
and, if $E\subseteq V\times W$, then 
$\cl C(E)$ is a convex (with respect to the usual linear structure on the space of all linear maps 
from $\cl D_V$ to $\cl D_W$) subset of $\cl C(V\times W)$.

A channel $\cl E : \cl D_V\to \cl D_W$ is called \emph{unital} if $\cl E(I_V) = I_W$. If $\cl E$ is unital then  
$$|W| = {\rm Tr}(\cl E(I_V)) = {\rm Tr}(I_V) = |V|;$$
in this case it will be natural to assume that $V = W$. 
If $\cl E : \cl D_V\to \cl D_W$ is a unital channel then 
the map $\cl E^* : \cl D_W\to \cl D_V$, given by $\cl E^*(v|w) := \cl E(w|v)$ (or, equivalently, 
$$\langle \cl E(S),T\rangle = \langle S,\cl E^*(T)\rangle, \ \ \ S\in \cl D_V, T\in \cl D_W)$$
is also a channel. 

Let $V_i$ and $W_i$ be finite sets, $i = 1,2$. 
A \emph{no-signalling (NS) correlation} 
on the quadruple $(V_2,W_1,V_1,W_2)$ is an information channel 
$\Gamma : \cl D_{V_2W_1} \to \cl D_{V_1W_2}$
for which the \emph{marginal channels}
$$\Gamma_{V_2\to V_1} : \cl D_{V_2}\to \cl D_{V_1}, \ \ \ 
\Gamma_{V_2\to V_1}(v_1|v_2) := \sum_{w_2\in W_2} \Gamma(v_1,w_2 | v_2,w_1')$$
and 
$$\Gamma^{W_1\to W_2} : \cl D_{W_1}\to \cl D_{W_2}, \ \ \ 
\Gamma^{W_1\to W_2}(w_2|w_1) := \sum_{v_1\in V_1} \Gamma(v_1,w_2 | v_2',w_1)$$
are well-defined (independently of the choice of $w_1'$ and $v_2'$). 
In the sequel, when there is no risk of confusion, the indicating sub/superscripts in the notation 
for the marginal channels will be dropped. 
We denote by 
$\cl C_{\rm ns}$ the collection of all NS correlations (the quadruple $(V_2,W_1,V_1,W_2)$ 
will usually be understood from the context).

A \emph{positive operator-valued measure (POVM)} is a (finite) family $(E_i)_{i=1}^k$ of 
positive operators acting on a Hilbert space $H$ such that $\sum_{i=1}^k E_i = I$. 
An NS correlation $\Gamma : \cl D_{V_2W_1} \to \cl D_{V_1W_2}$ is called \emph{quantum commuting}
if there exists a Hilbert space $H$, a unit vector $\xi\in H$ and POVM's
$(E_{v_2,v_1})_{v_1\in V_1}$, $v_2\in V_2$, and $(F_{w_1,w_2})_{w_2\in W_2}$, $w_1\in W_1$, such that
$$\Gamma(v_1,w_2 | v_2,w_1) = \langle E_{v_2,v_1}F_{w_1,w_2}\xi,\xi\rangle, \ \ \ v_i\in V_i, w_i\in W_i, i = 1,2.$$
We call $\Gamma$ \emph{quantum} if it is quantum commuting and $H$ can be chosen of the form 
$H = H_V\otimes H_W$ in such a way that $H_{V}, H_{W}$ are finite dimensional, $E_{v_2,v_1} = E'_{v_2,v_1}\otimes I_{W}$ and 
$F_{w_1,w_2} = I_V\otimes F'_{w_1,w_2}$, $v_i\in V_i$, $w_i\in W_i$, $i = 1,2$. 
The correlation $\Gamma$ is \emph{approximately quantum} if it is the limit of quantum correlations, and 
\emph{local} if it is a convex combination of correlations of the form $\Gamma_V\otimes \Gamma_W$,
where $\Gamma_{V} : \cl D_{V_2}\to \cl D_{V_1}$ and $\Gamma_{W} : \cl D_{W_1}\to \cl D_{W_2}$
are information channels. 
We refer the reader to \cite{lmprsstw, psstw} for further details, and denote 
the subclasses of 
local, quantum, approximately quantum and quantum commuting NS correlations
by $\cl C_{\rm loc}$, $\cl C_{\rm q}$, $\cl C_{\rm qa}$ and $\cl C_{\rm qc}$, respectively.
We point out the inclusions
\begin{equation}\label{eq_chain}
\cl C_{\rm loc} \subseteq \cl C_{\rm q} \subseteq \cl C_{\rm qa} \subseteq \cl C_{\rm qc}
\subseteq \cl C_{\rm ns},
\end{equation}
all of them strict:
$\cl C_{\rm loc} \neq \cl C_{\rm q}$ is the Bell Theorem \cite{bell},  
$\cl C_{\rm q} \neq \cl C_{\rm qa}$ is a negative answer to the weak Tsirelson Problem
\cite{slofstra} (see also \cite{dpp, slofstra_JAMS}), 
and $\cl C_{\rm qa} \neq \cl C_{\rm qc}$ -- in view of \cite{fritz, jnpp, oz}, 
a negative answer to the announced solution of the
Connes Embedding Problem \cite{jnvwy}.

A \emph{non-local game}  on the quadruple $(V_2,W_1,V_1,W_2)$ is a hypergraph 
$\Lambda\subseteq V_2W_1 \times V_1W_2$.
For a non-local game $\Lambda$ and 
${\rm t}\in \{{\rm loc}, {\rm q}, {\rm qa}, {\rm qc}, {\rm ns}\}$,  
we write 
$\cl C_{\rm t}(\Lambda) = \cl C_{\rm t}\cap \cl C(\Lambda)$. 
The elements of $\cl C_{\rm t}(\Lambda)$ will be referred to as 
\emph{perfect ${\rm t}$-strategies} of the game $\Lambda$, and the set 
$V_2W_1$ (resp. $V_1W_2$) as the \emph{question} (resp. \emph{answer}) set for the two \emph{players}
of the game $\Lambda$. 
Note that non-local games are usually defined as a tuple $(V_2,W_1,V_1,W_2,\lambda)$, 
where $\lambda : V_2 \times W_1 \times V_1 \times W_2 \to \{0,1\}$ is a function 
(referred to as a \emph{rule function}); 
in the above definition, we have identified $\Lambda$ with the support of $\lambda$.

We recall the \emph{simulation paradigm} in information theory \cite{clmw}. 
Given an NS correlation $\Gamma$ on the quadruple 
$(V_2,W_1,V_1,W_2)$
and a channel $ V_1\stackrel{\cl E}{\mapsto} W_1$, let 
$\Gamma[\cl E] : \cl D_{V_2} \to \cl D_{W_2}$ be the linear map, given by 
\begin{equation}\label{eq_G[E]}
\Gamma[\cl E](w_2|v_2) = \sum_{v_1\in V_1} \sum_{w_1\in W_1} 
\Gamma(v_1,w_2|v_2,w_1) \cl E(w_1|v_1).
\end{equation}
It is straightforward to check that $\Gamma[\cl E]$ is a channel (see \cite{clmw}) and that 
the map $\cl E\to \Gamma[\cl E]$ is 
an affine map from $\cl C(V_1\times W_1)$ into $\cl C(V_2\times W_2)$.
We say that a channel $\cl F :  \cl D_{V_2} \to \cl D_{W_2}$ 
is \emph{simulated by $\cl E$ with the assistance of} $\Gamma$ \cite{clmw} if $\cl F = \Gamma[\cl E]$,
we call $\Gamma$ a \emph{simulator}, and we write 
$(V_1\mapsto W_1)\stackrel{\Gamma}{\to} (V_2\mapsto W_2)$.
The simulation procedure is illustrated by the diagram:

\begin{center}
\begin{tabular}{ccccccc} 
& 
& 
& 
& 
$V_1$
& 
$\stackrel{\cl E}{\longrightarrow}$
& 
$W_1$
\\ 
& 
& 
& 
& 
\hspace{-0.2cm}
$\UpArrow[0.8cm][>=stealth,black, thick, dashed]$
& 
&
\hspace{-0.3cm}
$\DownArrow[0.8cm][>=stealth,black, thick, dashed]$
\\ 
& 
& 
& 
& 
$V_2$
& 
$\stackrel{\Gamma[\cl E]}{\longrightarrow}$
& 
$W_2$.
\\ 
\end{tabular}
\end{center}

Suppose that $(V_1\mapsto W_1)\stackrel{\Gamma_1}{\to} (V_2\mapsto W_2)$ and 
$(V_2\mapsto W_2)\stackrel{\Gamma_2}{\to} (V_3\mapsto W_3)$, and define the linear map
$\Gamma_2\ast \Gamma_1 : \cl D_{V_3 W_1}\to \cl D_{V_1 W_3}$ by letting
$$(\Gamma_2 \ast \Gamma_1)(v_1,w_3 | v_3,w_1) 
= \sum_{v_2\in V_2} \sum_{w_2\in W_2}
\Gamma_1(v_1,w_2 | v_2,w_1) \Gamma_2(v_2,w_3 | v_3,w_2).$$

\begin{theorem}\label{th_comsimu}
If $(V_1\mapsto W_1)\stackrel{\Gamma_1}{\to} (V_2\mapsto W_2)$ and 
$(V_2\mapsto W_2)\stackrel{\Gamma_2}{\to} (V_3\mapsto W_3)$
then 
$$(V_1\mapsto W_1)\stackrel{\Gamma_2 \ast \Gamma_1}{\longrightarrow} (V_3\mapsto W_3).$$
Moreover, 
\begin{itemize}
\item[(i)] 
if ${\rm t}\in \{{\rm loc}, {\rm q}, {\rm qa}, {\rm qc}, {\rm ns}\}$
and $\Gamma_i\in \cl C_{\rm t}$, $i = 1,2$, then $\Gamma_2\ast \Gamma_1\in \cl C_{\rm t}$; 

\item[(ii)] 
if $\cl E \in \cl C(V_1\times W_1)$ then $(\Gamma_2 \ast \Gamma_1)[\cl E] = \Gamma_2 [\Gamma_1[\cl E]]$.
\end{itemize}
\end{theorem}

\begin{proof}
It is clear that the map $\Gamma_2 \ast \Gamma_1 : \cl D_{V_3W_1}\to \cl D_{V_1W_3}$ is positive. 
Fix $w_1\in W_1$. Then 
\begin{eqnarray*}
& & 
\sum_{w_3\in W_3} (\Gamma_2 \ast \Gamma_1)(v_1,w_3 | v_3,w_1)\\
& = & 
\sum_{w_3\in W_3} \sum_{v_2\in V_2} \sum_{w_2\in W_2}
\Gamma_1(v_1,w_2 | v_2,w_1) \Gamma_2(v_2,w_3 | v_3,w_2)\\
& = & 
\sum_{v_2\in V_2} \sum_{w_2\in W_2}
\Gamma_1(v_1,w_2 | v_2,w_1) \Gamma_2(v_2| v_3)
= 
\sum_{v_2\in V_2} 
\Gamma_1(v_1| v_2) \Gamma_2(v_2| v_3).
\end{eqnarray*}
Therefore, if $v_3\in V_3$ then 
\begin{eqnarray*}
\sum_{v_1\in V_1} \sum_{w_3\in W_3} (\Gamma_2 \ast \Gamma_1)(v_1,w_3 | v_3,w_1)
& = & 
\sum_{v_1\in V_1}
\sum_{v_2\in V_2}
\Gamma_1(v_1 | v_2) \Gamma_2(v_2| v_3)\\
& = & 
\sum_{v_2\in V_2} \Gamma_2(v_2| v_3) = 1.
\end{eqnarray*}
It follows that $\Gamma_2 \ast \Gamma_1$ is trace preserving and the marginal channel
$(\Gamma_2 \ast \Gamma_1)_{V_3\to V_1}$ is well-defined. By symmetry, so is 
the marginal channel $(\Gamma_2 \ast \Gamma_1)^{W_1\to W_3}$. 

\smallskip

(i) 
Assume that $\Gamma_i\in \cl C_{\rm qc}$, $i = 1,2$. 
Let $H$ (resp. $H'$) be a Hilbert space, 
$\xi \in H$ (resp. $\xi' \in H'$) a unit vector, 
and $(E_{v_{2}, v_{1}})_{v_{1} \in V_{1}}$ and $(F_{w_{1}, w_{2}})_{w_{2} \in W_{2}}$ 
(resp. $(E_{v_{3}, v_{2}}')_{v_{2} \in V_{2}}$ and $(F_{w_{2}, w_{3}}')_{w_{3} \in W_{3}}$) be 
families of POVM's such that $E_{v_{2}, v_{1}}F_{w_{1}, w_{2}} = F_{w_{1}, w_{2}}E_{v_{2}, v_{1}}$
(resp. $E_{v_{3}, v_{2}}'F_{w_{2}, w_{3}}' = F_{w_{2}, w_{3}}' E_{v_{3}, v_{2}}'$)
for all $v_i\in V_i$, $w_i\in W_i$, $i = 1,2,3$, and 
$$\Gamma_1\hspace{-0.02cm}(v_{1},\hspace{-0.03cm} w_{2}|v_{2},\hspace{-0.03cm} w_{1}\hspace{-0.02cm}) 
\hspace{-0.07cm}=\hspace{-0.07cm}\langle E_{v_{2}, v_{1}}F_{w_{1}, w_{2}}\xi, \xi\rangle, \ 
\Gamma_2\hspace{-0.02cm}(v_{2},\hspace{-0.03cm} w_{3}|v_{3},\hspace{-0.03cm} w_{2}\hspace{-0.02cm}) \hspace{-0.07cm}=\hspace{-0.07cm} 
\langle E_{v_{3}, v_{2}}'F_{w_{2}, w_{3}}'\xi', \xi'\rangle\hspace{-0.04cm},$$
for all $v_{i} \in V_{i}, w_{i} \in W_{i}$, $i = 1, 2, 3$. 
Let $H'' = H \otimes H', \xi'' = \xi\otimes \xi'$, and 
$$E_{v_{3}, v_{1}}'' = \sum\limits_{v_{2} \in V_{2}}E_{v_{2}, v_{1}}\otimes E_{v_{3}, v_{2}}' \ \mbox{ and } \ 
F_{w_{1}, w_{3}}'' = \sum\limits_{w_{2} \in W_{2}}F_{w_{1}, w_{2}}\otimes F_{w_{2}, w_{3}}'.$$
It is clear that 
$E_{v_{3}, v_{1}}''F_{w_{1}, w_{3}}'' = F_{w_{1}, w_{3}}''E_{v_{3}, v_{1}}''$ for all $v_i\in V_i$, $w_i\in W_i$, 
$i = 1,3$.
In addition, if $v_{3} \in V_{3}$ then 
$$	\sum\limits_{v_{1} \in V_{1}}E_{v_{3}, v_{1}}''
= 
	 \sum\limits_{v_{1} \in V_{1}}\sum\limits_{v_{2} \in V_{2}}E_{v_{2}, v_{1}}\otimes E_{v_{3}, v_{2}}'
= 
	\sum\limits_{v_{2} \in V_{2}}I \otimes E_{v_{3}, v_{2}}' = I.$$
Thus, $(E_{v_{3}, v_{1}}'')_{v_{1}}$ is a POVM for every $v_{3} \in V_{3}$. 
Similarly, $(F_{w_{1}, w_{3}}'')_{w_{3}}$ is a POVM for every $w_{1} \in W_{1}$. Finally, we see
\begin{eqnarray*}
& &
	\langle E_{v_{3}, v_{1}}''F_{w_{1}, w_{3}}''\xi'', \xi'' \rangle\\
& = &
	\sum\limits_{v_{2} \in V_{2}}\sum\limits_{w_{2} \in W_{2}}\langle (E_{v_{2}, v_{1}}\otimes E_{v_{3}, v_{2}}')(F_{w_{1}, w_{2}}\otimes F_{w_{2}, w_{3}}')(\xi\otimes \xi'), (\xi\otimes \xi')\rangle\\
& = &
	\sum\limits_{v_{2} \in V_{2}}\sum\limits_{w_{2} \in W_{2}}\langle E_{v_{2}, v_{1}}F_{w_{1}, w_{2}}\xi, \xi\rangle \langle E_{v_{3}, v_{2}}'F_{w_{2}, w_{3}}'\xi', \xi'\rangle \\
& = &
	\sum\limits_{v_{2} \in V_{2}}\sum\limits_{w_{2} \in W_{2}} \hspace{-0.15cm}
	\Gamma_1(v_{1},\hspace{-0.02cm} w_{2}|v_{2},\hspace{-0.02cm} w_{1}) 
	\Gamma_2(v_{2},\hspace{-0.02cm} w_{3}|v_{3},\hspace{-0.02cm} w_{2}) 
	\hspace{-0.08cm} =\hspace{-0.08cm} 
	(\Gamma_2 \hspace{-0.05cm}\ast \hspace{-0.05cm}\Gamma_1)(v_{1},\hspace{-0.02cm} w_{3}|v_{3},\hspace{-0.02cm} w_{1}).
\end{eqnarray*}
This implies that $\Gamma_2\hspace{-0.02cm} \ast\hspace{-0.02cm} \Gamma_1 \in \cl C_{\rm qc}$. 
The cases where $\rm t = \rm q$ and $\rm t = \rm loc$ are similar.
Finally, the case $\rm t = {\rm qa}$ follows from the case $\rm t = \rm q$ and the continuity of the operation 
$(\Gamma_1,\Gamma_2) \to \Gamma_2\hspace{-0.05cm} \ast\hspace{-0.05cm} \Gamma_1$. 

\smallskip

(ii) We have 
\begin{eqnarray*}
& & (\Gamma_2 \ast \Gamma_1)[\cl E](w_3 | v_3) \\
& = & 
\sum_{v_1\in V_1} \sum_{w_1\in W_1}(\Gamma_2 \ast \Gamma_1)(v_1,w_3 | v_3,w_1)\cl E(w_1|v_1)\\
& = & 
\sum_{v_1\in V_1} \sum_{w_1\in W_1}\sum_{v_2\in V_2} \sum_{w_2\in W_2}
\Gamma_1(v_1,w_2 | v_2,w_1) \Gamma_2(v_2,w_3 | v_3,w_2) \cl E(w_1|v_1)\\
& = & 
\sum_{v_2\in V_2} \sum_{w_2\in W_2} \Gamma_2(v_2,w_3 | v_3,w_2)\Gamma_1[\cl E](w_2 | v_2) 
= 
\Gamma_2 [\Gamma_1[\cl E]](w_3 | v_3).
\end{eqnarray*}
\end{proof}


\section{Quasi-homomorphism games}\label{s_hyper}

In this section, we study an auxiliary notion, 
which will be specialised to the main cases of interest in 
Section \ref{s_hypis}. 
Let $E_1\subseteq V_1\times W_1$ and $E_2\subseteq V_2\times W_2$.
We write 
$$E_1 \hspace{-0.1cm} \leadsto\hspace{-0.1cm} E_2 
\ = \ \left\{(v_2,w_1,v_1,w_2) : (v_1,w_1)\in E_1 \Rightarrow (v_2,w_2)\in E_2\right\};$$
thus, 
$$E_1 \hspace{-0.1cm} \leadsto\hspace{-0.1cm} E_2 = \left\{(v_2,w_1,v_1,w_2) : (v_1,w_1)\in E_1^c  \mbox{ or } 
(v_1,w_1,v_2,w_2)\in E_1\hspace{-0.1cm}\times \hspace{-0.07cm} E_2\right\}.$$
We consider $E_1\hspace{-0.1cm} \leadsto\hspace{-0.1cm} E_2$ as a hypergraph in 
$V_2 W_1\times V_1W_2$, and hence as a game with question and answer sets $V_2 \times W_1$ and 
$V_1 \times W_2$, respectively.

\begin{definition}\label{d_homom}
Let $E_i \subseteq V_i\times W_i$ be a hypergraph, $i = 1,2$, and 
${\rm t}\in \{{\rm loc}, {\rm q}, {\rm qa}, {\rm qc}, {\rm ns}\}$.
We say that 
$E_1$ is \emph{${\rm t}$-quasi-homomorphic} to $E_2$ (denoted $E_1 \leadsto_{\rm t} E_2$)
if $\cl C_{\rm t}(E_1 \hspace{-0.1cm} \leadsto \hspace{-0.1cm} E_2) \neq \emptyset$.
\end{definition}

If $\Gamma\in \cl C_{\rm t}(E_1 \hspace{-0.1cm} \leadsto \hspace{-0.1cm} E_2)$,
we say that $E_1\leadsto_{\rm t} E_2$ \emph{via} $\Gamma$.

\begin{proposition}\label{p_mapping}
Let $E_i \subseteq V_i\times W_i$ be a full hypergraph, $i = 1,2$, and $\Gamma$
be an NS correlation over $(V_2,W_1,V_1,W_2)$. 
Then $E_1 \leadsto_{\rm ns} E_2$ via $\Gamma$ if and only if 
$\cl E\to \Gamma[\cl E]$ restricts to a well-defined affine map from $\cl C(E_1)$ into $\cl C(E_2)$.
\end{proposition}

\begin{proof}
Assume that $E_1 \leadsto_{\rm ns} E_2$ via $\Gamma$ and let $\cl E\in \cl C(E_1)$. 
Suppose that $(v_2,w_2)\not\in E_2$. 
If $\cl E(w_1|v_1) \neq 0$ then $(v_1,w_1)\in E_1$ and hence, since $\Gamma$ fits 
$E_1 \hspace{-0.1cm} \leadsto \hspace{-0.1cm} E_2$, we have that $\Gamma(v_1,w_2|v_2,w_1) = 0$.
It follows that $\Gamma[\cl E](w_2|v_2) = 0$; thus, $E_{\Gamma[\cl E]}\subseteq E_2$.

Conversely, assume that $\cl E\to \Gamma[\cl E]$ restricts to a well-defined map from 
$\cl C(E_1)$ into $\cl C(E_2)$. Suppose that $(v_1,w_1)\in E_1$ and $(v_2,w_2)\not\in E_2$.
Let $\cl E : \cl D_{V_1}\to \cl D_{W_1}$ be any channel that fits $E_1$ such that 
$\cl E(w_1|v_1) = 1$ (note that the existence of such a channel is guaranteed by the fact that $E_1$ is full). 
We have that 
\begin{equation}\label{eq_sumnog}
\Gamma[\cl E](w_2|v_2) = \Gamma(v_1,w_2|v_2,w_1) + \sum_{(v_1',w_1')\neq (v_1,w_1)}
\Gamma(v_1',w_2|v_2,w_1')\cl E(w_1'|v_1').
\end{equation}
Since $(v_2,w_2)\not\in E_2$, we have that $\Gamma[\cl E](w_2|v_2) = 0$. 
Now (\ref{eq_sumnog}) implies that $\Gamma(v_1,w_2|v_2,w_1) = 0$; 
thus, $\Gamma$ fits $E_1 \hspace{-0.1cm} \leadsto \hspace{-0.1cm} E_2$, that is, 
$E_1\leadsto_{\rm ns} E_2$ via $\Gamma$.
\end{proof}

\begin{theorem}\label{th_ordeqre}
Let ${\rm t}\in \{{\rm loc},{\rm q}, {\rm qa}, {\rm qc}, {\rm ns}\}$.
The relation $\leadsto_{\rm t}$ is a quasi-order on the set of all hypergraphs.
\end{theorem}

\begin{proof}
Fix ${\rm t}\in \{{\rm loc},{\rm q}, {\rm qa}, {\rm qc}, {\rm ns}\}$.
Reflexivity follows from the fact that the identity channel is a local correlation. 
Suppose that $E_1$, $E_2$ and $E_3$ are hypergraphs such that 
$E_1\leadsto_{\rm t} E_2$ via $\Gamma_1$ and $E_2\leadsto_{\rm t} E_3$ via $\Gamma_2$. 
By Theorem \ref{th_comsimu}, $\Gamma_2 \ast \Gamma_1 \in \cl C_{\rm t}$. 
It suffices to show that $\Gamma_2 \ast \Gamma_1$ fits $E_1\to E_3$. 
Suppose that $(v_1,w_1)\in E_1$ and 
$(\Gamma_2 \ast \Gamma_1)(v_1, w_3|v_3,w_1)\neq 0$. 
Then there exists $(v_2,w_2)\in V_2\times W_2$ such that 
$$\Gamma_1(v_1,w_2 | v_2,w_1) \neq 0 \ \mbox{ and } \ \Gamma_2(v_2,w_3 | v_3,w_2) \neq 0.$$
Since $\Gamma_1$ fits $E_1 \hspace{-0.1cm} \leadsto \hspace{-0.1cm} E_2$, we have that $(v_2,w_2)\in E_2$, and since
$\Gamma_2$ fits $E_2 \hspace{-0.1cm} \leadsto \hspace{-0.1cm} E_3$, 
we conclude that $(v_3,w_3)\in E_3$. The proof is complete. 
\end{proof}

Let $E_1 \subseteq V_1 \times W_1$ and $E_2 \subseteq V_2\times W_2$ be hypergraphs. 
A map $f : V_2\to V_1$ will be called a
\emph{quasi-homomorphism} from $E_1$ into $E_2$ if 
$f^{-1}(\alpha)$ is contained in an edge of $E_2$ for every edge $\alpha$ of $E_1$. 
A quasi-homomorphism $f : V_2\to V_1$ gives rise to an accompanying map $g : W_1\to W_2$
such that 
\begin{equation}\label{eq_quasih}
f^{-1}(E_1(w_1))\subseteq E_2(g(w_1)) \ \ \mbox{ for every } w_1\in W_1;
\end{equation}
conversely, if the maps $f : V_2\to V_1$ and $g : W_1\to W_2$ satisfy 
(\ref{eq_quasih}), then $f$ is a quasi-homomorphism. 
If there exist a quasi-homomorphism from $E_1$ into $E_2$, we say that $E_1$ is
\emph{quasi-homomorphic} to $E_2$.

\begin{proposition}\label{p_lochy}
Let $E_1 \subseteq V_1\times W_1$ and $E_2 \subseteq V_2 \times W_2$ be hypergraphs. Then
$E_1\leadsto_{\rm loc} E_2$ if and only $E_1$ is quasi-homomorphic to $E_2$.
\end{proposition}

\begin{proof}
Assume that $\Gamma$ is a local correlation that fits $E_1 \hspace{-0.1cm} \leadsto\hspace{-0.1cm} E_2$. 
We may assume that $\Gamma$ is an extreme point in $\cl C_{\rm loc}$ and hence, by 
no-signalling, there exist functions $f : V_2\to V_1$ and $g : W_1\to W_2$ such that 
$$\{(v_2, w_1, f(v_2),g(w_1)) : w_1\in W_1, v_2 \in V_2\}
\hspace{0.1cm}\subseteq \hspace{0.1cm} E_1\hspace{-0.1cm} \leadsto\hspace{-0.1cm} E_2.$$
This implies
\begin{equation}\label{eq_fvw}
f(v_2)\in E_{1}(w_1) \ \Longrightarrow \  v_2\in E_{2}(g(w_1)) \;\;\;\; \text{for all } w_1 \in W_{1},
\end{equation}
which means that $(f,g)$ determines a quasi-homomorphism from $E_1$ to $E_2$. 

Conversely, suppose that (\ref{eq_fvw}) is satisfied.
Let $\Phi : \cl D_{V_2}\to \cl D_{V_1}$ 
(resp. $\Psi : \cl D_{W_1}\to \cl D_{W_2}$) be the channel given by 
$\Phi(v_1|v_2) = \delta_{v_1,f(v_2)}$ (resp. $\Psi(w_2|w_1) = \delta_{w_2,g(w_1)}$) and 
$\Gamma = \Phi\otimes \Psi$; then $\Gamma$ 
fits $E_1 \hspace{-0.1cm} \leadsto\hspace{-0.1cm} E_2$. 
Indeed, assume that $(v_1,w_1)\in E_1$ and $(v_2,w_2)\not\in E_2$, but 
$$\delta_{v_1,f(v_2)}\delta_{w_2,g(w_1)} = \Gamma(v_1,w_2|v_2,w_1) \neq 0.$$ 
This means that 
$f(v_2) = v_1$ and hence $f(v_2)\in E_{1}(w_1)$, implying $v_2\in E_{2}(g(w_1))$.
Since $g(w_1) = w_2$, we have $(v_2,w_2)\in E_2$, a contradiction. 
\end{proof}

It is clear that, if ${\rm t}$ and ${\rm t'}$ are correlation types such that 
$\cl C_{\rm t} \subseteq \cl C_{\rm t'}$, then 
$$E_1 \leadsto_{\rm t} E_2 \ \ \Longrightarrow \ \ E_1 \leadsto_{\rm t'} E_2.$$
We next show the irreversibility of the latter implication for some of the hypergraph quasi-homomorphism types. 
If $G$ is a simple graph with vertex set $X$, we write $x\sim_G x'$ if $\{x,x'\}$ is an edge of $G$, and 
$x\simeq_G x'$ if $x\sim x'$ or $x = x'$ (if $G$ is understood from the context, we 
write $x\sim x'$ and $x\simeq x'$, respectively). 
We let $\alpha(G)$ be the \emph{independence number} of $G$, defined as the maximum cardinality $|S|$ of an
independent set $S$ of vertices (that is, a subset $S\subseteq X$ such that $x,x'\in S \ \Rightarrow \ x\not\sim x'$). 
We fix a hypergraph $E\subseteq X\times Y$, and write 
$G_E$ for the corresponding \emph{confusability graph}: the vertex set of $G_E$ is $X$ and the adjacency is given by 
letting 
$$x\sim x' \ \mbox{ if } x\neq x' \mbox{ and } \exists \ y\in Y \mbox{ s.t. } (x,y)\in E \mbox{ and } (x',y)\in E.$$
If $X\stackrel{\cl E}{\mapsto} Y$, the confusability graph $G_{\cl E}$ of $\cl E$  \cite{shannon}
is defined by letting $G_{\cl E} = G_{E_{\cl E}}$. 
For a given set $Z$, let 
$$\Delta_Z = \{(z,z) : z\in Z\}$$ 
be the diagonal over $Z$, considered as a hypergraph in $Z\times Z$.

\begin{lemma}\label{l_localforin}
Let $E\subseteq X\times Y$ be a full hypergraph. 
We have that $E \leadsto_{\rm loc} \Delta_Z$ if and only if $\alpha(G_E)\geq |Z|$. 
\end{lemma}

\begin{proof}
Suppose that $E \leadsto_{\rm loc} \Delta_Z$ and, using Proposition \ref{p_lochy}, let 
$f : Z\to X$ and $g : Y \to Z$ be maps realising a quasi-homomorphism from $E$ into $\Delta_Z$. 
This means
\begin{equation}\label{eq_fzygy}
(f(z),y)\in E \ \ \Longrightarrow \ \ z = g(y).
\end{equation}
Let $z,z' \in Z$ with $f(z)\simeq f(z')$ in $G_E$. Since $E$ is full, 
there exists $y\in Y$ such that 
$(f(z),y)\in E$ and $(f(z'),y)\in E$. By (\ref{eq_fzygy}), $g(y) = z = z'$. 
Thus, 
\begin{equation}\label{eq_zz'}
z\neq z' \ \Rightarrow \ f(z)\not\simeq_{G_E} f(z');
\end{equation} 
it follows that $f(Z)$ is an independent set in $G_E$ and hence, since $f$ is injective by (\ref{eq_zz'}), 
$\alpha(G_E)\geq |Z|$. 

Conversely, let $f : Z\to X$ be an injective map such that $f(Z)$ is an independent set 
in $G_E$. Fix $z_0\in Z$. 
Let $y\in Y$; then $E(y)$ is a clique in $G_E$; thus,
$|f(Z)\cap E(y)|\leq 1$. If $f(Z)\cap E(y) = \{f(z)\}$ for some $z\in Z$, define $g(y) = z$; 
if $f(Z)\cap E(y) = \emptyset$, define $g(y) = z_0$. 
It is straightforward that the pair $(f,g)$ of maps realises a quasi-homomorphism from $E$ into $\Delta_Z$. 
\end{proof}

\begin{proposition}\label{p_sep}
The implications 
$$E\leadsto_{\rm loc} F \ \Rightarrow \ E\leadsto_{\rm q} F \ \ \mbox{ and } \ \ 
E\leadsto_{\rm q} F \ \Rightarrow \ E\leadsto_{\rm ns} F$$
are not reversible. 
\end{proposition}

\begin{proof}
Recall \cite{clmw, dsw} that the quantum independence number $\alpha_{\rm q}(G)$ of a graph $G$ with vertex set $X$ is 
defined by letting
$$\alpha_{\rm q}(G) \hspace{-0.05cm} = \hspace{-0.05cm} \max\{k\in \bb{N} : 
\exists \ X \hspace{-0.05cm} \stackrel{\cl E}{\mapsto} \hspace{-0.05cm} Y \mbox{ with } 
G_{\cl E} \hspace{-0.05cm} = \hspace{-0.05cm} G \mbox{ and } 
\Gamma \hspace{-0.05cm} \in \hspace{-0.05cm} \cl C_{\rm q} \mbox{ s.t. } 
\Gamma[\cl E] \hspace{-0.05cm} = \hspace{-0.05cm} \id\mbox{}_{\cl D_k}\}.$$
By \cite[Theorem 13]{clmw}, there exists a graph $G$ such that $\alpha(G) < \alpha_{\rm q}(G)$. 
Let $Y$ be a (finite) set and $\cl E : \cl D_X\to \cl D_Y$ be a channel that achieves the maximum in the 
definition of $\alpha_{\rm q}(G)$. Let $E = E_{\cl E}$;  since $E$ is the hypergraph of a channel, $E$ is full. 
Letting $Z$ be a set with $|Z| = \alpha_{\rm q}(G)$, we have that $E\leadsto_{\rm q} \Delta_Z$.
On the other hand, by Lemma \ref{l_localforin}, $E\not\leadsto_{\rm loc} \Delta_Z$. 

To show that the second implication fails in general, use
\cite[Theorem 7]{clmw}, according to which, if $E\subseteq X\times Y$ is a hypergraph then 
$$\max\{{Z} : E\leadsto_{\rm ns} \Delta_Z \} = \lfloor \alpha^*(E)\rfloor,$$
where $\alpha^*(E)$ is the fractional packing number of $E$ 
(see \cite[Definition 5]{clmw}).
By the first paragraph, it suffices to exhibit an example of a hypergraph $E$ such that 
$\lfloor \alpha^*(E)\rfloor > \alpha_{\rm q}(G_E)$. 
Let $\vartheta(G)$ (resp. $\bar\chi_{\rm f}(G)$) be the the Lov\'asz number \cite{lo} 
of a graph $G$ (resp. the fractional chromatic number of the complement of $G$).
By \cite{dsw} and the discussion surrounding
\cite[Proposition 6]{clmw}, we have the inequalities 
$$\alpha_{\rm q}(G_E) \leq \vartheta(G_E)\leq \bar \chi_{\rm f}(G_E) \leq \alpha^*(E).$$

It hence suffices to exhibit an example of a graph $G$ with 
$[\vartheta(G)] <  \lfloor\bar\chi_{\rm f}(G)\rfloor$.
Let $n,r\in \bb{N}$ with $r \leq n$ and let $K(n,r)$ be the graph whose vertices are the 
subsets of $[n]$ of cardinality $r$, with two such subsets $S$ and $T$ being adjacent if
$S\cap T = \emptyset$. 
By (the proof of) \cite[Theorem 13]{lo}, $\vartheta(K(n,r)) = {n-1\choose r-1}$ 
while, as stated after \cite[Corollary 7]{lo}, $\bar\chi_{\rm f}(K(n,r)) = {n \choose r}/\lfloor\frac{n}{r}\rfloor$.  
Thus, an example is furnished by letting, e.g., $n = 5$ and $r = 3$. 
\end{proof}

\noindent {\bf Remark. } 
We do not have counterexamples that show the irreversibility of the implications 
$E_1 \leadsto_{\rm q} E_2 \ \Rightarrow \ E_1 \leadsto_{\rm qa} E_2$
or $E_1 \leadsto_{\rm qa} E_2 \ \Rightarrow \ E_1 \leadsto_{\rm qc} E_2$. 
Such counterexamples would provide an alternative way to observe the inequalities
$\cl C_{\rm q}\neq \cl C_{\rm qa}$ and $\cl C_{\rm qa}\neq \cl C_{\rm qc}$, respectively, and 
would thus be of substantial interest. 

\medskip

In the rest of the section, we link hypergraph quasi-homomorphisms to tensor products of operator systems.
Recall that an \emph{operator system} is a selfadjoint subspace of $\cl B(H)$, 
for some Hilbert space $H$, containing $I_H$. 
If $\cl S$ is an 
operator system, we write $M_n(\cl S)^+$ for the cone of positive elements in the space $M_n(\cl S)$ of all 
$n$ by $n$ matrices with entries in $\cl S$. If $\cl S$ and $\cl T$ are operator systems and $\phi : \cl S\to \cl T$
is a linear map, we let $\phi^{(n)} : M_n(\cl S)\to M_n(\cl T)$ be the map, given by 
$\phi^{(n)}((x_{i,j})_{i,j}) = (\phi(x_{i,j}))_{i,j}$. The map $\phi$ is called \emph{positive} if $\phi(\cl S^+)\subseteq \cl T^+$, \emph{completely positive} if $\phi^{(n)}$ is positive for every $n\in \bb{N}$, and \emph{unital} if $\phi(1) = 1$. 
We call $\cl S$ and $\cl T$ \emph{completely order isomorphic}, and write $\cl S\cong_{\rm c.o.i.}\cl T$, if 
there exists a unital completely positive bijection $\phi : \cl S\to \cl T$ with completely positive inverse.
We write $\cl S\subseteq_{\rm c.o.i.}\cl T$ if $\cl S\subseteq \cl T$ and the inclusion map $\cl S\to \cl T$ is 
a complete order isomorphism onto its range. 
We note that if $\cl S$ is a finite dimensional operator system,
the Banach space dual $\cl S^{\rm d}$ can be viewed, via 
Choi-Effros Theorem \cite[Theorem 13.1]{Pa}, as an operator system \cite[Corollary 4.5]{CE2}.

We refer to \cite{Pa} for further details about operator systems, and recall here 
the three types operator system tensor products \cite{kptt} of operator systems $\cl S$ and $\cl T$ 
that will be used in the sequel:

\begin{itemize}
\item[(i)] 
the \emph{minimal} operator system tensor product $\cl S\otimes_{\min}\cl T$
arises from viewing $\cl S\otimes\cl T$ as a subspace of 
$\cl B(H\otimes K)$, where $\cl S$
and $\cl T$ are realised as operator systems in $\cl B(H)$ and $\cl B(K)$, 
respectively (and $H$ and $K$ are Hilbert spaces); 

\item[(ii)] 
the \emph{commuting} tensor product $\cl S\otimes_{\rm c}\cl T$ has the smallest family of matricial cones that makes
the maps $\phi\cdot\psi$, where $\phi : \cl S\to \cl B(H)$ and $\psi : \cl T\to \cl B(H)$ and completely positive maps 
with commuting ranges, completely positive; 
here, $\phi\cdot\psi$ is the linear map, given by 
$(\phi\cdot\psi)(x\otimes y) = \phi(x)\psi(y)$, $x\in \cl S$, $y\in \cl T$;

\item[(iii)] 
the \emph{maximal} tensor product $\cl S\otimes_{\max}\cl T$ has matricial cones generated by the elementary 
tensors of the form $S\otimes T$, where $S\in M_n(\cl S)^+$ and $T\in M_m(\cl T)^+$, $n,m\in \bb{N}$. 
\end{itemize}

For finite sets $X$ and $A$, 
let $\cl A_{X,A} = \underbrace{\cl D_A\ast_1\cdots\ast_1\cl D_A}_{|X| \ {\rm times}}$, 
a C*-algebra free product, amalgamated over the units. 
Let $(\tilde{e}_{x,a})_{a\in A}$ be the standard basis of 
the $x$-th copy of $\cl D_A$ in $\cl A_{X,A}$, and 
$$\cl S_{X,A} = {\rm span}\{\tilde{e}_{x,a} : x\in X, a\in A\},$$
viewed as an operator subsystem of $\cl A_{X,A}$. 
As is readily seen, the operator system $\cl S_{X,A}$ satisfies the following universal property: 
for every family $\{(E_{x,a})_{a\in A} : x\in X\}$ of POVM's, acting on a Hilbert space $H$, 
there exists a unital completely positive map $\phi : \cl S_{X,A}\to \cl B(H)$ such that 
$\phi(\tilde{e}_{x,a}) = E_{x,a}$, $x\in X$, $a\in A$; conversely, if 
$\phi : \cl S_{X,A}\to \cl B(H)$ is a unital completely positive map for some Hilbert space $H$,
then $(\phi(\tilde{e}_{x,a}))_{a\in A}$ is a POVM, $x\in X$. 

Set 
$$\cl R_{X,A} = \left\{(\lambda_{x,a})_{x\in X, a\in A} : \lambda_{x,a}\in \bb{C}, \ 
\sum_{a\in A} \lambda_{x,a} = \sum_{a\in A} \lambda_{x',a}, \ x,x'\in X\right\},$$
viewed as an operator subsystem of $\cl D_{XA}$. 
By \cite[Theorem 5.9]{fkpt_NYJ}, 
\begin{equation}\label{eq_SdR}
\cl S_{X,A}^{\rm d}\cong_{\rm c.o.i.} \cl R_{X,A}.
\end{equation}

\begin{remark}\label{r_aff}
\rm 
By \cite[Theorem 3.1]{lmprsstw}, $\Gamma$ is a no-signalling correlation over the quadruple
$(V_2,W_1,V_1,W_2)$ if and only if there exists a state 
$s : \cl S_{V_2,V_1}\otimes_{\max}\cl S_{W_1,W_2}\to \bb{C}$ such that  
$$\Gamma(v_1,w_2|v_2,w_1) = s\left(\tilde{e}_{v_2,v_1}\otimes \tilde{e}_{w_1,w_2}\right), \ \ \ 
v_i\in V_i, w_i\in W_i, \ i = 1,2.$$
By (\ref{eq_SdR}) and \cite[Proposition 1.9]{fp}, 
$$\left(\cl S_{V_2,V_1}\otimes_{\max}\cl S_{W_1,W_2}\right)^{\rm d} \cong_{\rm c.o.i.} 
\cl R_{V_2,V_1}\otimes_{\min}\cl R_{W_1,W_2};$$
thus, the simulators 
$(V_1\mapsto W_1)\stackrel{\Gamma}{\to} (V_2\mapsto W_2)$
correspond canonically to the elements of the subset
$$\left\{\Lambda \in \left(\cl R_{V_2,V_1}\otimes_{\min}\cl R_{W_1,W_2}\right)^+ : 
({\rm Tr}_{V_1}\otimes {\rm Tr}_{W_2})(\Lambda) = 1\right\}.$$
Hence we have the following are equivalent for hypergraphs
$E_1 \subseteq V_1\times W_1$ and $E_2 \subseteq V_2 \times W_2$:
\begin{itemize}
\item[(i)] 
the relation $E_1\leadsto_{\rm ns} E_2$ holds true; 

\item[(ii)] 
there exists a matrix 
$\Lambda\in (\cl R_{V_2,V_1}\otimes_{\min} \cl R_{W_1,W_2})^+$
supported on the set $E_1 \hspace{-0.1cm} \leadsto \hspace{-0.1cm} E_2$.
\end{itemize}
\end{remark}

\begin{remark}\label{r_dGamma}
\rm
In Proposition \ref{p_mapping}, we saw that a simulator $\Gamma$ that fits 
$E_1 \hspace{-0.1cm} \leadsto \hspace{-0.1cm} E_2$ induces an affine map from $\cl C(E_1)$ to $\cl C(E_2)$. 
We point out that not all such affine maps arise via simulation. 
We identify the set $\cl C(V\times W)$ of all information channels $V\mapsto W$ with 
the subset 
\begin{equation}\label{eq_TrW1}
\{\Lambda \in \cl R_{V,W} : {\rm Tr}_{W}(\Lambda) = 1\}
\end{equation}
of the operator system $\cl R_{V,W}$.
Let $\Phi : \cl C(V_1\times W_1) \to \cl C(V_2\times W_2)$ be an affine map, and extend it linearly to 
a map (denoted in the same way) $\Phi : \cl R_{V_1,W_1}\to \cl R_{V_2,W_2}$. 
By \cite[Theorem 3.9]{Pa}, $\Phi$ is completely positive. 
By \cite[Lemma 4.6]{CE2}, $\Phi$ corresponds in a canonical fashion to an element 
$\varphi \in \cl R_{V_1,W_1}\otimes_{\min} \cl R_{V_2,W_2}^{\rm d}$;
by (\ref{eq_SdR}), 
we can view $\nph$ as an element of $\cl R_{V_1,W_1}\otimes_{\min} \cl S_{V_2,W_2}$. 
Reversing these steps, we see that every 
element $\nph\in \cl R_{V_1,W_1}\otimes_{\min} \cl S_{V_2,W_2}$
gives rise in a canonical fashion to an affine map $\Phi : \cl C(V_1\times W_1) \to \cl C(V_2\times W_2)$.

On the other hand, suppose that the map $\Phi : \cl C(V_1\times W_1) \to \cl C(V_2\times W_2)$
has the form $\Phi(\cdot) = \Gamma[\cdot]$ for some simulator $\Gamma$. 
According to Remark \ref{r_aff}, $\Gamma$ can be canonically identified with an element
$\gamma$ of $\cl R_{V_2,V_1}\otimes_{\min}\cl R_{W_1,W_2}$. 
Matrix multiplication 
$$\frak{m} : \cl R_{V_2,V_1}\otimes_{\min}\cl R_{V_1,W_1} \otimes_{\min} \cl R_{W_1,W_2}\to \cl R_{V_2,W_2}$$
is completely positive and can hence be viewed, via \cite[Lemma 4.6]{CE2}, as a positive element
$$\tilde{\frak{m}} \in 
\cl R_{V_2,V_1}\otimes_{\min}\cl R_{V_1,W_1} \otimes_{\min} \cl R_{W_1,W_2}
\otimes_{\min} \cl R_{V_2,W_2}^{\rm d}$$
which, taking into account (\ref{eq_SdR}), induces a completely positive map
$$\hat{\frak{m}} : \cl R_{V_2,V_1}\otimes_{\min}\cl R_{W_1,W_2}\to 
\cl S_{V_1,W_1}\otimes_{\max} \cl R_{V_2,W_2}.$$
The simulator $\Gamma$ can thus be identified with the element $\hat{\frak{m}}(\gamma)$. 
The difference between all affine maps $\cl C(V_1\times W_1) \to \cl C(V_2\times W_2)$ and 
the simulators can now be visualised as the difference between the operator systems
$\cl R_{V_1,W_1}\otimes_{\min} \cl S_{V_2,W_2}$ and
$\cl S_{V_1,W_1}\otimes_{\max} \cl R_{V_2,W_2}$. 

To be more specific, fix $w_1^0\in W_1$ and let $\cl E\in \cl C(V_1\times W_1)$ be the channel, given by 
$$\cl E(w_1|v_1) = \delta_{w_1,w_1^0}, \ \ \ v_1\in V_1, w_1\in W_1.$$
If $(V_1\mapsto W_1)\stackrel{\Gamma}{\to} (V_2\mapsto W_2)$ then 
$$\Gamma[\cl E](w_2|v_2) = \sum_{v_1\in V_1} \Gamma(v_1,w_2|v_2,w_1^{0}) = \Gamma(w_2|w_1^{0}).$$
We see that the probability distribution $\Gamma[E](\cdot|v_2)$ is independent of the variable
$v_2$ and of the choice of $w_1^{0}$, a property not enjoyed by arbitrary affine maps 
from $\cl C(V_1\times W_1)$ to $\cl C(V_2\times W_2)$. 
\end{remark}


\section{Homomorphism games}\label{s_hypis}

In this section, we adapt the set-up from Section \ref{s_hyper} to define quantum versions of 
hypergraph homomorphisms and hypergraph isomorphisms.


\subsection{Bicorrelations}\label{ss_bicore}

We recall the definitions of a bicorrelation and of the various 
bicorrelation types introduced in \cite{bhtt2}, which will be needed in the sequel. 
Suppose that $V_1 = V_2 =: V$ and $W_1 = W_2 =: W$.
We call an element $\Gamma\in \cl C_{\rm ns}$ a \emph{no-signalling (NS) bicorrelation} \cite{bhtt2} 
if $\Gamma$ is a unital channel and its dual $\Gamma^*$ is a no-signalling correlation. 

Recall \cite{ddn} that, if $H$ is a Hilbert space, a \emph{quantum magic square} over $V$ on $H$ is a 
block operator matrix $(E_{v_1,v_2})_{v_1,v_2\in V}$ whose entries are positive operators, and 
$$\sum_{v_2'\in V} E_{v_1,v_2'} = \sum_{v_1'\in V} E_{v_1',v_2} = I, \ \ \ v_1,v_2\in V.$$
An NS bicorrelation $\Gamma$ is called \emph{quantum commuting} if 
there exists a Hilbert space $H$, a unit vector $\xi\in H$ and 
quantum magic squares $(E_{v_2,v_1})_{v_2,v_1\in V}$ and $(F_{w_1,w_2})_{w_1,w_2\in W}$ with 
commuting entries, such that 
\begin{equation}\label{eq_bicore}
\Gamma(v_1,w_2|v_2,w_1) = \langle E_{v_2,v_1} F_{w_1,w_2}\xi,\xi\rangle, \ \ 
v_i\in V, w_i\in W, i = 1,2.
\end{equation}
The bicorrelation $\Gamma$ is called \emph{quantum} if
the expression (\ref{eq_bicore}) is achieved for $H = H_A\otimes H_B$, where $H_A$ and $H_B$ are 
finite dimensional Hilbert spaces, $E_{v_2,v_1} = E_{v_2,v_1}'\otimes I_{H_B}$ and 
$F_{w_1,w_2} = I_{H_A}\otimes F_{w_1,w_2}'$, $v_i\in V$, $w_i\in W$, $i = 1,2$. 
The \emph{quantum approximate} bicorrelations are the limits of quantum bicorrelations. 
Finally, the \emph{local bicorrelations} are 
the convex combinations of correlations of the form 
$p^{(1)}(v_1|v_2) p^{(2)}(w_2|w_1)$, where $(p^{(1)}(v_1|v_2))_{v_1,v_2}$ and 
$(p^{(2)}(w_2|w_1))_{w_2,w_1}$ are (scalar) bistochastic matrices. 

We use the notation
$\cl C_{\rm t}^{\rm bi}$ for the (convex) set of all bicorrelations of type ${\rm t}$.
For a subset $\Lambda\subseteq V_2W_1\times V_1W_2$, we set
$\cl C_{\rm t}^{\rm bi}(\Lambda) = \cl C_{\rm t}^{\rm bi}\cap \cl C_{\rm ns}(\Lambda)$. 

\begin{remark}\label{r_adjo}
\rm 
Let ${\rm t}$ be a correlation type. 
If $\Gamma\in \cl C_{\rm t}^{\rm bi}$ then $\Gamma^*\in \cl C_{\rm t}^{\rm bi}$. 
For ${\rm t}\neq {\rm ns}$, this is a consequence of the fact that the transpose of a quantum magic square 
is again a quantum magic square, while for ${\rm t} = {\rm ns}$ this is part of the definition. 
\end{remark}


\subsection{Game definitions and properties}\label{ss_defpro2}

We fix finite sets $V_i$ and $W_i$, and let $E_i\subseteq V_i\times W_i$ be a hypergraph, $i = 1,2$. 
Set 
$$E_1 \hspace{-0.1cm} \leftrightarrow\hspace{-0.1cm} E_2 
\ = \ \left\{(v_2,w_1,v_1,w_2) : (v_1,w_1)\in E_1 \Leftrightarrow (v_2,w_2)\in E_2\right\};$$
thus, 
$E_1 \hspace{-0.1cm} \leftrightarrow\hspace{-0.1cm} E_2$ consist of the quadruples 
$(v_2,w_1,v_1,w_2)$ for which 
$$(v_1,w_1,v_2,w_2)\in E_1\times E_2 \ \mbox{ or } \ 
(v_1,w_1,v_2,w_2)\in E_1^c \times E_2^c.$$
We consider
$E_1 \hspace{-0.1cm} \leftrightarrow\hspace{-0.1cm} E_2$ as a non-local game 
with question and answer sets $V_2W_1$ and $V_1W_2$, respectively.

\begin{definition}\label{d_isomorhy}
Let $E_i \subseteq V_i\times W_i$ be a hypergraph, $i = 1,2$, and 
${\rm t}\in \{{\rm loc}, {\rm q}, {\rm qa}, {\rm qc}, {\rm ns}\}$.
We say that
\begin{itemize}
\item[(i)] 
$E_1$ is \emph{${\rm t}$-homomorphic} to $E_2$ (denoted $E_1 \hspace{-0.05cm} \to_{\rm t} \hspace{-0.05cm} E_2$) if 
$\cl C_{\rm t}(E_1 \hspace{-0.1cm} \leftrightarrow\hspace{-0.1cm} E_2)\neq \emptyset$; 

\item[(ii)]
$E_1$ is \emph{${\rm t}$-isomorphic} to $E_2$ (denoted $E_1\simeq_{\rm t} E_2$) if $V_1 = V_2$, $W_1 = W_2$ and 
$\cl C_{\rm t}^{\rm bi}(E_1 \hspace{-0.1cm} \leftrightarrow\hspace{-0.1cm} E_2)\neq \emptyset$.
\end{itemize}
\end{definition}

An element $\Gamma$ of $\cl C_{\rm t}(E_1 \hspace{-0.1cm} \leftrightarrow\hspace{-0.1cm} E_2)$
(resp. $\cl C_{\rm t}^{\rm bi}(E_1 \hspace{-0.1cm} \leftrightarrow\hspace{-0.1cm} E_2)$) will be referred to as 
a ${\rm t}$-homomorphism (resp. ${\rm t}$-isomorphism) from $E_1$ to $E_2$. 
It is clear that 
\begin{equation}\label{eq_impli}
E_1 \simeq_{\rm t} E_2 \ \ \Longrightarrow \ \ E_1 \to_{\rm t} E_2  \ \ \Longrightarrow \ \ E_1 \leadsto_{\rm t} E_2.
\end{equation}

\begin{remark}\label{r_Gamma*}
\rm 
Let $\Gamma\in \cl C_{\rm ns}^{\rm bi}$. The following are equivalent:
\begin{itemize}
\item[(i)] 
$E_1 \simeq_{\rm ns} E_2$ via $\Gamma$;

\item[(ii)] the map $\cl E\to \Gamma[\cl E]$ is well-defined from $\cl C(E_1)$ into $\cl C(E_2)$, and 
the map $\cl F\to \Gamma^*[\cl F]$ is well-defined from $\cl C(E_2)$ into $\cl C(E_1)$.
\end{itemize}
Indeed, taking into account (\ref{eq_impli}), condition (i) implies that 
$E_1 \leadsto_{\rm t} E_2$ via $\Gamma$ and $E_2 \leadsto_{\rm t} E_1$ via $\Gamma^*$, and (ii) follows from
Proposition \ref{p_mapping}.
Conversely, assuming (ii), Proposition \ref{p_mapping} implies that 
$\Gamma$ fits $E_1 \to E_2$ while $\Gamma^*$ fits $E_2 \to E_1$. This means that $\Gamma$ fits 
$E_1 \hspace{-0.1cm} \leftrightarrow\hspace{-0.1cm} E_2$. 
\end{remark}

\begin{theorem}\label{th_ordeqre2}
For ${\rm t}\in \{{\rm loc},{\rm q}, {\rm qa}, {\rm qc}, {\rm ns}\}$, the relation 
$\to_{\rm t}$ (resp. $\simeq_{\rm t}$) is a quasi-order (resp. an equivalence relation). 
\end{theorem}

\begin{proof}
Similarly to the proof of Theorem \ref{th_ordeqre}, 
one can verify that if $\Gamma_1$ and $\Gamma_2$ are correlations such that 
$\Gamma_1$ fits $E_1 \hspace{-0.1cm} \leftrightarrow\hspace{-0.1cm} E_2$, while
$\Gamma_2$ fits $E_2 \hspace{-0.1cm} \leftrightarrow\hspace{-0.1cm} E_3$, 
then the correlation $\Gamma_2 \ast \Gamma_1$ fits $E_1 \hspace{-0.1cm} \leftrightarrow\hspace{-0.1cm} E_3$. 
The claim about $\to_{\rm t}$ now follows from Theorem \ref{th_comsimu} (i).

To see the claim about the relation $\simeq_{\rm t}$, it suffices to establish its transitivity. 
It is therefore enough to show that whenever $\Gamma_1, \Gamma_2\in \cl C_{\rm t}^{\rm bi}$, we also have 
$\Gamma_2 \ast \Gamma_1\in \cl C_{\rm t}^{\rm bi}$. 
In the case ${\rm t} = {\rm ns}$, the claim is a consequence of Remark \ref{r_adjo}, 
Theorem \ref{th_comsimu} and the fact that 
$(\Gamma_2 \ast \Gamma_1)^* = \Gamma_1^* \ast \Gamma_2^*,$
which we verify:
\begin{eqnarray*}
(\Gamma_2 \ast \Gamma_1)^*(v_3,w_1 | v_1,w_3)
& = & 
(\Gamma_2 \ast \Gamma_1)(v_1,w_3 | v_3,w_1)\\
& = & 
\sum_{v_2\in V_2} \sum_{w_2\in W_2}
\Gamma_1(v_1,w_2 | v_2,w_1) \Gamma_2(v_2,w_3 | v_3,w_2)\\
& = & 
\sum_{v_2\in V_2} \sum_{w_2\in W_2}
\Gamma_2^*(v_3,w_2 | v_2,w_3) \Gamma_1^*(v_2,w_1 | v_1,w_2)\\
& = & 
(\Gamma_1^* \ast \Gamma_2^*)(v_3,w_1 | v_1,w_3).
\end{eqnarray*}
In the case ${\rm t} = {\rm qc}$, the claim follows from the proof of Theorem \ref{th_comsimu} and the 
fact that the transpose $(E_{v_2,v_1})_{v_1,v_2}$ of a quantum magic square 
$(E_{v_1,v_2})_{v_1,v_2}$ is also a quantum magic square. The claim in the case of the 
remaining types, 
${\rm t} = {\rm qa}, {\rm q}, {\rm loc}$, follow similarly. 
\end{proof}

Let $E_i\subseteq V_i\times W_i$ be a hypergraph, $i = 1,2$. 
A map $f : V_2\to V_1$ is called a \emph{homomorphism} from $E_1$ to $E_2$ if 
$f^{-1}(\alpha)$ is an edge of $E_2$ whenever $\alpha$ is an edge of $E_1$; 
equivalently, $f : V_2\to V_1$ is a homomorphism precisely when there exists a map 
$g : W_1\to W_2$ such that 
\begin{equation}\label{eq_quasih2}
f^{-1}(E_1(w_1)) = E_2(g(w_1)) \ \ \mbox{ for every } w_1\in W_1.
\end{equation}
If $V_1 = V_2$ and $W_1 = W_2$, an \emph{isomorphism} from $E_1$ to $E_2$ is a bijective homomorphism 
$f$, for which the map $g$ in (\ref{eq_quasih2}) can be chosen to be a bijection.

\begin{proposition}\label{p_lochy2}
Let $E_1 \subseteq V_1\times W_1$ and $E_2 \subseteq V_2 \times W_2$ be hypergraphs. Then
\begin{itemize}
\item[(i)]
$E_1 \to_{\rm loc} E_2$
if and only if there exists a homomorphism from $E_1$ to $E_2$;

\item[(ii)]
if $V_1 = V_2$ and $W_1 = W_2$, then
$E_1 \simeq_{\rm loc} E_2$ if and only if the hypergraphs $E_1$ and $E_2$ are isomorphic.
\end{itemize}
\end{proposition}

\begin{proof}
(i) 
As in the proof of Proposition \ref{p_lochy}, 
the existence of a perfect local strategy for the homomorphism game 
$E_1 \hspace{-0.1cm} \to\hspace{-0.1cm} E_2$ implies 
the existence of maps
$f : V_2\to V_1$ and $g : W_1\to W_2$ such that 
$$(f(v_2),w_1)\in E_1 \ \Longleftrightarrow \ (v_2,g(w_1))\in E_2,$$
which is equivalent to (\ref{eq_quasih2}). 

Conversely, assuming (\ref{eq_quasih2}) and adopting the notation from the proof of Proposition \ref{p_lochy}, 
we have that 
$$(\Phi\otimes\Psi)(v_1,w_2|v_2,w_1) = 
\begin{cases}
1 & \text{if } v_1 = f(v_2) \mbox{ and } w_2 = g(w_1)\\
0 & \text{otherwise.}
\end{cases}
$$
Thus, assuming that $(\Phi\otimes\Psi)(v_1,w_2|v_2,w_1) = 1$, we have that 
$$(v_1,w_1)\in E_1 \ \Leftrightarrow \ (f(v_2),w_1)\in E_1
\ \Leftrightarrow \ (v_2,g(w_1))\in E_2 
\ \Leftrightarrow \ (v_2,w_2)\in E_2,$$
which shows that $\Phi\otimes\Psi$ fits $E_1 \hspace{-0.1cm} \leftrightarrow\hspace{-0.1cm} E_2$. 

(ii)
Assume that the bicorrelation $\Gamma$ is a perfect local strategy for the 
hypergraph isomorphism game 
$E_{1}\hspace{-0.1cm} \leftrightarrow\hspace{-0.1cm} E_{2}$. 
By definition, 
$\Gamma = \sum_{i=1}^k \lambda_i \Phi_i\otimes \Psi_i$ as a convex combination,
where $\Phi_i : \cl D_V\to \cl D_V$ and $\Psi_i : \cl D_W\to \cl D_W$ arise from (scalar) bistochastic matrices. 
Using Birkhhoff's Theorem, we decompose these bistochastic matrices as convex combinations of permutation 
matrices; this allows us to assume that $\Phi_i$ and $\Psi_i$ each arise from permutation matrices.
By positivity, $\Phi_i\otimes \Psi_i$ is a perfect strategy for the game
$E_{1}\hspace{-0.1cm} \leftrightarrow\hspace{-0.1cm} E_{2}$. 
We may thus assume that $\Gamma = \Phi \otimes \Psi$, where 
$\Phi(v_1|v_2) = \delta_{f(v_2),v_1}$ and $\Psi(w_2|w_1) = \delta_{g(w_1),w_2}$ for some 
bijections $f : V_2\to V_1$ and $g : W_1\to W_2$. 
Identifying $\Phi$ and $\Psi$ with the corresponding conditional probability distributions, we have that 
$${\rm supp}(\Phi\otimes\Psi) = \left\{(v_2,w_1,f(v_2),g(w_1)) : v_2,w_1\in V\right\}.$$
It follows that $E_1$ and $E_2$ are isomorphic via the pair $(f,g)$. 

Conversely, assuming that $f$ and $g$ are bijections that fulfill (\ref{eq_quasih2}),
the channel $\Gamma = \Phi\otimes\Psi$, defined in the previous 
paragraph, is a 
bicorrelation that is a perfect strategy for 
the game $E_1 \hspace{-0.1cm} \leftrightarrow\hspace{-0.1cm} E_2$.
\end{proof}


\subsection{Values of probabilistic hypergraphs}\label{ss_values}

A \emph{probabilistic hypergraph} is a hypergraph $E\subseteq V\times W$, equipped with a probability 
distribution $\pi : V\to [0,1]$ on its vertex set. 
Given a convex subset $\frak{E}\subseteq \cl C(V\times W)$ of channels from $V$ to $W$, we let 
\begin{equation}\label{eq_frakEpi}
\omega_{\frak{E}}(E,\pi) = \sup_{\cl E\in \frak{E}} \sum_{(v,w)\in E} \pi(v) \cl E(w|v)
\end{equation}
be the $\frak{E}$-\emph{value} of $(E,\pi)$. 
Suppose that $\Gamma$ is a perfect no-signalling 
strategy for the homomorphism game $E_1 \hspace{-0.1cm} \to\hspace{-0.1cm} E_2$.
Given a probability distribution $\pi_2$ on $V_2$, let $\pi_1 = \Gamma_{V_2\to V_1}(\pi_2)$; 
thus, $\pi_1$ is the probability distribution on $V_1$ given by 
$$\pi_1(v_1) = \sum_{v_2\in V_2} \Gamma_{V_2\to V_1}(v_1|v_2)\pi_2(v_2), \ \ \ v_1\in V_1.$$
Similarly, for a probability distribution $\pi_1$ on $V_1$, let $\pi_2 = \Gamma^*_{V_1\to V_2}(\pi_1)$. 
If $V_1 = V_2 =: V$, the probability distribution $\pi$ on $V$ will be called 
\emph{$\Gamma$-stationary} if $V_1 = V_2$ and $\pi = \pi_1 = \pi_2$.

\begin{proposition}\label{p_inval}
Let $E_i\subseteq V_i\times W_i$ be a hypergraph,
$\frak{E}_i\subseteq \cl C(V_i\times W_i)$, $i = 1,2$, $\pi_2$ be a probability distribution on $V_2$, 
and $\pi_1 = \Gamma_{V_2\to V_1}(\pi_2)$. 
\begin{itemize}
\item[(i)] 
Suppose that 
$E_1\to_{\rm ns} E_2$ via a correlation $\Gamma$ such that 
$\Gamma[\frak{E}_1]\subseteq \frak{E}_2$.
Then 
$$\omega_{\frak{E}_1}(E_1,\pi_1)\leq \omega_{\frak{E}_2}(E_2,\pi_2).$$

\item[(ii)] 
Suppose that $V_1 = V_2 = :V$, that $\pi$ is a $\Gamma$-stationary probability distribution on $V$ and 
that $E_1\simeq_{\rm ns} E_2$ via a bicorrelation $\Gamma$ 
such that $\Gamma[\frak{E}_1]\subseteq \frak{E}_2$ and $\Gamma^*[\frak{E}_2]\subseteq \frak{E}_1$. 
Then 
$\omega_{\frak{E}_1}(E_1,\pi) = \omega_{\frak{E}_2}(E_2,\pi)$.
\end{itemize}
\end{proposition}

\begin{proof}
(i)
Let $\Gamma$ be a perfect no-signalling strategy for the homomorphism game 
$E_1 \hspace{-0.1cm} \to\hspace{-0.1cm} E_2$. 
Let $\cl E\in \frak{E}_1$ and $\cl F = \Gamma[\cl E]$. 
We have that 
\begin{eqnarray*}
& & 
\sum_{(v_2,w_2)\in E_2} \cl F(w_2|v_2) \pi_2(v_2)\\
& = & 
\sum_{(v_2,w_2)\in E_2} \sum_{v_1\in V_1}\sum_{w_1\in W_1} \Gamma(v_1,w_2|v_2,w_1) \cl E(w_1|v_1)\pi_2(v_2)\\
& \geq & 
\sum_{(v_2,w_2)\in E_2} \sum_{(v_1,w_1)\in E_1}\Gamma(v_1,w_2|v_2,w_1) \cl E(w_1|v_1)\pi_2(v_2)\\
& = & 
\sum_{(v_1,w_1)\in E_1} \sum_{v_2\in V_2}\sum_{w_2\in W_2}\Gamma(v_1,w_2|v_2,w_1) \cl E(w_1|v_1)\pi_2(v_2)\\
& = & 
\sum_{(v_1,w_1)\in E_1} \sum_{v_2\in V_2} \Gamma(v_1|v_2) \cl E(w_1|v_1)\pi_2(v_2)
= 
\sum_{(v_1,w_1)\in E_1} \cl E(w_1|v_1) \pi_1(v_1).
\end{eqnarray*}
Taking the supremum over all $\cl E\in \frak{E}_1$ yields the desired inequality.

(ii) follows by symmetry from (i). 
\end{proof}

If $E\subseteq V\times W$ and $F\subseteq X\times Y$ are hypergraphs, their 
\emph{product} is the hypergraph 
$E\otimes F \subseteq (VX)\times (WY)$, given by 
$$E\otimes F = \{(v,x,w,y) : (v,w)\in E, (x,y)\in F\}.$$
We write $E^{\otimes n} = \otimes_{i=1}^n E$. Given a probability distribution $\pi$ on $V$, let 
$\pi^n = \otimes_{i=1}^n \pi$ be the $n$-fold product distribution of $\pi$ on the vertex set $V^n$ of $E^{\otimes n}$. 
We fix subsets $\frak{E}_n\subseteq \cl C(V^n\times W^n)$ with the property that 
\begin{equation}\label{eq_incnm}
\cl E\in \frak{E}_n, \ \cl F\in \frak{E}_m \ \Longrightarrow \ \cl E\otimes \cl F \in \frak{E}_{n+m}, \ \ \ n,m\in \bb{N},
\end{equation}
and write $\bar{\frak{E}} = (\frak{E}_n)_{n\in \bb{N}}$. 
The \emph{asymptotic $\bar{\frak{E}}$-value} of $(E,\pi)$ is the quantity 
\begin{equation}\label{eq_bar}
\bar{\omega}_{\bar{\frak{E}}}(E,\pi) = \limsup_{n\in \bb{N}} \omega_{\frak{E}_n}(E^{\otimes n},\pi^n)^{\frac{1}{n}}.
\end{equation}

\begin{remark}\label{r_lim}
\rm 
Let $\bar{\frak{E}} = (\frak{E}_n)_{n\in \bb{N}}$ be a sequence of families of channels, where 
$\frak{E}_n\subseteq \cl C(V^n\times W^n)$, $n\in \bb{N}$, satisfying (\ref{eq_incnm}). 
By the definition of the value (\ref{eq_frakEpi}), we then have 
$$\omega_{\frak{E}_n}(E^{\otimes n},\pi^n)\omega_{\frak{E}_m}(E^{\otimes m},\pi^m)
\leq \omega_{\frak{E}_{n+m}}(E^{\otimes (n+m)},\pi^{n+m})$$
for all $n,m\in \bb{N}$. Thus, by Fekete's Lemma (see \cite{fekete}), 
the limsup in the definition (\ref{eq_bar}) can be replaced by a limit. 
\end{remark}

\begin{corollary}\label{p_invalcor}
Let $E_i\subseteq V_i\times W_i$ be a hypergraph and 
$\bar{\frak{E}}_i  = (\frak{E}_n^{(i)})_{n\in \bb{N}}$ be sequences of families of channels, 
where $\frak{E}_n^{(i)}\subseteq \cl C(V_i^n\times W_i^n)$, satisfying (\ref{eq_incnm}),
$i = 1,2$.

\begin{itemize}
\item[(i)] 
Suppose that 
$E_1\to_{\rm ns} E_2$ via a correlation $\Gamma$ such that 
$\Gamma^{\otimes n}[\frak{E}_n^{(1)}]\subseteq \frak{E}_n^{(2)}$, $n\in \bb{N}$, and let
$\pi_2$ be a probability distribution on $V_2$. Set $\pi_1 = \Gamma_{V_2\to V_1}(\pi_2)$.
Then

$$\bar{\omega}_{\bar{\frak{E}}_1}(E_1,\pi_1)\leq \bar{\omega}_{\bar{\frak{E}}_2}(E_2,\pi_2).$$

\item[(ii)] 
Suppose that $V_1 = V_2 = :V$, that $\pi$ is a $\Gamma$-stationary probability distribution on $V$ and 
that $E_1\simeq_{\rm ns} E_2$ via a bicorrelation $\Gamma$ 
such that $\Gamma^{\otimes n}[\frak{E}_n^{(1)}]\subseteq \frak{E}_n^{(2)}$ and 
$\Gamma^{* \otimes n}[\frak{E}_n^{(2)}]\subseteq \frak{E}_n^{(1)}$. 
Then 
$\bar{\omega}_{\bar{\frak{E}}_1}(E_1,\pi) = \bar{\omega}_{\bar{\frak{E}}_2}(E_2,\pi)$.
\end{itemize}
\end{corollary}

\begin{proof}
Proposition \ref{p_inval} (i) implies that 
$$\bar{\omega}_{\bar{\frak{E}}_1}(E_1^{\otimes n},\pi_1^n)
\leq \bar{\omega}_{\bar{\frak{E}}_2}(E_2^{\otimes n},\pi_2^n), \ \ \ n\in \bb{N},$$
implying part (i). Part (ii) follows by symmetry, applying Proposition \ref{p_inval} (ii). 
\end{proof}


\subsection{An operator system approach}\label{ss_opsysa}

We recall the universal operator system for bicorrelations, introduced in \cite{bhtt2}.
A \emph{ternary ring of operators (TRO)}  
is a subspace $\cl V\subseteq \cl B(H,K)$, for 
some Hilbert spaces $H$ and $K$, such that $ST^*R\in \cl V$ whenever $S,T,R\in \cl V$ 
(see e.g. \cite{blecher, kt}). 
Let $\cl V_V$ be the universal TRO
generated by the entries $u_{v,v'}$ of a \emph{bi-isometry}, that is, a block operator matrix 
$U = (u_{v,v'})_{v,v'\in V}$ such that both $U$ and its transpose $U^{\rm t} := (u_{v',v})_{v,v'\in V}$ are isometries.
Thus, $\cl V_V$ is the universal TRO with generators $u_{v,v'}$, $v,v'\in V$, and 
relations 
$$\sum_{a\in V} [u_{a'',x''},\hspace{-0.05cm}u_{a,x},\hspace{-0.05cm} u_{a,x'}] \hspace{-0.07cm}=\hspace{-0.07cm} \delta_{x,x'}u_{a'',x''} 
\mbox{ and }
\sum_{x\in V} [u_{a'',x''},\hspace{-0.07cm}u_{a,x},\hspace{-0.07cm}u_{a',x}] 
\hspace{-0.05cm}=\hspace{-0.05cm} \delta_{a,a'}u_{a'',x''},$$
for all $x, x',x'',a, a', a''\in V$.
Let $\cl C_V$ be the right C*-algebra of $\cl V_V$, when the latter is viewed as an 
imprimitivity bimodule \cite{rieffel}; thus, up to a *-isomorphism, we have that 
$\cl C_V\simeq \overline{{\rm span}}(\theta(\cl V_V)^*\theta(\cl V_V))$, for any faithful ternary representation
$\theta : \cl V_V\to \cl B(H,K)$ ($H$ and $K$ being Hilbert spaces).
We write 
$$e_{v_1,v_1',v_2,v_2'} := u_{v_2,v_1}^*u_{v_2',v_1'}, \ \ \ v_1,v_2,v_1',v_2'\in V;$$
note that $\cl C_V$ is generated, as a C*-algebra, 
by the elements $e_{v_1,v_1',v_2,v_2'}$, $v_1$, $v_1'$, $v_2$, $v_2'\in V$. 
Set $e_{v_1,v_2} := e_{v_1,v_1,v_2,v_2}$, $v_1,v_2\in V$, and let 
$$\cl S_V \hspace{-0.05cm}=\hspace{-0.05cm} {\rm span}\{e_{v_1,v_2} : v_1,v_2\in V\} \mbox{ and } 
\cl T_V \hspace{-0.05cm}=\hspace{-0.05cm} {\rm span}\{e_{v_1,v_1',v_2,v_2'} : v_1,v_1',v_2, v_2'\in V\},$$
viewed as operator subsystems of $\cl C_V$. 

The following was shown in \cite{bhtt2}:

\begin{theorem}[\cite{bhtt2}]\label{th_magicsq}
Let $H$ be a Hilbert space. 
If $\phi : \cl S_V\to \cl B(H)$ is a unital completely positive map
then $(\phi(e_{v_1,v_2}))_{v_1,v_2\in V}$ is a quantum magic square. 
Conversely, if $(E_{v_1,v_2})_{v_1,v_2\in V}$ is a quantum magic square on $H$
then there exists a (unique) unital completely positive map 
$\phi : \cl S_V\to \cl B(H)$ such that $E_{v_1,v_2} = \phi(e_{v_1,v_2})$, $v_1,v_2\in V$.
\end{theorem}

\begin{lemma}\label{l_SVflip}
The flip map $\frak{f} : e_{v_1,v_2}\to e_{v_2,v_1}$ extends to a unital complete order automorphism 
of $\cl S_V$.
\end{lemma}

\begin{proof}
Let $\phi : \cl S_V\to \cl B(H)$ be a unital complete order embedding and set 
$E_{v_1,v_2} = \phi(e_{v_1,v_2})$, $v_1,v_2\in V$.
By Theorem \ref{th_magicsq}, $(E_{v_1,v_2})_{v_1,v_2\in V}$ is a quantum magic square. 
Therefore, $(E_{v_2,v_1})_{v_1,v_2\in V}$ is a quantum magic square and
Theorem \ref{th_magicsq} gives rise to a unital completely positive map 
$\psi : \cl S_V\to \cl B(H)$ with the property 
$$\psi(e_{v_1,v_2}) = E_{v_2,v_1}, \ \ \ v_1,v_2\in V.$$
Note that $\psi = \phi\circ\frak{f}$; hence $\frak{f}$ is completely positive. 
By symmetry, $\frak{f}$ is a complete order isomorphism.
\end{proof}

We now assume that $V_1 = V_2 =: V$ and $W_1 = W_2 =: W$.
Let 
$$E_1 \hspace{-0.1cm} \Leftrightarrow\hspace{-0.1cm} E_2 \ := \ (E_1\times E_2) \cup (E_1^c\times E_2^c);$$
consider $E_1 \hspace{-0.1cm} \Leftrightarrow\hspace{-0.1cm} E_2$
as a non-local game with question and answer sets $V_1W_1$ and $V_2W_2$, respectively, 
and refer to it as an \emph{equivalence game}.
Note that, if
$$\frak{r} : V_2 \times W_1 \times V_1 \times W_2 \to V_1 \times W_1 \times V_2 \times W_2$$
is the shuffle map, given by $\frak{r}(v_2,w_1,v_1,w_2) = (v_1,w_1,v_2,w_2)$, then
$E_1 \hspace{-0.1cm} \Leftrightarrow\hspace{-0.1cm} E_2 
= \frak{r}(E_1 \hspace{-0.1cm} \leftrightarrow\hspace{-0.1cm} E_2).$

For clarity, we denote the canonical generators of the operator system $\cl S_W$ by $f_{w_1,w_2}$, 
$w_1,w_2\in W$.
Given a linear functional $s : \cl S_V\otimes \cl S_W\to \bb{C}$, we let 
$\Gamma_s : \cl D_{VW}\to \cl D_{VW}$ be the linear map given by 
$$\Gamma_s(v_1,w_2 |v_2,w_1) = s(e_{v_1,v_2}\otimes f_{w_1,w_2}).$$

For an NS correlation $\Gamma$ over $(V_2,W_1,V_1,W_2)$, let 
$\frak{F}(\Gamma) : \cl D_{V_1W_1}\to \cl D_{V_2W_2}$ be the linear map, defined by letting
$$\frak{F}(\Gamma)(v_2,w_2|v_1,w_1) := \Gamma(v_1,w_2|v_2,w_1).$$

\begin{proposition}\label{p_flip}
Let ${\rm t}\in \{{\rm loc}, {\rm q}, {\rm qa}, {\rm qc}, {\rm ns}\}$. 
The map $\Gamma\to \frak{F}(\Gamma)$ is an affine isomorphism between 
$\cl C_{\rm t}^{\rm bi}(E_1 \hspace{-0.1cm} \leftrightarrow\hspace{-0.1cm} E_2)$ and 
$\cl C_{\rm t}^{\rm bi}(E_1 \hspace{-0.1cm} \Leftrightarrow\hspace{-0.1cm} E_2)$.
\end{proposition}

\begin{proof}
Let ${\rm t} = {\rm qc}$. 
and $\Gamma\in \cl C_{\rm t}^{\rm bi}(E_1 \hspace{-0.1cm} \leftrightarrow\hspace{-0.1cm} E_2)$.
By \cite{bhtt2}, there exists a state $s : \cl S_V\otimes_{\rm c}\cl S_W\to \bb{C}$ such that 
$\Gamma = \Gamma_s$. 
Let $\tilde{s} := s\circ (\frak{f}\otimes \id)$; by Lemma \ref{l_SVflip} and 
the functoriality of the commuting tensor product, $\tilde{s}$ is a state on the operator system
$\cl S_V\otimes_{\rm c}\cl S_W$.
Since $\frak{F}(\Gamma) = \Gamma_{\tilde{s}}$, we have that $\frak{F}(\Gamma)\in \cl C_{\rm qc}^{\rm bi}$. 
The fact that 
$\supp(\frak{F}(\Gamma)) = \frak{r}(\supp(\Gamma))$ is straightforward. 
We finally note that $\frak{F}^2 = \id$, showing that $\frak{F}$ is an isomorphism. 

The cases ${\rm t} = {\rm qa}$ and ${\rm t} = {\rm ns}$ are analogous, using the minimal (resp. maximal)
tensor product instead of the commuting one. 

For the case ${\rm t} = {\rm q}$, assuming that $\Gamma \in \cl C_{\rm t}^{\rm bi}(E_{1} \hspace{-0.1cm} \leftrightarrow\hspace{-0.1cm} E_{2})$, 
there exist quantum magic squares $(E_{v_{1}, v_{2}})_{v_{1}, v_{2} \in V}$ (resp. $(F_{w_{1}, w_{2}})_{w_{1}, w_{2} \in W}$) acting on Hilbert spaces $H_{V}$ (resp. $H_{W}$) and unit vectors $\xi \in H_{V}, \eta \in H_{W}$ so that
$$\Gamma(v_{1}, w_{2}|v_{2}, w_{1}) = \langle (E_{v_{2}, v_{1}}\otimes F_{w_{1}, w_{2}})(\xi\otimes \eta), \xi\otimes \eta\rangle.$$
Let $\tilde{E} = (\tilde{E}_{v_{2}, v_{1}})_{v_{1}, v_{2} \in V}$, where 
$\tilde{E}_{v_{2}, v_{1}} := E_{v_{1}, v_{2}}$; 
it is clear that $\tilde{E}$ is a quantum magic square and that  
$$ \frak{F}(\Gamma)(v_{2}, w_{2}|v_{1}, w_{1}) = \langle (\tilde{E}_{v_{1}, v_{2}}\otimes F_{w_{1}, w_{2}})(\xi \otimes \eta), \xi\otimes \eta\rangle.$$
Thus, $\frak{F}$ induces an isomorphism from $\cl C_{\rm q}^{\rm bi}(E_{1} \hspace{-0.1cm}\leftrightarrow \hspace{-0.1cm}E_{2})$ onto $\cl C_{\rm q}^{\rm bi}(E_{1} \hspace{-0.1cm}\Leftrightarrow \hspace{-0.1cm}E_{2})$.

Assume that $\Gamma\in \cl C_{\rm loc}^{\rm bi}(E_1 \hspace{-0.1cm} \leftrightarrow\hspace{-0.1cm} E_2)$; 
thus, $\Gamma = \sum_{i=1}^k \lambda_i p_i^{(1)}\otimes p_i^{(2)}$, where 
$p_i^{(1)} = \{(p_i^{(1)}(v_1|v_2))_{v_1} : v_2\in V\}$  
(resp. $p_i^{(2)} = \{(p_i^{(2)}(w_2|w_1))_{w_2} : w_1\in W\}$) is a 
channel in $\cl C(V_2\times V_1)$ (resp. $\cl C(W_1\times W_2)$), with the property that 
the matrices $(p_i^{(1)}(v_1|v_2))_{v_2,v_1}$ and $(p_i^{(2)}(w_2|w_1))_{w_1,w_2}$ are bistochastic. 
It follows that the matrix $\tilde{p}_i^{(1)} := (p_i^{(1)}(v_2|v_1))_{v_1,v_2}$, $i = 1,\dots,k$, is bistochastic.
In addition, 
$$\frak{F}(\Gamma) = \sum_{i=1}^k \lambda_i \tilde{p}_i^{(1)}\otimes p_i^{(2)}.$$
As in the first paragraph, $\frak{F}$ induces an isomorphism from 
$\cl C_{\rm loc}^{\rm bi}(E_1 \hspace{-0.1cm} \leftrightarrow\hspace{-0.1cm} E_2)$ onto 
$\cl C_{\rm loc}^{\rm bi}(E_1 \hspace{-0.1cm} \Leftrightarrow\hspace{-0.1cm} E_2)$. 
\end{proof}

Let $E_1\subseteq V_1\times W_1$ and $E_2\subseteq V_2\times W_2$ be hypergraphs, and let 
$$\cl J = {\rm span} \{e_{v_2,v_1}\otimes f_{w_1,w_2} : 
(v_2,w_1,v_1,w_2) \not\in E_1 \hspace{-0.1cm} \leftrightarrow\hspace{-0.1cm} E_2\}.$$

\begin{corollary}\label{c_kernel}
The map $s\to \Gamma_s$ is an affine surjective correspondence between 
\begin{itemize}
\item[(i)] 
the states of $\cl S_V\otimes_{\max}\cl S_W$ that annihilate $\cl J$ and the 
perfect ${\rm ns}$-strategies of $E_1 \hspace{-0.1cm} \leftrightarrow\hspace{-0.1cm} E_2$; 

\item[(ii)]
the states of $\cl S_V\otimes_{\rm c}\cl S_W$ that annihilate $\cl J$ and the 
perfect ${\rm qc}$-strategies of $E_1 \hspace{-0.1cm} \leftrightarrow\hspace{-0.1cm} E_2$;

\item[(iii)]
the states of $\cl S_V\otimes_{\min}\cl S_W$ that annihilate $\cl J$ and the 
perfect ${\rm qa}$-strategies of $E_1 \hspace{-0.1cm} \leftrightarrow\hspace{-0.1cm} E_2$.
\end{itemize}
\end{corollary}


\subsection{Faithful isomorphisms}\label{ss_faithful}

In this subsection, we assume that $V_1 = W_1 = V_2 = W_2 =: V$.

\begin{definition}
Let $E_1\subseteq V_1\times W_1$ and $E_2\subseteq V_2\times W_2$. 
A bicorrelation $\Gamma\in \cl C_{\rm t}^{\rm bi}$ over $(V_2,W_1,V_1,W_2)$ 
is called \emph{faithful} if 
$$\Gamma(v_1,w_2 | v_2,w_1)\hspace{-0.07cm} =\hspace{-0.05cm} 0 
\mbox{ whenever } (v_1\hspace{-0.05cm} =\hspace{-0.05cm} w_1 \ \& \ 
v_2\hspace{-0.05cm} \neq\hspace{-0.05cm} w_2) 
\mbox{ or } (v_1\hspace{-0.05cm} \neq\hspace{-0.05cm} w_1 \ \& \ v_2\hspace{-0.05cm} =\hspace{-0.05cm} w_2).$$
A \emph{faithful ${\rm t}$-isomorphism} between $E_1$ and $E_2$
is a faithful bicorrelation $\Gamma\in \cl C_{\rm t}^{\rm bi}(E_1 \hspace{-0.1cm} \leftrightarrow\hspace{-0.1cm} E_2)$.
\end{definition}

A faithful isomorphism $\Gamma$ between the hypergraphs $E_1$ and $E_2$
can be thought of as a means of mutually simulating the noiseless channels $\id : V_2\to W_2$
and $\id : V_1\to W_1$ by each other: every time the original channel $\cl E : V_1\to W_1$
transmits faithfully a certain symbol $v\in V_1$, the simulated channel $\Gamma[\cl E]$ does so too, 
and vice versa. 

We note that a correlation $\Gamma$ over $(V_2,W_1,V_1,W_2)$ is faithful if and only if the 
correlation $\frak{F}(\Gamma)$ over $(V_1,W_1,V_2,W_2)$ is bisynchronous in the sense of \cite[Definition 1.2]{pr}. 
This enables us to use the works \cite{lmr} and \cite{pr} in the sequel. 
Recall that a \emph{quantum permutation} acting on Hilbert space $H$ is 
a unitary matrix $(P_{v,v'})_{v,v'\in V}$, whose entries $P_{v,v'}$ are projections in $\cl B(H)$.
(We note that every quantum permutation is automatically a quantum magic square.)
The quantum permutation group over $V$ is 
the universal C*-algebra $\frak{A}_V$ generated by the entries of a quantum permutation \cite[Section 2.2]{lmr}.
We write $p_{v,w}$, $v,w\in V$, for a fixed family of generators of $\frak{A}_V$ (so that 
$(p_{v,w})_{v,w\in V}$ is a universal quantum permutation). 
We call a quantum permutation $(P_{v,v'})_{v,v'\in V}$ a \emph{quantum q-permutation} (resp. a 
\emph{quantum qc-permutation}) if its entries act on a finite dimensional Hilbert space
(resp. there exists a C*-algebra $\cl A$ with a trace containing its entries).

Given a hypergraph $E\subseteq V\times V$, let 
$$\mathsf{A}_E = \sum_{(v,v')\in E} \epsilon_{v,v'}$$ 
be the incidence matrix of $E$.

\begin{theorem}\label{th_intertwine}
Let $t \in \{{\rm loc}, {\rm q}, {\rm qc}\}$. The following are equivalent:
\begin{itemize}
\item[(i)] $E_1$ is faithfully ${\rm t}$-isomorphic to $E_2$;

\item[(ii)] there exists a quantum ${\rm t}$-permutation $P = (P_{v,v'})_{v,v'\in V}$ such that 
\begin{equation}\label{eq_commre}
P(\mathsf{A}_{E_1}\otimes I_H) = (\mathsf{A}_{E_2}\otimes I_H)P.
\end{equation}
\end{itemize}
\end{theorem}


\begin{proof}
The proof relies on the ideas from the proof of \cite[Lemma 5.8]{amrssv}. 
We only consider the case ${\rm t} = {\rm qc}$.

(ii)$\Rightarrow$(i)
Suppose that $\cl A$ is a unital C*-algebra, equipped with a trace $\tau$ (which can be assumed to be faithful), 
and that $(P_{v, v'})_{v, v' \in V}$ is a quantum permutation with entries in $\cl A$, satisfying (\ref{eq_commre}). 
If $v, w' \in V$ then, denoting by $(\mathsf{A})_{v,w'}$ the $(v,w')$-entry of a matrix $\mathsf{A}$ over $V\times V$, 
we have 
\begin{equation}\label{eq_commre2}
\sum\limits_{\substack{v' \in E_{2}^*(v)}} \hspace{-0.2cm} P_{v',w'} 
= ((\mathsf{A}_{E_2}\otimes I_{\cl{H}})P)_{v, w'}
= (P(\mathsf{A}_{E_{1}}\otimes I_{\cl{H}}))_{v, w'}
=  \hspace{-0.2cm} \sum\limits_{\substack{w \in E_{1}(w')}}  \hspace{-0.2cm} P_{v,w}.
\end{equation}
Since the columns of $P$ are PVM's, 
\begin{gather*}
	\bigg(\sum\limits_{\substack{v' \in E_{2}^*(v)}}P_{v',w'}\bigg)^{2} 
	= \sum\limits_{\substack{v' \in E_{2}^*(v)}}P_{v',w'}.
\end{gather*}
Pairing this with (\ref{eq_commre2}), we see
\begin{eqnarray*}
\sum\limits_{\substack{v' \in E_{2}^*(v)}}P_{v',w'}\sum\limits_{\substack{w \in E_{1}(w')}}P_{v,w}
& = & 
	 \bigg(\sum\limits_{\substack{v' \in E_{2}^*(v)}}P_{v',w'}\bigg)^{2} \\
& = &
	\sum\limits_{\substack{v' \in E_{2}^*(v)}}P_{v',w'}
= 
	\sum\limits_{\substack{v' \in E_{2}^*(v)}}P_{v',w'}\sum\limits_{w \in V}P_{v,w}.
\end{eqnarray*}
\noindent This implies
$$\sum\limits_{\substack{v' \in E_{2}^*(v)}}P_{v',w'}\sum\limits_{\substack{w \not\in E_{1}(w')}}P_{v,w} = 0,$$
hence
$$\sum\limits_{\substack{v' \in E_{2}^*(v)}}\sum\limits_{\substack{w \not\in E_{1}(w')}}\tau(P_{v',w'}P_{v,w}) = 0,$$
forcing $\tau(P_{v',w'}P_{v,w}) = 0$ whenever $(v, v') \in E_{2}$ while $(w, w') \not \in E_{1}$. 
By symmetry, 
$\tau(P_{v',w'}P_{v,w}) = 0$ whenever $(w, w') \in E_{1}$ while $(v, v') \not \in E_{2}$. As $\tau$ is faithful, this implies
\begin{equation}\label{eq_relP}
P_{v',w'}P_{v,w} = P_{v,w}P_{v',w'} = 0 \ \ \mbox{ whenever } 
(v, v') \in E_{2} \not\Leftrightarrow (w, w') \not \in E_{1}.
\end{equation}
Define the linear map $\Gamma: \cl{D}_{V_{2}W_{1}}\rightarrow \cl{D}_{V_{1}W_{2}}$ by letting
\begin{gather*}
	\Gamma(v_{1},w_{2}|v_{2},w_{1}) := \tau(P_{v_{2},v_{1}}P_{w_{1},w_{2}}).
\end{gather*}
\noindent We claim $\Gamma \in \cl{C}_{\rm qc}^{\rm bi}(E_1 \hspace{-0.1cm} \leftrightarrow\hspace{-0.1cm} E_2)$, 
and that it is faithful. 
It is clear that $\Gamma$ is a quantum commuting correlation. The unitality of $\Gamma$
is straightforward, while the fact that $\Gamma^*$ is quantum commuting follows from Remark \ref{r_adjo}. 
The faithfulness of $\Gamma$ is an immediate consequence of the fact that the rows and columns of a quantum 
permutation are PVM's.
Finally, since $\frak{F}(\Gamma)$ is a perfect strategy for the equivalence game 
$E_1 \hspace{-0.1cm} \Leftrightarrow\hspace{-0.1cm} E_2$, Proposition \ref{p_flip} implies that 
$\Gamma$ is a perfect (quantum commuting) strategy for the hypergraph isomorphism game 
$E_1 \hspace{-0.1cm} \leftrightarrow\hspace{-0.1cm} E_2$. 


(i)$\Rightarrow$(ii) 
Assume that $E_{1}$ is faithfully ${\rm qc}$-isomorphic to $E_{2}$ via 
$\Gamma\hspace{-0.06cm} \in\hspace{-0.03cm} \cl{C}_{\rm qc}^{{\rm bi}}(E_{1}\hspace{-0.1cm} \leftrightarrow \hspace{-0.1cm} E_{2})$. 
By Proposition \ref{p_flip}, $\frak{F}(\Gamma)$ is a perfect quantum commuting and bisynchronous 
strategy of the equivalence game 
$E_{1}\hspace{-0.1cm} \Leftrightarrow \hspace{-0.1cm} E_{2}$. 
By \cite[Theorem 2.2]{pr} and \cite[Lemma 5.13]{amrssv}, there exists a faithful 
tracial state $\tau$ on a C*-algebra $\cl A$, and a *-representation $\pi : \frak{A}_V\to \cl A$
such that, if $P_{v,w} = \pi(p_{v,w})$ for $v, w \in V$, then 
$$\Gamma(v_{1}, w_{2}|v_{2}, w_{1}) = \tau(P_{v_{1},v_{2}} P_{w_{1},w_{2}}).$$
Since $\Gamma$ fits $E_{1}\hspace{-0.1cm} \leftrightarrow \hspace{-0.1cm} E_{2}$, we have that 
if $(v_{1}, w_{1}) \in E_{1}$ and $(v_{2}, w_{2}) \not \in E_{2}$, then 
$\tau(P_{v_{1},v_{2}}P_{w_{1},w_{2}}) = 0$,
whenever 
$(v_{1}, w_{1}) \in E_{1} \not\Leftrightarrow (v_{2}, w_{2}) \in E_{2}$. 
It follows that 
$$P_{v_{1},v_{2}}P_{w_{1},w_{2}} = 0 \ \mbox{ whenever } 
(v_{1}, w_{1}) \in E_{1} \not\Leftrightarrow (v_{2}, w_{2}) \in E_{2}.$$
Let $P = (P_{v,w})_{v, w \in V}$; it is clear that $P$ is a quantum permutation. For any $v, w \in V$, we have
\begin{eqnarray*}
((\mathsf{A}_{E_{2}}\otimes I_{\cl{H}})P)_{v, w} 
& = &
	\sum_{v' \in E_{2}^*(v)} \hspace{-0.2cm} P_{v',w} 
= 
	\sum\limits_{w' \in V} \hspace{-0.1cm} P_{v,w'}
	\sum\limits_{\substack{v' \in E_{2}^*(v)}} \hspace{-0.2cm} P_{v',w} \\
& = &
	\sum\limits_{\substack{v' \in E_{2}^*(v)}} 
	\sum\limits_{w' \in V} \hspace{-0.1cm} P_{v,w'} P_{v',w} 
= 
\sum\limits_{\substack{v' \in E_{2}^*(v)}} 
	\sum\limits_{w' \in E_1(w)} \hspace{-0.1cm} P_{v,w'} P_{v',w}\\
& = &
\sum\limits_{\substack{v' \in V}} 
	\sum\limits_{w' \in E_1(w)} \hspace{-0.1cm} P_{v,w'} P_{v',w}
= 
\sum\limits_{w' \in E_1(w)} \hspace{-0.1cm} P_{v,w'} \sum\limits_{\substack{v' \in V}} P_{v',w}\\
& = & 
	\sum\limits_{\substack{w' \in E_{1}(w)}}\hspace{-0.2cm} P_{v,w'}
= (P(\mathsf{A}_{E_{1}}\otimes I_{\cl{H}}))_{v, w}.
\end{eqnarray*}
This shows the validity of (\ref{eq_commre}). 
\end{proof}

\noindent {\bf Remark. } 
Let $G_1$ and $G_2$ be graphs on a vertex set $X$. The graph isomorphism game $G_1\cong G_2$
is defined in \cite{amrssv} and, according to \cite[Theorems 5.9 and 5.14]{amrssv}, 
$G_1$ and $G_2$ are quantum (resp. quantum commuting) isomorphic 
(denoted $G_1\cong_{\rm q} G_2$ (resp.  $G_1\cong_{\rm qc} G_2$))
if and only if there exists a quantum q-permutation (resp. quantum qc-permutation) $P$ acting on a Hilbert space $H$, 
such that 
$P(\mathsf{A}_{G_1}\otimes I_H) = (\mathsf{A}_{G_2}\otimes I_H)P$, where $\mathsf{A}_{G_1}$ and $\mathsf{A}_{G_2}$
are the adjacency matrices of the graphs $G_1$ and $G_2$, respectively. 
Thus, Theorem \ref{th_intertwine} generalises \cite[Theorems 5.9 and 5.14]{amrssv}.

\begin{lemma}\label{induce_gr_iso}
Let $E_{i} \subseteq V_{i}\times W_{i}$ be a hypergraph, $i = 1, 2$. If the pair $(f, g)$ of functions determines an isomorphism from $E_{1}$ to $E_{2}$, then $g$ is an isomorphism from $G_{E_{1}^*}$ to $G_{E_{2}^*}$. 
\end{lemma}

\begin{proof}
We set $G_i = G_{E_{i}^*}$, $i = 1,2$. 
Suppose that $x \sim x'$ in $G_1$, and let $v \in V$ be such that 
$v \in E_{1}(x)\cap E_{1}(x')$. Let $v' \in V$ be the unique element such that $f(v') = v$.
We have that $v' \in f^{-1}(E_{1}(x)) \cap f^{-1}(E_{1}(x'))$. 
As $f^{-1}(E_{1}(x)) = E_{2}(g(x))$ and $f^{-1}(E_{1}(x')) = E_{2}(g(x'))$, we have 
$v' \in E_{2}(g(x)) \cap E_{2}(g(x'))$, that is, $g(x) \sim g(x')$ in $G_2$. 

Now suppose $g(x) \sim g(x')$ in $G_2$, and let $v' \in V$ be such that 
$v' \in E_{2}(g(x)) \cap E_{2}(g(x'))$; then $f(v') \in f(E_{2}(g(x))) \cap f(E_{2}(g(x')))$. 
As $f^{-1}(E_{1}(x))$ $=$ $E_{2}(g(x))$, this implies $E_{1}(x) = f(E_{2}(g(x)))$. 
Thus, $f(v') \in E_{1}(x) \cap E_{1}(x')$, meaning that $x \sim x'$ in $G_1$. 
This shows that $g$ is an isomorphism. 
\end{proof}

Let $G$ be a graph with vertex set $X$. 
Recall \cite[Section 1.7]{GR} 
that the \emph{line graph} $L(G)$ of $G$ has as a vertex set
the set $L$ of all edges of $G$ and its adjacency relation is given by 
$$l \sim_{L(G)} l' \ \mbox{ if there exist } x,y,z\in X \mbox{ s.t. } x\neq z, l = \{x,y\} \mbox{ and } l' = \{y,z\}$$
(in other words, $l\sim_{L(G)} l'$ precisely when $l$ and $l'$ are distinct edges that share a common vertex). 
Let 
$$E_G = \{(x,x')\in X\times X : x\sim_G x'\},$$
considered as a hypergraph in $X\times X$, and 
$$F_G = \{((x,y),y) : x\sim_G y\},$$ 
considered as a hypergraph in $XX\times X$.

\begin{theorem}\label{th_examp}
Let $G_1$ and $G_2$ be graphs with vertex set $X$ such that $G_1\cong_{\rm q} G_2$ but $G_1\not\cong G_2$. 
Then 
\begin{itemize}
\item[(i)] 
$E_{G_1}\simeq_{\rm q} E_{G_2}$ but $E_{G_1}\not\simeq_{\rm loc} E_{G_2}$;

\item[(ii)] 
$F_{G_1}\simeq_{\rm q} F_{G_2}$ but $F_{G_1}\not\simeq_{\rm loc} F_{G_2}$. 
\end{itemize}
\end{theorem}


\begin{proof}
(i) 
Set $E_i = E_{G_i}$, $i = 1,2$. 
By \cite[Theorem 5.8]{amrssv}, 
there exist $d \in \bb{N}$ and a quantum permutation $P \in M_X\otimes M_d$ such that
\begin{gather*}
	(\mathsf{A}_{E_{1}}\otimes I_{d})P = P(\mathsf{A}_{E_{2}}\otimes I_{d}).
\end{gather*}
Theorem \ref{th_intertwine} now implies that $E_{1} \simeq_{\rm q} E_{2}$. 

By Proposition \ref{p_lochy2}, it now suffices to show that the hypergraphs $E_1$ and $E_2$ are not 
isomorphic. 
Assume, towards a contradiction, that there exists a pair $(f, g)$ of bijections, 
where $f: X\rightarrow X$ and $g: X\rightarrow X$, such that
$$f^{-1}(E_{1}(x)) = E_{2}(g(x)), \;\;\;\; x \in X.$$
By Lemma \ref{induce_gr_iso}, $g$ is an isomorphism from $G_{E_{1}^{*}}$ to $G_{E_{2}^{*}}$. 
Note that $G_{E_{i}^{*}}$ 
is either isomorphic to $L(G_{i})$ or 
contains $L(G_{i})$ as a connected component, potentially with additional isolated vertices.
It follows that $L(G_{1}) \cong L(G_{2})$. 
On the other hand, since quantum and classical isomorphism differ for the graphs $G_1$ and $G_2$, 
we have that the cardinality of the vertex sets of $G_1$ and $G_2$ exceeds four \cite[Section 3]{rs}. 
Thus, $L(G_{i}) \not \cong K_{3}, K_{1, 3}$ for $i = 1,2$ (where $K_{1, 3}$ stands for the bipartite graph on four vertices with three vertices in one disjoint set, and the remaining in the other).
Whitney's Isomorphism Theorem \cite{whitney} now implies that $G_{1} \cong G_{2}$, a contradiction.  

(ii) 
As in \cite[Section 2]{amrssv}, 
let ${\rm rel}(x,y)$ denote the relation between vertices of the graph $G$, of either being adjacent 
(${\rm rel}(x,y) = 1$), equal (${\rm rel}(x,y) = 0$) or non-adjacent (${\rm rel}(x,y) = -1$). 
By \cite[Theorem 5.9]{amrssv}, there exists a quantum permutation $P = (P_{x,y})_{x,y}$ over $X\times X$, 
acting on a (finite dimensional) Hilbert space, such that 
\begin{equation}\label{eq_xx'yy'}
P_{x,x'}P_{y,y'} = 0 \ \mbox{ if } {\rm rel}(x,y) \neq {\rm rel}(x',y').
\end{equation}
Abbreviating the notation $(x,y)$ to $xy$, given pairs $xy, ab\in XX$, let 
$Q_{xy,ab} = P_{y,b}P_{x,a}P_{y,b}$. 
Note that 
$$\sum_{ab\in XX} Q_{xy,ab} = \sum_{b\in X} P_{y,b} \left(\sum_{a\in X} P_{x,a}\right) P_{y,b} = 
\sum_{b\in X} P_{y,b} = I;$$
thus, the family $(Q_{xy,ab})_{ab\in XX}$ is a POVM, for every $xy\in XX$. 

Suppose that $(xy,y)\in F_{G_1}$ but $(ab,c)\not\in F_{G_2}$. 
Let $a\not\sim_{G_2} b$. Since $x\sim_{G_1} y$, by (\ref{eq_xx'yy'}) we have 
$$Q_{xy,ab}P_{y,c} = P_{y,b}(P_{x,a}P_{y,b})P_{y,c} = 0.$$
On the other hand, if $a\sim_{G_2} b$ then $c\neq b$ and hence, again, 
$$Q_{xy,ab}P_{y,c} = P_{y,b}P_{x,a}(P_{y,b}P_{y,c}) = 0.$$
Similarly, if $(xy,y)\not\in F_{G_1}$ but $(ab,c)\in F_{G_2}$, we obtain $Q_{xy,ab}P_{y,c} = 0$. 
Let $\xi$ be a maximally entangled vector in $H\otimes H$; thus, 
$$\langle (S\otimes T)\xi,\xi\rangle = {\rm Tr}(ST^{\rm t}), \ \ \ S,T\in \cl B(H).$$
A perfect quantum strategy $p$ for the isomorphism game $F_{G_1}\cong F_{G_2}$
is then given by letting 
$$p(ab,c | xy,z) = \left\langle (Q_{xy,ab}\otimes P_{y,c}^{\rm t})\xi,\xi\right\rangle, \ \ \ x,y,z,a,b,c\in X.$$
By Proposition \ref{p_flip}, $p$ gives rise to a perfect quantum strategy for the 
hypergraph isomorphism game $F_{G_1}\simeq F_{G_2}$.

Suppose that $F_{G_1}\simeq_{\rm loc} F_{G_2}$. 
By Proposition \ref{p_lochy2}, there exist bijections 
$f : XX\to XX$ and $g : X\to X$ such that 
$$(xy,z)\in F_{G_1} \ \Longleftrightarrow \ (f(xy),g(z))\in F_{G_2}.$$
We check that $f$ is an isomorphism from $L(G_1)$ onto $L(G_2)$, 
thus arriving at a contradiction as in (i). 
Suppose that $xy\sim_{L(G_1)} zy$, where $x,y,z\in X$, $x\neq z$, $x\sim_{G_1} y \sim_{G_1} z$. 
Write $f(xy) = ab$ and $f(zy) = cd$. Then $g(y) = b$ and hence $d = b$. 
Since $a\sim_{G_2} c\sim_{G_2} b$, we have that 
$ab\sim_{L(G_2)} cd$, that is, $f(xy)\sim_{L(G_2)} f(zy)$. By symmetry, 
$$f(xy)\sim_{L(G_2)} f(uv) \ \Longleftrightarrow \ xy\sim_{L(G_1)} uv.$$
\end{proof}

\begin{corollary}\label{c_exist}
There exist hypergraphs $E_1$ and $E_2$ such that $E_1\simeq_{\rm q} E_2$ but $E_1\not\simeq_{\rm loc} E_2$.
\end{corollary}

\begin{proof}
By \cite[Theorem 6.4]{amrssv}, there exist graphs $G_1$ and $G_2$ that are 
quantum isomorphic, but not isomorphic. The statement now is a consequence of Theorem \ref{th_examp}. 
\end{proof}


\section{Strongly no-signalling correlations}\label{s_strongly}

In the rest of the paper, we restrict the setup of Sections \ref{ss_gensetup}-\ref{s_hypis} to the special 
case where the underlying hypergraphs are non-local games. 
We start by introducing, in this section, the types of correlations that will serve as suitable strategies.



Let $X$, $Y$, $A$ and $B$ be finite sets, and $H$ be a Hilbert space. 
In the sequel, to simplify notation, if there is no risk of confusion, 
we will abbreviate an ordered pair $(x,y)$ in $X\times Y$ to $xy$. 
A positive operator $P = (P_{xy,ab})_{xy,ab}\in \cl D_{XYAB}\otimes \cl B(H)$ will be called 
a \emph{no-signalling (NS) operator matrix} if the \emph{marginal operators} 
$$P_{x,a} := \sum_{b\in B} P_{xy,ab} \ \mbox{ and } \ P^{y,b} :=\sum_{a\in A} P_{xy,ab}$$ 
are well-defined, and 
$(P_{x,a})_{a\in A}$ (and hence $(P^{y,b})_{b\in B}$) is a POVM for every $x\in X$ (and every $y\in Y$). 
This notion formed the base for the concept of a \emph{nonsignalling operator system} in \cite[Definition 5.2]{art}, 
although it was not defined there explicitly. 
An NS operator matrix $P = (P_{xy,ab})_{xy,ab}$ is called \emph{dilatable} if 
there exist a Hilbert space $K$, an isometry $V : H\to K$ and POVM's
$(E_{x,a})_{a\in A}$ and $(F_{y,b})_{b\in B}$ on $K$, $x\in X$, $y\in Y$, such that $E_{x,a}F_{y,b} = F_{y,b}E_{x,a}$
and 
\begin{equation}\label{eq_ExaFyb}
P_{xy,ab} = V^* E_{x,a}F_{y,b} V, \ \ \ x\in X, y\in Y, a\in A, b\in B.
\end{equation}

\medskip

\noindent {\bf Remark. } 
If the entries $P_{xy,ab}$ of an NS operator matrix 
are projections then
$$P_{x,a}P^{y,b} = \sum_{a'\in A}\sum_{b'\in B} P_{xy,ab'}P_{xy,a'b} = P_{xy,ab}$$
for all $x,y,a,b$. 

\medskip

We recall the operator system $\cl S_{X,A}$ and the C*-algebra $\cl A_{X,A}$, introduced 
before Remark \ref{r_aff}, whose canonical generators are the the elements 
$\tilde{e}_{x,a}$ of universal PVM's $\{\tilde{e}_{x,a}\}_{a\in A}$, $x\in X$. 
For clarity, we will denote the canonical generators of the operator system 
$\cl S_{Y,B}$ by $\tilde{f}_{y,b}$, $y\in Y$, $b\in B$.

\begin{proposition}\label{p_dedil}
If $P = (P_{xy,ab})_{xy,ab}$ is a dilatable NS operator matrix acting on the Hilbert 
space $H$ then there exists a unital completely positive map
$\gamma : \cl S_{X,A}\otimes_{\rm c} \cl S_{Y,B} \to \cl B(H)$, such that 
$\gamma(\tilde{e}_{x,a}\otimes \tilde{f}_{y,b}) = P_{xy,ab}$.
Conversely, if 
$\gamma : \cl S_{X,A}\otimes_{\rm c} \cl S_{Y,B} \to \cl B(H)$ is a unital completely positive map then 
$\left(\gamma(\tilde{e}_{x,a}\otimes \tilde{f}_{y,b})\right)_{xy,ab}$ is a dilatable NS operator matrix. 
\end{proposition}

\begin{proof}
Let $K$ be a Hilbert space, $V : H\to K$ be an isometry, 
and $(E_{x,a})_{a\in A}$ and $(F_{y,b})_{b\in B}$ be mutually commuting 
POVM's on $K$ satisfying (\ref{eq_ExaFyb}).
The linear map $\phi : \cl S_{X,A}\to \cl B(K)$ (resp. 
$\psi : \cl S_{Y,B}\to \cl B(K)$), given by $\phi(\tilde{e}_{x,a}) = E_{x,a}$ (resp. $\psi(\tilde{f}_{x,a}) = F_{x,a}$) 
is (unital and) completely positive. By the definition of the commuting tensor product, 
the map $\phi\cdot \psi : \cl S_{X,A}\otimes_{\rm c} \cl S_{Y,B} \to \cl B(K)$, given by 
$(\phi\cdot \psi)(u\otimes v) = \phi(u)\psi(v)$, is (unital and) completely positive. 
Set
$$\gamma(w) = V^*(\phi\cdot \psi)(w)V, \ \ \ w\in \cl S_{X,A}\otimes_{\rm c} \cl S_{Y,B};$$
we have that $\gamma$ is unital and completely positive, and 
$\gamma(\tilde{e}_{x,a}\otimes \tilde{f}_{y,b}) = P_{xy,ab}$, $x\in X$, $y\in Y$, $a\in A$, $b\in B$.

Conversely, suppose that 
$\gamma : \cl S_{X,A}\otimes_{\rm c} \cl S_{Y,B} \to \cl B(H)$ is a unital completely positive map. 
By \cite[Lemma 2.8]{pt}, 
$\cl S_{X,A}\otimes_{\rm c} \cl S_{Y,B} \subseteq_{\rm c.o.i} \cl A_{X,A}\otimes_{\max} \cl A_{Y,B}$. 
Using Arveson Extension Theorem, let 
$\tilde{\gamma} : \cl A_{X,A}\otimes_{\max} \cl A_{Y,B} \to \cl B(H)$ be a completely positive extension of $\gamma$. 
Applying Stinespring's Theorem, write 
$$\tilde{\gamma}(w) = V^*\pi(w)V, \ \ \ w\in \cl A_{X,A}\otimes_{\max} \cl A_{Y,B},$$
for some *-representation $\pi$ of $\cl A_{X,A}\otimes_{\max} \cl A_{Y,B}$ on a Hilbert space $K$ and an isometry 
$V : H\to K$. Letting $E_{x,a} = \pi(\tilde{e}_{x,a}\otimes 1)$ and $F_{y,b} = \pi(1\otimes \tilde{f}_{y,b})$, we 
obtain a representation (\ref{eq_ExaFyb}) for the matrix
$\left(\gamma(\tilde{e}_{x,a}\otimes \tilde{f}_{y,b})\right)_{xy,ab}$. 
\end{proof}

The following fact is implicit in the proof of Proposition \ref{p_dedil}:

\begin{corollary}\label{c_dedil2}
If $P = (P_{xy,ab})_{xy,ab}$ is a dilatable NS operator matrix then $(E_{x,a})_{a\in A}$, $x\in X$, 
and $(F_{y,b})_{b\in B}$, $y\in Y$, in (\ref{eq_ExaFyb}) can be chosen to be PVM's.
\end{corollary}

\begin{remark}\label{r_nondi}
\rm 
There exist non-dilatable NS operator matrices whenever the cardinalities of $X$, $Y$, $A$
are at least $2$. 
Indeed, let $\phi : \cl S_{X,A}\otimes_{\max}\cl S_{Y,B}\to \cl B(H)$ be a unital complete order embedding, 
and set $P_{xy,ab} = \phi(\tilde{e}_{x,a}\otimes \tilde{f}_{y,b})$, $x\in X$, $y\in Y$, $a\in A$, $b\in B$. 
It is clear that $(P_{xy,ab})_{xy,ab}$ is an NS operator matrix. Suppose that it is dilatable; 
by Proposition \ref{p_dedil}, there exists a unital completely positive map 
$\psi : \cl S_{X,A}\otimes_{\rm c}\cl S_{Y,B}\to \cl B(H)$ such that 
$\psi(\tilde{e}_{x,a}\otimes \tilde{f}_{y,b}) = P_{xy,ab}$, $x\in X$, $y\in Y$, $a\in A$, $b\in B$.
The map $\phi^{-1}\circ \psi : \cl S_{X,A}\otimes_{\rm c}\cl S_{Y,B} \to \cl S_{X,A}\otimes_{\max}\cl S_{Y,B}$ 
is completely positive, and by the extremal property of the maximal 
operator system tensor product (see \cite[Theorem 5.5]{kptt}), it is a (unital) complete order isomorphism. 
By virtue of \cite[Theorem 3.1]{lmprsstw}, 
this contradicts the fact that $\cl C_{\rm ns}\neq \cl C_{\rm qc}$ (see e.g. \cite[Corollary 7.12]{fkpt_NYJ}). 
\end{remark}

In the next proposition, we identify the NS operator matrices that give rise to 
local NS correlations. 
Call a NS operator matrix $(P_{xy,ab})_{x,y,a,b}$ \emph{locally dilatable}
if it admits a dilation of the form (\ref{eq_ExaFyb}), where the family
$\{E_{x,a}, F_{y,b} : x\in X, y\in Y, a\in A, b\in B\}$ is commutative.

\begin{proposition}\label{p_locdedi}
An NS correlation $p = \{(p(a,b|x,y))_{a,b} : (x,y)\in X\times Y\}$ over $(X,Y,A,B)$ is 
local if and only if there exists a Hilbert space $H$, a locally dilatable
NS operator matrix $(P_{xy,ab})_{xy,ab}$ and a unit vector $\xi\in H$ such that 
\begin{equation}\label{eq_loccdi}
p(a,b|x,y) = \langle P_{xy,ab}\xi,\xi\rangle, \ \ \ x\in X, y\in Y, a\in A, b\in B.
\end{equation}
\end{proposition}

\begin{proof}
Assume that $p\in \cl C_{\rm loc}$, namely, 
\begin{equation}\label{eq_cocom}
p = \sum_{i=1}^k \lambda_i p_i^{(1)}\otimes p_i^{(2)} 
\end{equation}
as a convex combination, where $p_i^{(1)} = (p_i^{(1)}(a|x))$ (resp. $p_i^{(2)} = (p_i^{(2)}(b|y))$) are 
conditional probability distributions.
Set $P_{xy,ab} = (p_i^{(1)}(a|x)p_i^{(2)}(b|y))_{i=1}^k$, 
considered as a matrix in $\cl D_k$. 
The representation (\ref{eq_loccdi}) is obtained by 
letting $\xi = (\sqrt{\lambda_i})_{i=1}^k\in \bb{C}^k$ and
$E_{x,a}$ (resp. $F_{y,b}$) be the diagonal matrix with diagonal $(p_i^{(1)}(a|x))_{i=1}^k$
(resp. $(p_i^{(2)}(b|y))_{i=1}^k$). 

Conversely, suppose that $(P_{xy,ab})_{xy,ab}$ is a locally dilatable NS operator matrix  
satisfying (\ref{eq_loccdi}). By replacing the Hilbert space $H$ with the Hilbert space $K$
arising from the dilation (\ref{eq_ExaFyb}) of $(P_{xy,ab})_{xy,ab}$, and the vector $\xi$ with the vector 
$V\xi$,  we may assume that 
$P_{xy,ab} = E_{x,a}F_{y,b}$, $x\in X$, $y\in Y$, $a\in A$, $b\in B$, 
where the family $\{E_{x,a}, F_{y,b} : x\in X, y\in Y, a\in A, b\in B\}$ is commutative.

Let $\cl A$ (resp. $\cl B)$ be the C*-algebra, generated by 
$\{E_{x,a} : x\in X, a\in A\}$ (resp. $\{F_{y,b} : y\in Y, b\in B\}$), 
and let $s : \cl A\otimes_{\max} \cl B\to \bb{C}$ be the state, given by 
$s(S\otimes T) = \langle ST\xi,\xi\rangle$. 
Using the nuclearity of abelian C*-algebras, we view
$s$ as a state on $\cl A\otimes_{\min}\cl B$. 
Identify $\cl A = C(\Omega_1)$ and $\cl B = C(\Omega_2)$, for some compact Hausdorff spaces $\Omega_1$ and 
$\Omega_2$, 
and the state $s$ with a Borel probability measure $\mu$ on the product 
topological space $\Omega_1\times\Omega_2$. 
We thus have 
$$p(a,b|x,y) = \int_{\Omega_1\times\Omega_2} E_{x,a}(\omega_1)F_{y,b}(\omega_2) d\mu(\omega_1,\omega_2),$$
$x\in X$, $y\in Y$, $a\in A$, $b\in B$. 
Approximating $\mu$ with convex combinations of product measures $\mu_1\times\mu_2$, 
we see that $p$ can be approximated by convex combinations of the form (\ref{eq_cocom}). 
By the Carath\'eodory Theorem, the number of 
terms in the sum in each of the approximants of the form (\ref{eq_cocom})
can be chosen to be at most $|X||Y||A||B| + 1$. Using a standard compactness argument, 
we conclude that $p$ is itself of the form (\ref{eq_cocom}). 
\end{proof}

\begin{remark}\label{r_qos}
\rm 
Call an NS operator matrix $P = (P_{xy,ab})_{xy,ab}$, acting on a Hilbert space $H$, 
\emph{quantum dilatable} if 
there exist families
$(E_{x,a})_{a\in A}$ and $(F_{y,b})_{b\in B}$ of POVM's, acting on
finite dimensional Hilbert spaces $H_A$ and $H_B$, respectively, and an isometry $V : H\to H_A\otimes H_B$, 
such that 
$$P_{xy,ab} = V^* (E_{x,a}\otimes F_{y,b}) V, \ \ \ x\in X, y\in Y, a\in A, b\in B.$$
It is clear that a NS correlation $p$ over $(X,Y,A,B)$ is quantum if and only if 
there exists a quantum dilatable 
NS operator matrix $P = (P_{xy,ab})_{xy,ab}$ acting on a Hilbert space $H$ 
and a unit vector $\xi\in H$ such that 
$$p(a,b|x,y) = \langle P_{xy,ab}\xi,\xi\rangle, \ \ \ x\in X, y\in Y, a\in A, b\in B.$$
It follows from \cite[Theorem 7.4]{art2} that an NS operator matrix is quantum dilatable 
precisely when its entries 
generate a quantum $k$-AOU space for some $k\in \bb{N}$ in the sense of \cite{art2}
(we refer the reader to \cite[Section 7]{art2}
for the definition and the properties of the latter type of Archimedean ordered spaces).
\end{remark}

We now introduce the suitable correlation types for hypergraph homomorphism games, 
provided the hypergraphs are non-local games in their own right. 
Thus, in the notation of Section \ref{s_hyper}, we assume that 
$V_i = X_i Y_i$ and $W_i = A_i B_i$ for some finite sets $X_i$, $Y_i$, $A_i$ and $B_i$, 
and let $E_i\subseteq X_i Y_i\times A_i B_i$, $i = 1,2$. 
We note that $X_i Y_i$ (resp. $A_i B_i$) is interpreted as 
the question (resp. answer) set for the game $E_i$, $i = 1,2$. 
A channel 
$$\Gamma : \cl D_{X_2Y_2 \times A_1B_1}\to \cl D_{X_1Y_1 \times A_2B_2}$$
will be called a \emph{strongly no-signalling (SNS) correlation} if 
$$\sum_{b_2\in B_2} \hspace{-0.1cm}
\Gamma(x_1y_1, a_2b_2 | x_2y_2, a_1b_1) \hspace{-0.05cm} = \hspace{-0.2cm}
\sum_{b_2\in B_2} \hspace{-0.1cm}
\Gamma(x_1y_1, a_2b_2 | x_2y_2, a_1b_1'), \ \ b_1,b_1'\in B_1,$$
$$\sum_{a_2\in A_2} \hspace{-0.1cm}
\Gamma(x_1y_1, a_2b_2 | x_2y_2, a_1b_1) \hspace{-0.05cm} = \hspace{-0.2cm}
\sum_{a_2\in A_2} \hspace{-0.1cm}
\Gamma(x_1y_1, a_2b_2 | x_2y_2, a_1'b_1), \ \ a_1,a_1'\in A_1,$$
$$\sum_{y_1\in Y_1} \hspace{-0.1cm}
\Gamma(x_1y_1, a_2b_2 | x_2y_2, a_1b_1) \hspace{-0.05cm} = \hspace{-0.2cm}
\sum_{y_1\in Y_1} \hspace{-0.1cm}
\Gamma(x_1y_1, a_2b_2 | x_2y_2', a_1b_1), \ \ y_2,y_2'\in Y_2,$$
and
$$\sum_{x_1\in X_1} \hspace{-0.1cm}
\Gamma(x_1y_1, a_2b_2 | x_2y_2, a_1b_1) \hspace{-0.05cm} = \hspace{-0.2cm}
\sum_{x_1\in X_1} \hspace{-0.1cm}
\Gamma(x_1y_1, a_2b_2 | x_2'y_2, a_1b_1), \ \ x_2,x_2'\in X_2.$$

\noindent 
We denote by $\cl C_{\rm sns}$ the (convex) set of all SNS correlations 
(the specific question-answer sets will be understood from the context), and note that 
$\cl C_{\rm sns}\subseteq \cl C_{\rm ns}$. 
For a subset $\Lambda\subseteq V_2\times W_1\times V_1\times W_2$, we write
$\cl C_{\rm sns}(\Lambda)$ for the set of all SNS correlations with support contained in $\Lambda$. 
If $\Gamma\in \cl C_{\rm sns}$, we write 
$$
\Gamma(x_1y_1, a_2 | x_2y_2, a_1), \ \ 
\Gamma(x_1y_1, b_2 | x_2y_2, b_1),$$
$$\Gamma(x_1, a_2b_2 | x_2, a_1b_1) \ \mbox{ and } \ 
\Gamma(y_1, a_2b_2 | y_2, a_1b_1)$$
for the corresponding marginal conditional probability distributions, which are well-defined
by the definition of strong no-signalling. 
The SNS conditions imply that the further conditional probability distributions 
$$\Gamma(x_1, a_2 | x_2, a_1) = \sum_{y_1\in Y_1} \Gamma(x_1y_1, a_2 | x_2y_2, a_1),$$
$$\Gamma(y_1, b_2 | y_2, b_1) = \sum_{x_1\in X_1} \Gamma(x_1y_1, b_2 | x_2y_2, b_1)$$
$$\Gamma(y_1, a_2 | y_2, a_1) = \sum_{x_1\in X_1} \Gamma(x_1y_1, a_2 | x_2y_2, a_1)$$
and 
$$\Gamma(x_1, b_2 | x_2, b_1) = \sum_{y_1\in Y_1} \Gamma(x_1y_1, b_2 | x_2y_2, b_1)$$
are well-defined and are no-signalling correlations in their own right; 
for example, for $x_1\in X_1$, $x_2\in X_2$, $a_1,a_1'\in A_1$ and $a_2\in A_2$ we have
\begin{eqnarray*}
\sum_{a_2\in A_2} \Gamma(x_1, a_2 | x_2, a_1)
& = & 
\sum_{a_2\in A_2} \sum_{y_1\in Y_1} \Gamma(x_1y_1, a_2 | x_2y_2, a_1)\\
& = & 
\sum_{a_2\in A_2} \sum_{y_1\in Y_1} \sum_{b_2\in B_2}\Gamma(x_1y_1, a_2b_2 | x_2y_2, a_1b_1')\\
& = & 
\sum_{y_1\in Y_1} \sum_{b_2\in B_2}\sum_{a_2\in A_2}  \Gamma(x_1y_1, a_2b_2 | x_2y_2, a_1b_1')\\
& = & 
\sum_{y_1\in Y_1} \sum_{b_2\in B_2}\sum_{a_2\in A_2}  \Gamma(x_1y_1, a_2b_2 | x_2y_2, a_1'b_1')\\
& = & 
\sum_{a_2\in A_2} \Gamma(x_1, a_2 | x_2, a_1').
\end{eqnarray*}

\begin{definition}\label{d_snsc}
An SNS correlation $\Gamma$ over the quadruple $(X_2Y_2, A_1B_1,$ $X_1Y_1,$ $A_2B_2)$
is called 
\begin{itemize}
\item[(i)] \emph{quantum commuting} if there exist a Hilbert space $H$, 
dilatable NS operator matrices 
$P = (P_{x_2y_2,x_1y_1})_{x_2y_2,x_1y_1}$ 
and $Q = (Q_{a_1b_1,a_2b_2})_{a_1b_1,a_2b_2}$ on $H$ with mutually commuting entries, 
and a unit vector $\xi\in H$, such that 
\begin{equation}\label{eq_qcsns}
\Gamma(x_1y_1,a_2b_2 | x_2y_2,a_1b_1) = 
\left\langle 
P_{x_2y_2,x_1y_1} Q_{a_1b_1,a_2b_2}\xi,\xi\right\rangle
\end{equation}
for all $x_i\in X_i$, $y_i\in Y_i$, $a_i\in A_i$ and $b_i\in B_i$, $i = 1,2$;

\item[(ii)] \emph{quantum} 
if there exist finite dimensional Hilbert spaces $H$ and $K$,
quantum dilatable NS operator matrices 
$$M = (M_{x_2a_1,x_1a_2})_{x_2,a_1,x_1,a_2} \mbox{ on } H, \mbox{ and } 
N = (N_{y_2b_1,y_1b_2})_{y_2,b_1,y_1,b_2} \mbox{ on } K,$$
and a unit vector $\xi\in H\otimes K$, such that 
\begin{equation}\label{eq_qsns}
\Gamma(x_1y_1,a_2b_2 | x_2y_2,a_1b_1) = 
\left\langle \left(M_{x_2a_1,x_1a_2}\otimes N_{y_2b_1,y_1b_2}\right)\xi,\xi\right\rangle
\end{equation}
for all $x_i\in X_i$, $y_i\in Y_i$, $a_i\in A_i$ and $b_i\in B_i$, $i = 1,2$;

\item[(iii)] \emph{approximately quantum} if it is a limit of quantum SNS correlations; 

\item[(iv)] \emph{local} if it is quantum, and the matrices $M$ and $N$ from (ii) can be chosen to
be locally dilatable.  
\end{itemize}
\end{definition}

We denote by $\cl C_{\rm sqc}$ (resp. $\cl C_{\rm sqa}$, $\cl C_{\rm sq}$ and $\cl C_{\rm sloc}$) 
the classes of quantum commuting (resp. approximately quantum, quantum and local) SNS correlations.


\begin{remark}\label{r_desloc}
\rm 
Let $\Gamma$ be a local SNS correlation $\Gamma$ over the quadruple $(X_2Y_2,$ $A_1B_1,$ $X_1Y_1,A_2B_2)$.
By choosing dilations of the matrices $M$ and $N$ in (\ref{eq_qsns}) with mutually commuting 
entries, we can write $\Gamma$ in the form 
$$\Gamma(x_1y_1,a_2b_2 | x_2y_2,a_1b_1) = 
\left\langle E_{x_2,x_1}E^{a_1,a_2} F_{y_2,y_1}F^{b_1,b_2}\xi,\xi\right\rangle,$$
where the POVM's $(E_{x_2,x_1})_{x_1}$, $(E^{a_1,a_2})_{a_2}$, 
$(F_{y_2,y_1})_{y_1}$ and $(F^{b_1,b_2})_{b_2}$ have mutually commuting entries. 
An argument, similar to the one in the proof of
Proposition \ref{p_locdedi} now shows that 
$\Gamma = \sum_{i=1}^k \lambda_i \Phi_i\otimes \Psi_i$ as a convex combination
for some local NS correlations $\Phi_i : \cl D_{X_2Y_2}\to \cl D_{X_1Y_1}$ and 
$\Psi_i : \cl D_{A_1B_1}\to \cl D_{A_2B_2}$, $i = 1,\dots,k$. 
\end{remark}

The following lemma will be used in Sections \ref{s_nonlocal} and \ref{s_repSNS}. 

\begin{lemma}\label{l_comlift}
The SNS correlation $\Gamma$ belongs to $\cl C_{\rm sqc}$ if and only if 
there exists a Hilbert space $K$, PVM's $(P_{x_2,x_1})_{x_1\in X_1}$, 
$(P^{y_2,y_1})_{y_1\in Y_1}$, $(Q_{a_1,a_2})_{a_2\in A_2}$ and $(Q^{b_1,b_2})_{b_2\in B_2}$ on $K$
with mutually commuting entries, and a unit vector $\eta\in K$, such that 
\begin{equation}\label{eq_qcsns2}
\Gamma(x_1y_1,a_2b_2 | x_2y_2,a_1b_1) = 
\left\langle 
P_{x_2,x_1}P^{y_2,y_1} Q_{a_1,a_2}Q^{b_1,b_2}\eta,\eta\right\rangle
\end{equation}
for all $x_i\in X_i$, $y_i\in Y_i$, $a_i\in A_i$ and $b_i\in B_i$, $i = 1,2$.
\end{lemma}

\begin{proof}
Let $\Gamma\in \cl C_{\rm sqc}$. Suppose that $H$ is a Hilbert space, 
$$P = (P_{x_2y_2,x_1y_1})_{x_2y_2,x_1y_1} \ \mbox{ and } \ Q = (Q_{a_1b_1,a_2b_2})_{a_1b_1,a_2b_2}$$ 
are dilatable NS operator matrices 
acting on $H$ with mutually commuting entries, 
and $\xi\in H$ is a unit vector, for which (\ref{eq_qcsns}) holds. 
By Proposition \ref{p_dedil}, \cite[Lemma 2.8]{pt} and the Stinespring Theorem, 
there exist a Hilbert space $K$, a unital *-representation
$\pi : \cl A_{X_2,X_1}\otimes_{\max}\cl A_{Y_2,Y_1}\to \cl B(K)$, and 
an isometry $V : H\to K$, such that 
$$P_{x_2y_2,x_1y_1}  = V^*\pi(\tilde{e}_{x_2,x_1}\otimes \tilde{e}_{y_2,y_1})V, \ \ \ 
x_i\in X_i, y_i\in Y_i, i = 1,2.$$
After replacing $K$ with the closure of the span of $\pi(\cl A_{X_2,X_1}\otimes_{\max}\cl A_{Y_2,Y_1})VH$, 
we may assume that the latter span is dense in $K$. 
Let $\cl N$ be the C*-algebra, generated by the family $\{Q_{a_1b_1,a_2b_2} : a_i\in A_i, b_i\in B_i, i = 1,2\}$. 
By Arveson's Commutant Lifting Theorem \cite[Theorem 12.7]{Pa}, 
there exists a *-representation
$$\rho : \cl N \to \pi(\cl A_{X_2,X_1}\otimes_{\max}\cl A_{Y_2,Y_1})',$$
which is unital by the uniqueness clause of the theorem, such that 
$$Vu = \rho(u) V, \ \ \ u\in \cl N.$$
Set $\tilde{Q}_{a_1b_1,a_2b_2} = \rho(Q_{a_1b_1,a_2b_2})$, $a_i\in A_i$, $b_i\in B_i$, $i = 1,2$.
Since the map $\rho : \cl S_{A_1,A_2}\otimes_{\rm c}\cl S_{B_1,B_2}\to \cl B(K)$ is unital and completely positive, 
Proposition \ref{p_dedil} implies that the NS operator matrix 
$(\tilde{Q}_{a_1b_1,a_2b_2})_{a_1,a_2,b_1,b_2}$ is dilatable. 
Let $\tilde{H}$ be a Hilbert space, 
$\tilde{\rho} : \cl A_{A_1,A_2}\otimes_{\max}\cl A_{B_1,B_2}\to \cl B(\tilde{H})$ be a unital *-representation, and 
$W : K\to \tilde{H}$ be an isometry, such that 
$$\tilde{Q}_{a_1b_1,a_2b_2}  = W^*\tilde{\rho}(\tilde{e}_{a_1,a_2}\otimes \tilde{e}_{b_1,b_2})W, \ \ \ 
a_i\in A_i, b_i\in B_i, i = 1,2;$$
assume, without loss of generality, that $\tilde{\rho}(\cl A_{A_1,A_2}\otimes_{\max}\cl A_{B_1,B_2})WK$ 
has dense span in $\tilde{H}$.
Applying Arveson's Commutant Lifting Theorem again, we obtain a 
*-homomorphism 
$$\tilde{\pi} : \cl A_{X_2,X_1}\otimes_{\max}\cl A_{Y_2,Y_1}\to \tilde{\rho}(\cl A_{A_1,A_2}\otimes_{\max}\cl A_{B_1,B_2})'$$
such that 
$$W\pi(u) = \tilde{\pi}(u)W,  \ \ \  u\in \cl A_{X_2,X_1}\otimes_{\max}\cl A_{Y_2,Y_1}.$$
Set 
$$P_{x_2,x_1} = \tilde{\pi}(\tilde{e}_{x_2,x_1}\otimes 1), \ \ \ 
P^{y_2,y_1}  = \tilde{\pi}(1\otimes \tilde{e}_{y_2,y_1}),$$
$$Q_{a_1,a_2} = \tilde{\rho}(\tilde{e}_{a_1,a_2}\otimes 1), \ \ \ 
Q^{b_1,b_2}  = \tilde{\rho}(1\otimes \tilde{e}_{b_1,b_2}),$$
and $\eta = WV\xi$ to obtain the representation (\ref{eq_qcsns2}).

Conversely, assuming (\ref{eq_qcsns2}), we have that the 
NS operator matrices with entries 
$P_{x_2y_2,x_1y_1} := P_{x_2,x_1}P^{y_2,y_1}$ and 
$Q_{a_1b_1,a_2b_2} := Q_{a_1,a_2}Q^{b_1,b_2}$ are (trivially dilatable and) commuting, 
showing that $\Gamma$ is a quantum commuting SNS correlation. 
\end{proof}


\section{Homomorphisms between non-local games}\label{s_nonlocal}

In this section, we demonstrate how hypergraph homomorphisms, defined in 
Sections \ref{s_hyper} and \ref{s_hypis}, give rise to homomorphisms between non-local games. 
In the next subsection we restrict the simulation paradigm to the case of NS correlations.



\subsection{Strategy transport}\label{ss_transport}

Let $X_i$, $Y_i$, $A_i$ and $B_i$ be finite sets, $i = 1,2$.

\begin{theorem}\label{th_pres}
Let $\Gamma$ be an SNS correlation over the quadruple
$(X_2Y_2, A_1B_1,$ $X_1Y_1,A_2B_2)$ and $\cl E$ be an NS correlation 
over the quadruple $(X_1,Y_1,A_1,B_1)$. The following hold:
\begin{itemize}
\item[(i)] $\Gamma[\cl E] \in \cl C_{\rm ns}$;

\item[(ii)] if $\Gamma\in \cl C_{\rm sqc}$ and $\cl E\in \cl C_{\rm qc}$ then $\Gamma[\cl E] \in \cl C_{\rm qc}$;

\item[(iii)] if $\Gamma\in \cl C_{\rm sqa}$ and $\cl E\in \cl C_{\rm qa}$ then $\Gamma[\cl E] \in \cl C_{\rm qa}$;

\item[(iv)] if $\Gamma\in \cl C_{\rm sq}$ and $\cl E\in \cl C_{\rm q}$ then $\Gamma[\cl E] \in \cl C_{\rm q}$;

\item[(v)] if $\Gamma\in \cl C_{\rm sloc}$ and $\cl E\in \cl C_{\rm loc}$ then $\Gamma[\cl E] \in \cl C_{\rm loc}$.
\end{itemize}
\end{theorem}

\begin{proof}
(i) 
Set $\cl F = \Gamma[\cl E]$, and fix $x_2\in X_2$, $y_2\in Y_2$ and 
$a_2\in A_2$. 
We have 
\begin{eqnarray*}
& & 
\sum_{b_2\in B_2} \cl F(a_2,b_2|x_2,y_2)\\
& = & 
\sum_{b_2\in B_2}
\sum_{x_1y_1\in X_1Y_1} \sum_{a_1b_1\in A_1B_1}
\Gamma(x_1y_1, a_2b_2 | x_2y_2, a_1b_1)\cl E(a_1,b_1|x_1,y_1)\\
& = & 
\sum_{x_1y_1\in X_1Y_1} \sum_{a_1b_1\in A_1B_1}
\Gamma(x_1y_1, a_2| x_2y_2, a_1)\cl E(a_1,b_1|x_1,y_1)\\
& = & 
\sum_{x_1y_1\in X_1Y_1} \sum_{a_1\in A_1}
\Gamma(x_1y_1, a_2| x_2y_2, a_1)\cl E(a_1|x_1)\\
& = & 
\sum_{x_1\in X_1} \sum_{a_1\in A_1}
\Gamma(x_1, a_2| x_2, a_1)\cl E(a_1|x_1),
\end{eqnarray*}
and hence the marginal $\cl F(a_2|x_2)$ is well-defined. 
Similarly, the marginal $\cl F(b_2|y_2)$ is well-defined. 

\smallskip

(ii) 
Appealing to Lemma \ref{l_comlift}, let 
$(P_{x_2,x_1})_{x_1\in X_1}$, $(P^{y_2,y_1})_{y_1\in Y_1}$, 
$(Q_{a_1,a_2})_{a_2\in A_2}$ and $(Q^{b_1,b_2})_{b_2\in B_2}$ be mutually commuting POVM's
on a Hilbert space $H$ and $\xi\in H$ be a unit vector, such that 
$$\Gamma(x_1y_1,a_2b_2 | x_2y_2,a_1b_1) = 
\left\langle P_{x_2,x_1} P^{y_2,y_1} Q_{a_1,a_2}Q^{b_1,b_2}\xi,\xi\right\rangle$$
for all $x_i$, $y_i$, $a_i$, $b_i$, $i = 1,2$. 
Let $(E_{x_1,a_1})_{a_1\in A_1}$ and $(F_{y_1,b_1})_{b_1\in B_1}$ be mutually commuting families of 
POVM's on a Hilbert space $K$ and $\eta\in H$ be a unit vector 
such that 
$$\cl E(a_1,b_1|x_1,y_1) 
= \langle E_{x_1,a_1}F_{y_1,b_1}\eta,\eta\rangle, \ \ x_1\in X_1, y_1\in Y_1, a_1\in A_1, b_1\in B_1.$$
Set 
\begin{equation}\label{eq_tildeExa}
\tilde{E}_{x_2,a_2} = \sum_{x_1\in X_1}\sum_{a_1\in A_1} P_{x_2,x_1}Q_{a_1,a_2} \otimes E_{x_1,a_1}
\end{equation}
and
\begin{equation}\label{eq_tildeFyb}
\tilde{F}_{y_2,b_2} = \sum_{y_1\in Y_1}\sum_{b_1\in B_1} P^{y_2,y_1}Q^{b_1,b_2} \otimes F_{y_1,b_1}.
\end{equation}
We have that $(\tilde{E}_{x_2,a_2})_{a_2\in A_2}$ (resp. $(\tilde{F}_{y_2,b_2})_{b_2\in B_2}$) 
is a POVM, $x_2\in X_2$ (resp. $y_2\in Y_2$), acting on the Hilbert space $H\otimes K$. 
In addition, 
$$\tilde{E}_{x_2,a_2}\tilde{F}_{y_2,b_2} = \tilde{F}_{y_2,b_2}\tilde{E}_{x_2,a_2}, \ \ \ x_2\in X_2, y_2\in Y_2,
a_2\in A_2, b_2\in B_2,$$ 
and 
$$\Gamma[\cl E](a_2,b_2|x_2,y_2) = \left\langle \tilde{E}_{x_2,a_2}\tilde{F}_{y_2,b_2}
(\xi\otimes \eta),\xi\otimes \eta\right\rangle$$
for all $x_2,y_2,a_2,b_2$, showing that $\Gamma[\cl E]$ is quantum commuting.

(iv) Let 
$M = (M_{x_2a_1,x_1a_2})_{x_2,a_1,x_1,a_2}$ and 
$N = (N_{y_2b_1,y_1b_2})_{y_2,b_1,y_1,b_2}$ be quantum dilatable 
NS operator matrices, acting on finite dimensional 
Hilbert spaces $H$ and $K$, respectively, and 
$\xi\in H\otimes K$ be a unit vector, for which $\Gamma$ admits a representation of the form (\ref{eq_qsns}). 
Write 
$$\cl E(a_1,b_1|x_1,y_1) 
\hspace{-0.05cm} = \hspace{-0.05cm} \langle (E_{x_1,a_1}\otimes F_{y_1,b_1})\eta,\eta\rangle, 
\ x_1\in X_1, y_1\in Y_1, a_1\in A_1, b_1\in B_1,$$
where $(E_{x_1,a_1})_{a_1\in A_1}$ and $(F_{y_1,b_1})_{b_1\in B_1}$ are finite dimensionally acting. 
Define
$$\tilde{E}_{x_2,a_2} = \sum_{x_1\in X_1}\sum_{a_1\in A_1} M_{x_2a_1,x_1a_2} \otimes E_{x_1,a_1}$$
and
$$\tilde{F}_{y_2,b_2} = \sum_{y_1\in Y_1}\sum_{b_1\in B_1} N_{y_2b_1,y_1b_2} \otimes F_{y_1,b_1};$$
it is straightforward to see that $(\tilde{E}_{x_2,a_2})_{a_2\in A_2}$ (resp. 
$(\tilde{F}_{y_2,b_2})_{b_2\in B_2}$) is a finite dimensionally acting POVM.
The proof in (ii) can now continue without further modification. 

(iii) is a direct consequence of (iv).


(v) follows from Remark \ref{r_desloc} and the fact that, if 
$\Phi_X : \cl D_{X_2}\to \cl D_{X_1}$,
$\Phi_Y : \cl D_{Y_2}\to \cl D_{Y_1}$,
$\Psi_A : \cl D_{A_1}\to \cl D_{A_2}$ and 
$\Psi_B : \cl D_{B_1}\to \cl D_{B_2}$ are channels then 
$$(\Phi_X\otimes \Phi_Y \otimes\Psi_A\otimes \Psi_B)[\cl E\otimes \cl F] 
= (\Phi_X\otimes \Psi_A)[\cl E]\otimes (\Phi_Y\otimes \Psi_B)[\cl F]$$ 
for all channels 
$\cl E : \cl D_{X_1}\to \cl D_{A_1}$ and $\cl F : \cl D_{Y_1}\to \cl D_{B_1}$.
\end{proof}

\begin{remark}\label{r_simu}
\rm The proof of Theorem \ref{th_pres} (i) shows that, in its notation, 
letting $\Gamma_{X_2A_1\to X_1A_2}$ be the NS correlation determined by the 
conditional probability distributions $\Gamma(x_1, a_2 | x_2,a_1)$, we have that 
$$\Gamma_{X_2A_1\to X_1A_2}[\cl E_{X_1\to A_1}] = 
\Gamma[\cl E]_{X_2\to A_2}.$$
\end{remark}

We next define the suitable version of the notion of a 
bicorrelation, defined in \cite{bhtt2}, in the strongly no-signalling context. 
Assume that $X_1 = X_2 =: X$, $Y_1 = Y_2 =: Y$, $A_1 = A_2 =: A$ and $B_1 = B_2 =: B$. 
We further assume that $X = A$ and $Y = B$.
A positive operator $P = (P_{xy,ab})_{xy,ab}\in \cl D_{XYAB}\otimes \cl B(H)$ will be called 
a \emph{magic bisquare} if it is an NS operator matrix, and 
the matrices $(P_{x,a})_{x,a}$ and $(P^{y,b})_{y,b}$ are quantum magic squares. 
A magic bisquare $P = (P_{x,y,a,b})_{xy,ab}$ is called \emph{dilatable} if 
there exist a Hilbert space $K$, an isometry $V : H\to K$ and quantum magic squares
$(E_{x,a})_{x,a\in X}$ and $(F_{y,b})_{y,b\in Y}$ on $K$, $x\in X$, $y\in Y$, such that 
$E_{x,a}F_{y,b} = F_{y,b}E_{x,a}$ and relations (\ref{eq_ExaFyb}) hold for all 
$x,a\in X$ and all $y,b\in Y$.
\emph{Quantum dilatable} and \emph{locally dilatable} magic bisquares are described similarly to 
quantum and local NS operator matrices, using quantum magic squares in the place of 
families of POVM's. 

An SNS correlation $\Gamma$ will be called an \emph{SNS bicorrelation} if $\Gamma$ is unital 
and $\Gamma^*$ is also an SNS correlation. 
An SNS bicorrelation $\Gamma$ over the quadruple $(XY, XY,XY,XY)$ is called 
\emph{quantum commuting} if there exist a Hilbert space $H$, 
dilatable magic bisquares
$P = (P_{x_2y_2,x_1y_1})_{x_2y_2,x_1y_1}$ 
and $Q = (Q_{a_1b_1,a_2b_2})_{a_1b_1,a_2b_2}$ on $H$ with mutually commuting entries, 
and a unit vector $\xi\in H$, such that equation (\ref{eq_qcsns}) holds. 
The classes of quantum SNS bicorrelations (denoted $\cl C_{\rm sq}^{\rm bi}$),
approximately quantum SNS bicorrelations (denoted $\cl C_{\rm sqa}^{\rm bi}$), and 
local SNS bicorrelations (denoted $\cl C_{\rm sloc}^{\rm bi}$) are described similarly to 
the their correlation counterparts, using magic bisquares of the appropriate type in the place of 
NS operator matrices of that type.
The following remark is straightforward from the definitions:

\begin{remark}\label{r_dualasw}
For a correlation type ${\rm t}\in \{{\rm loc}, {\rm q}, {\rm qa}, {\rm qc}, {\rm ns}\}$, if 
$\Gamma\in \cl C_{\rm st}^{\rm bi}$ then $\Gamma^*\in \cl C_{\rm st}^{\rm bi}$. 
\end{remark}


\begin{theorem}\label{c_pres}
Let $\Gamma$ be an SNS bicorrelation over the quadruple
$(XY, AB,$ $XY,AB)$ and $\cl E$ be an NS correlation 
over the quadruple $(X,Y,A,B)$. The following hold:
\begin{itemize}
\item[(i)] $\Gamma[\cl E] \in \cl C_{\rm ns}^{\rm bi}$;

\item[(ii)] if $\Gamma\in \cl C_{\rm sqc}^{\rm bi}$ and $\cl E\in \cl C_{\rm qc}^{\rm bi}$ then 
$\Gamma[\cl E] \in \cl C_{\rm qc}^{\rm bi}$;

\item[(iii)] if $\Gamma\in \cl C_{\rm sqa}^{\rm bi}$ and $\cl E\in \cl C_{\rm qa}^{\rm bi}$ then 
$\Gamma[\cl E] \in \cl C_{\rm qa}^{\rm bi}$;

\item[(iv)] if $\Gamma\in \cl C_{\rm sq}^{\rm bi}$ and 
$\cl E\in \cl C_{\rm q}^{\rm bi}$ then $\Gamma[\cl E] \in \cl C_{\rm q}^{\rm bi}$;

\item[(v)] if $\Gamma\in \cl C_{\rm sloc}^{\rm bi}$ and $\cl E\in \cl C_{\rm loc}^{\rm bi}$ then 
$\Gamma[\cl E] \in \cl C_{\rm loc}^{\rm bi}$.
\end{itemize}
\end{theorem}

\begin{proof}
(i)
The claim follows as in Theorem \ref{th_pres}, using the fact that in our case we have, in addition, that 
$\Gamma[\cl E]^* = \Gamma^*[\cl E^*]$. We verify that latter identity:
\begin{eqnarray*}
& & \Gamma[\cl E]^*(x_2,y_2|a_2,b_2)\\
& = &
\sum_{x_1y_1\in X_1Y_1} \sum_{a_1b_1\in A_1B_1}
\Gamma(x_1y_1, a_2b_2 | x_2y_2, a_1b_1)\cl E(a_1,b_1|x_1,y_1)\\
& = & 
\sum_{a_1b_1\in A_1B_1}\sum_{x_1y_1\in X_1Y_1} 
\Gamma^*(x_2y_2, a_1b_1 | x_1y_1, a_2b_2)\cl E^*(x_1,y_1 | a_1,b_1)\\
& = & 
\Gamma^*[\cl E^*](x_2,y_2|a_2,b_2), 
\end{eqnarray*}
for all $x_2\in X, y_2\in Y, a_2\in A, b_2\in B$. 

(ii) follows as Theorem \ref{th_pres} (ii), after noting that the POVM's defined in 
(\ref{eq_tildeExa}) and (\ref{eq_tildeFyb}) are quantum magic squares. 
The proofs of (iii), (iv) and (v) is now similar to those of Theorem \ref{th_pres} (iii), (iv) and (v). 
\end{proof}

\begin{definition}\label{d_gameh}
Let $E_i\subseteq X_iY_i\times A_iB_i$ be a non-local game, $i = 1,2$,
and 
${\rm t}\in \{{\rm loc}, {\rm q}, {\rm qa}, {\rm qc}, {\rm ns}\}$. 
We say that 
\begin{itemize}
\item[(i)] 
\emph{$E_1$ is ${\rm t}$-quasi-homomorphic to $E_2$} 
(and write $E_1\leadsto_{\rm st} E_2$) if the hypergraph quasi-homomorphism game 
$E_1 \hspace{-0.1cm} \leadsto\hspace{-0.1cm} E_2$ has a perfect strategy $\Gamma\in \cl C_{\rm st}$;
\item[(ii)] 
\emph{$E_1$ is ${\rm t}$-homomorphic to $E_2$} 
(and write $E_1\to_{\rm st} E_2$) if the hypergraph homomorphism game 
$E_1 \hspace{-0.1cm} \to\hspace{-0.1cm} E_2$ has a perfect strategy $\Gamma\in \cl C_{\rm st}^{\rm bi}$;
\item[(iii)] 
\emph{$E_1$ is ${\rm t}$-isomorphic to $E_2$} 
(and write $E_1\simeq_{\rm st} E_2$) if the hypergraph isomorphism game 
$E_1 \hspace{-0.1cm} \leftrightarrow\hspace{-0.1cm} E_2$ has a perfect strategy 
$\Gamma\in \cl C_{\rm st}^{\rm bi}$.
\end{itemize}
\end{definition}

\begin{corollary}\label{c_perfnlg}
Let $E_i\subseteq X_iY_i\times A_iB_i$ be a non-local game, $i = 1,2$,
and ${\rm t}\in \{{\rm loc}, {\rm q}, {\rm qa}, {\rm qc}, {\rm ns}\}$. 
If $E_1\to_{\rm st} E_2$ and $E_1$ has a perfect ${\rm t}$-strategy then $E_2$ has a perfect ${\rm t}$-strategy. 
\end{corollary}

\begin{proof}
The statement follows after an application of Theorem \ref{th_pres} and Proposition \ref{p_mapping}.
\end{proof}

\begin{theorem}\label{c_ordeqre3}
For ${\rm t}\in \{{\rm loc},{\rm q}, {\rm qa}, {\rm qc}, {\rm ns}\}$, the relations
$\to_{\rm st}$ and $\leadsto_{\rm st}$ 
(resp. the relation $\simeq_{\rm st}$) are quasi-orders (resp. is an equivalence relation) on the set of non-local games. 
\end{theorem}

\begin{proof}
Let $\Gamma_1$ and $\Gamma_2$ be elements of $\cl C_{\rm sns}$. 
The range of the variables in the summations below being understood from the context, we have 
\begin{eqnarray*}
& &
\sum\limits_{b_{3}}(\Gamma_{2}\ast\Gamma_{1})(x_{1}y_{1}, a_{3}b_{3}|x_{3}y_{3}, a_{1}b_{1})\\
& = &
\sum\limits_{b_{3}}\sum\limits_{x_{2},y_{2}}\sum\limits_{a_{2},b_{2}}\Gamma_{1}(x_{1}y_{1}, a_{2}b_{2}|x_{2}y_{2}, a_{1}b_{1})\Gamma_{2}(x_{2}y_{2}, a_{3}b_{3}|x_{3}y_{3}, a_{2}b_{2})\\
& = & 
\sum\limits_{x_{2},y_{2}}\sum\limits_{a_{2},b_{2}}\Gamma_{1}(x_{1}y_{1}, a_{2}b_{2}|x_{2}y_{2}, a_{1}b_{1})\Gamma_{2}(x_{2}y_{2}, a_{3}|x_{3}y_{3}, a_{2})\\
& = &
\sum\limits_{x_{2},y_{2}}\sum\limits_{a_{2}}\Gamma_{1}(x_{1}y_{1}, a_{2}|x_{2}y_{2}, a_{1})\Gamma_{2}(x_{2}y_{2}, a_{3}|x_{3}y_{3}, a_{2})\\
& = &
\sum\limits_{b_{3}}(\Gamma_{2}\ast \Gamma_{1})(x_{1}y_{1}, a_{3}b_{3}|x_{3}y_{3}, a_{1}b_{1}')
\end{eqnarray*}
\noindent for all $b_{1}, b_{1}'$. 
One verifies similarly the remaining three relations required in the definition of an SNS correlations; 
thus, $\Gamma_{2}\ast\Gamma_{1}\in \cl C_{\rm sns}$. 
Using Remark \ref{r_dualasw}, we now also see that, 
if $\Gamma_1$ and $\Gamma_2$ be elements of $\cl C_{\rm sns}^{\rm bi}$ then 
$\Gamma_{2}\ast\Gamma_{1}\in \cl C_{\rm sns}^{\rm bi}$. 
The cases where ${\rm t} = {\rm ns}$ now follows taking into account Theorem \ref{th_ordeqre}. 

Among the remaining correlation types, we only consider the case ${\rm t} = {\rm qc}$. 
Suppose that $\Gamma_1,\Gamma_2 \in \cl C_{\rm sqc}$.
Suppose that  
$P = (P_{x_2y_2,x_1y_1})_{x_2y_2,x_1y_1}$
and $Q = (Q_{a_1b_1,a_2b_2})_{a_1b_1,a_2b_2}$ are dilatable NS operator matrices 
on a Hilbert space $H$ with mutually commuting entries, 
and $\xi\in H$ is a unit vector, such that 
$$\Gamma_1(x_1y_1,a_2b_2 | x_2y_2,a_1b_1) = 
\left\langle 
P_{x_2y_2,x_1y_1} Q_{a_1b_1,a_2b_2}\xi,\xi\right\rangle$$
for all $x_i\in X_i$, $y_i\in Y_i$, $a_i\in A_i$ and $b_i\in B_i$, $i = 1,2$.
Similarly, write
$$\Gamma_2(x_2y_2,a_3b_3 | x_3y_3,a_2b_2) = 
\left\langle 
P_{x_3y_3,x_2y_2}' Q_{a_2b_2,a_3b_3}'\xi',\xi'\right\rangle,$$
for all $x_i\in X_i$, $y_i\in Y_i$, $a_i\in A_i$ and $b_i\in B_i$, $i = 2,3$,
in a Hilbert space $H'$. 
Similarly to the proof of Theorem \ref{th_comsimu}, set
$$P''_{{x_3y_3,x_1y_1}} = \sum\limits_{x_2\in X_2} \sum\limits_{y_2\in Y_2}
P_{x_2y_2,x_1y_1}\otimes P_{x_3y_3,x_2y_2}',$$
$$Q''_{a_1b_1,a_3b_3} = \sum\limits_{a_2\in A_2}  \sum\limits_{b_2\in B_2}
Q_{a_1b_1,a_2b_2} \otimes Q_{a_2b_2,a_3b_3}',$$
and $\xi'' = \xi\otimes\xi'$, in the Hilbert space $H'' := H \otimes H'$. 
It is straighforward that $(P''_{{x_3y_3,x_1y_1}})$ and $(Q''_{a_1b_1,a_3b_3})$ 
are dilatable NS operator matrices with commuting entries, which give rise to the correlation $\Gamma_2\ast \Gamma_1$
as in (\ref{eq_qcsns}). 

The case where $\Gamma_1$ and $\Gamma_2$ are (quantum commuting)
SNS bicorrelations follows from the previous paragraph and the fact that 
the matrices $(P''_{{x_3y_3,x_1y_1}})$ and $(Q''_{a_1b_1,a_3b_3})$ are 
magic bisquares, provided $(P_{{x_2y_2,x_1y_1}})$, $(Q_{a_1b_1,a_2b_2})$,
$(P'_{{x_3y_3,x_2y_2}})$ and $(Q'_{a_2b_2,a_3b_3})$ are such. 
\end{proof}

\begin{remark}\label{r_E1toE2}
\rm 
The statements in Theorem \ref{th_pres} are not reversible: e.g. not every affine map 
$\cl C_{\rm ns}(E_1)\mapsto \cl C_{\rm ns}(E_2)$ has the form 
$\cl E\to \Gamma[\cl E]$ for some $\Gamma\in \cl C_{\rm sqc}$ that fits $E_1\leftrightarrow E_2$.  
Indeed, let
$E_i = [1]\times Y_i \times [1] \times B_i$, $i = 1,2$. 
In this case, $\cl C_{\rm ns}(E_i) = \cl C(Y_i\times B_i)$, $i = 1,2$, and the statement
follows from Remark \ref{r_dGamma}. 
\end{remark}


\begin{remark}\label{r_localnonlo}
\rm 
Let $E_i\subseteq X_iY_i\times A_iB_i$ be a non-local game, $i = 1,2$.
It follows from Remark \ref{r_desloc} and Proposition \ref{p_lochy2} (i) 
that $E_1\to_{\rm loc} E_2$ if and only if 
there exist maps $f_X : X_2\to X_1$, $f_Y : Y_2\to Y_1$, 
$g_A : A_1\to A_2$ and $g_B : B_1\to B_2$ such that 
\begin{equation}\label{eq_eqinc12}
(f_X\times f_Y)^{-1} (E_1((a_1,b_1))) = E_2((g_A(a_1),g_B(b_1))), \ \ \ 
(a_1,b_1)\in A_1\times B_1.
\end{equation}
Similarly, if $X_1 = X_2$, $Y_1 = Y_2$, $A_1 = A_2$ and $B_1 = B_2$ then 
$E_1\simeq_{\rm loc} E_2$ if and only if the maps
$f_X$, $f_Y$, $g_A$ and  $g_B$ can be chosen bijective, that is, there exist 
permutations on the corresponding question and answer sets of the four players that 
transform the rules functions of the two games into one another. 
Finally, 
$E_1\leadsto_{\rm loc} E_2$ amounts to having inclusions in (\ref{eq_eqinc12}) instead of 
equalities. 
\end{remark}

\begin{remark}\label{r_sike}
\rm 
By Proposition \ref{p_sep}, there exists hypergraphs $E_i \subseteq X_i\times A_i$, $i = 1,2$, 
such that $E_1\leadsto_{\rm q} E_2$ via a quantum correlation $\Gamma$, 
but $E_1\not\leadsto_{\rm loc} E_2$. 
Let
$$\tilde{E}_i = \{((1,x),(1,a)) : (x,a)\in E_i\},$$
considered as a non-local game over $([1],X_i,[1],A_i)$, $i = 1,2$. 
It is straightforward that the correlation $\tilde{\Gamma}$, given by
$$\tilde{\Gamma}(1x_1, 1a_2|1x_2, 1a_1) := \Gamma(x_1,a_2|x_2,a_1),$$
is a quantum SNS correlation that realises a homomorphism $\tilde{E}_1\leadsto_{\rm q} \tilde{E}_2$. 
On the other hand, $\tilde{E}_1\not\leadsto_{\rm loc} \tilde{E}_2$ as 
the relation $\tilde{E}_1\leadsto_{\rm loc} \tilde{E}_2$ would force
$E_1\leadsto_{\rm loc} E_2$.
\end{remark}


\subsection{Optimal winning probabilities}\label{ss_optimal}

The results in this subsection can be viewed as a strengthening of Corollary \ref{c_perfnlg}.
We first recall the notion of a \emph{product game} from non-local game theory (see e.g. \cite{mptw}). 
Given non-local games 
$\Lambda_k\subseteq X_2^{(k)}Y_1^{(k)}\times X_1^{(k)}Y_2^{(k)}$, $k = 1,2$, 
their \emph{product} $\Lambda_1\otimes \Lambda_2$ is 
the non-local game on 
$X_2^{(1)}X_2^{(2)}Y_1^{(1)}Y_1^{(2)}\times X_1^{(1)}X_1^{(2)}Y_2^{(1)}Y_2^{(2)}$
arising from the product set $\Lambda_1 \times \Lambda_2$ after the natural reshuffling 
(see \cite[Section 3]{mptw}). 
Here, $(X_2^{(1)}X_2^{(2)}, Y_1^{(1)}Y_1^{(2)})$ (resp. $(X_1^{(1)}X_1^{(2)},Y_2^{(1)}Y_2^{(2)})$)
is the corresponding question (resp. answer) set. 

\begin{lemma}\label{l_prodsns}
Let $X_i^{(k)}$, $Y_i^{(k)}$, $A_i^{(k)}$ and $B_i^{(k)}$ be finite sets, $i,k = 1,2$, and 
$$\Gamma_k : \cl D_{X_2^{(k)}Y_2^{(k)} \times A_1^{(k)}B_1^{(k)}}
\to \cl D_{X_1^{(k)}Y_1^{(k)} \times A_2^{(k)}B_2^{(k)}}$$
be an SNS correlation, $k = 1,2$. 
\begin{itemize}
\item[(i)]
If $\Gamma_k\in \cl C_{\rm st}$, $k = 1,2$, then $\Gamma_1\otimes \Gamma_2\in \cl C_{\rm st}$; 

\item[(ii)]
If $\Gamma_k\in \cl C_{\rm st}^{\rm bi}$, $k = 1,2$, then $\Gamma_1\otimes \Gamma_2\in \cl C_{\rm st}^{\rm bi}$.
\end{itemize}
\end{lemma}

\begin{proof}
We give details for the case ${\rm t} = {\rm qc}$ only; the arguments
for the rest of the correlation types follow along similar lines. 

(i) 
Let $P^{(k)} = (P_{x_2y_2,x_1y_1}^{(k)})_{x_2y_2,x_1y_1}$ 
and $Q^{(k)} = (Q_{a_1b_1,a_2b_2}^{(k)})_{a_1b_1,a_2b_2}$ on $H_k$ with mutually commuting entries, 
and a unit vector $\xi_k\in H$, such that 
$$\Gamma^{(k)}(x_1y_1,a_2b_2 | x_2y_2,a_1b_1) = 
\left\langle 
P_{x_2y_2,x_1y_1}^{(k)} Q_{a_1b_1,a_2b_2}^{(k)}\xi,\xi\right\rangle, \ \ \ k = 1,2,$$
for all $x_i\in X_i^{(k)}$, $y_i\in Y_i^{(k)}$, $a_i\in A_i^{(k)}$ and $b_i\in B_i^{(k)}$, $i = 1,2$.
Letting $H = H_1\otimes H_2$, $\xi = \xi_1\otimes\xi_2$, 
$$P_{x_2y_2,x_1y_1} = P_{x_2y_2,x_1y_1}^{(1)}\otimes P_{x_2y_2,x_1y_1}^{(2)} 
\ \mbox{ and } \ 
Q_{a_1b_1,a_2b_2} = Q_{a_1b_1,a_2b_2}^{(1)} \otimes Q_{a_1b_1,a_2b_2}^{(2)},$$
we see that the operator matrices
$P = (P_{x_2y_2,x_1y_1})$ and $Q = (Q_{a_1b_1,a_2b_2})$
are dilatable and have mutually commuting entries. 
In addition, 
$\Gamma_1\otimes \Gamma_2$ arises as in equation (\ref{eq_qcsns}) from Definition \ref{d_snsc} via 
the quadruple $(H,\xi,P,Q)$.


(ii) is similar to (i); the case ${\rm t} = {\rm ns}$ 
uses the fact that $(\Gamma_1\otimes \Gamma_2)^* = \Gamma_1^*\otimes \Gamma_2^*$, the case 
${\rm t} = {\rm qc}$ -- the fact that the NS operator matrices $(P_{x_2y_2,x_1y_1})_{x_2y_2,x_1y_1}$ and 
$(Q_{a_1b_1,a_2b_2})_{a_1b_1,a_2b_2}$ defined in (i) are magic bisquares, and the rest of the cases 
are analogous. 
\end{proof}

Given a correlation type ${\rm t}$, the \emph{${\rm t}$-value} of a non-local game 
$E\subseteq V_2W_1\times V_1W_2$, equipped with a probability distribution $\pi$ on $V_2W_1$, 
is the parameter
$\omega_{\rm t}(E,\pi) = \omega_{\cl C_{\rm t}}(E,\pi)$ (see (\ref{eq_frakEpi}) in Subsection \ref{ss_values}). 
We set 
$$\bar{\omega}_{\rm t}(E,\pi) = \limsup_{n\in \bb{N}} \omega_{\cl C_{\rm t}}(E^{\otimes n},\pi^n)^{\frac{1}{n}};$$
the parameter $\bar{\omega}_{\rm t}(E,\pi)$ is the optimal ${\rm t}$-value of the game $E$ under
\emph{parallel repetition} \cite{raz}.

\begin{theorem}\label{th_equivval}
Let $E_i\subseteq X_iY_i\times A_iB_i$ be a non-local game, $\pi_2$ be a probability distribution on 
$X_2Y_2$, and $\pi_1 = \Gamma_{X_2Y_2\to X_1Y_1}(\pi_2)$.
\begin{itemize}
\item[(i)] 
If $E_1\to_{\rm st} E_2$ via $\Gamma\in \cl C_{\rm st}$ then 
$$\omega_{\rm t}(E_1,\pi_1)\leq \omega_{\rm t}(E_2,\pi_2) \ \ \mbox{ and } \ \ 
\bar{\omega}_{\rm t}(E_1,\pi_1)\leq \omega_{\rm t}(E_2,\pi_2);$$

\item[(ii)] 
Suppose that $X_1 = X_2 =: X$, $Y_1 = Y_2 =: Y$ and that
that $\pi$ is a $\Gamma$-stationary probability distribution on $XY$
If $E_1\simeq_{\rm st} E_2$ via $\Gamma\in \cl C_{\rm st}^{\rm bi}$
then 
$$\omega_{\rm t}(E_1,\pi) = \omega_{\rm t}(E_2,\pi) \ \ 
\mbox{ and } \ \ \bar{\omega}_{\rm t}(E_1,\pi) = \bar{\omega}_{\rm t}(E_2,\pi).$$
\end{itemize}
\end{theorem}


\begin{proof}
The statements in (i) and (ii) regarding the ${\rm t}$-values 
follow from Theorems \ref{th_pres} and \ref{c_pres}, and Proposition \ref{p_inval}. 
The statements regarding the parameters $\bar{\omega}_{\rm t}$ follow from 
Lemma \ref{l_prodsns} and Corollary \ref{p_invalcor}.
\end{proof}


\section{Representations of SNS correlations}\label{s_repSNS}

Our goal in this section is to obtain representations of the quantum commuting and the 
approximately quantum correlation types in terms of operator system tensor products; 
this is achieved in Subsection \ref{ss_reptensor}. In the next subsection, we develop the required 
multivariate tensor product theory in the operator system category, which 
extends the bivariate theory developed in \cite{kptt} and may be of interest in its own right. 
We will emphasise the differences with the bivariate theory, and will omit those arguments
that can be easily adapted from \cite{kptt}.


\subsection{Multivariate operator system tensor products}\label{s_multivar}

We fix throughout this subsection operator systems $\cl S_1,\dots,\cl S_k$.
Following \cite{kptt}, we define a \emph{tensor product} of $\cl S_1,\dots,\cl S_k$ (in this order)
to be an operator system structure $\sigma = (\Sigma_n)_{n\in \bb{N}}$
on the algebraic tensor product $\cl S_1\otimes\cdots \otimes\cl S_k$ satisfying the following properties:
\begin{itemize}
\item[(T1)] 
$(\otimes_{j=1}^k \cl S_j,(\Sigma_n)_{n\in \bb{N}},\otimes_{j=1}^k 1_{\cl S_j})$ is an operator system; 

\item[(T2)] 
if $P_j\in M_{n_j}(\cl S_j)^+$, $j\in [k]$, then $\otimes_{j=1}^k P_j\in \Sigma_{n_1\cdots n_k}$, and 

\item[(T3)] 
if $f_j : \cl S_j\to M_{l_j}$ is a unital completely positive map, $j\in [k]$, then 
$\otimes_{j=1}^k f_j$ is completely positive.
\end{itemize}
A \emph{tensor $k$-product} in the operator system category is an assignment $\sigma$
of a tensor product $\sigma\mbox{-}\otimes_{j=1}^k \cl S_j$ to each ordered $k$-tuple
$(\cl S_1,\dots,\cl S_k)$ of operator systems. 
We call $\sigma$
\emph{functorial} if, whenever $\cl T_j$ are operator systems and $\psi_j : \cl S_j \to \cl T_j$ are unital completely positive
maps, $j\in [k]$, we have that the tensor product map $\otimes_{j=1}^k \psi_j : \sigma\mbox{-}\otimes_{j=1}^k \cl S_j\to 
\sigma\mbox{-}\otimes_{j=1}^k \cl T_j$ is (unital and) completely positive. 

We will see that, as in the bivariate case, the operator system category admits
natural minimal, maximal and commuting tensor $k$-products. 
Equip the algebraic tensor product $\cl S_1\otimes\cdots \otimes\cl S_k$ 
with the involution, given by 
$$(u_1\otimes\cdots\otimes u_k)^* := u_1^*\otimes\cdots\otimes u_k^*,$$
and extend it to an involution on $M_n(\cl S_1\otimes\cdots \otimes\cl S_k)$ by letting 
$(u_{i,j})_{i,j}^* := (u_{j,i}^*)_{i,j}$. Write $M_n(\cl S_1\otimes\cdots \otimes\cl S_k)_h$ for the real vector 
space of all hermitian elements of $M_n(\cl S_1\otimes\cdots \otimes\cl S_k)$.


\subsubsection{The maximal tensor product}\label{sss_max}

By \cite[Theorem 5.5]{kptt}, the maximal tensor product between two operator systems is associative; 
we can thus unambiguously give a meaning to the multivariate maximal tensor product 

$$\max\hspace{-0.05cm}\mbox{-}\hspace{-0.09cm}\otimes_{j=1}^k \cl S_j 
:= \cl S_1\otimes_{\max}\cdots \otimes_{\max}\cl S_k.$$ 
We will need an explicit description of its positive cones, in the spirit of the one given in \cite[Section 5]{kptt}.
For notational simplicity, we restrict to the case $k = 3$.
For each $n \in \bb{N}$, let
\begin{gather*}
	D_{n}^{\rm max} 
	\hspace{-0.03cm}=\hspace{-0.03cm} 
	\{\alpha(P_{1}\hspace{-0.05cm}\otimes\hspace{-0.05cm} 
	P_{2}\hspace{-0.07cm}\otimes\hspace{-0.07cm} 
	P_{3})\alpha^{*}: \; P_{i} \hspace{-0.05cm}\in\hspace{-0.05cm} M_{n_{i}}(\cl{S}_{i})^{+}, 
	\; \alpha\hspace{-0.05cm} \in\hspace{-0.05cm} M_{n, n_{1}n_{2}n_{3}}, 
	n_{i} \hspace{-0.05cm} \in \hspace{-0.05cm} \bb{N}, 
	i \hspace{-0.05cm} \in \hspace{-0.05cm} [3]\}.
\end{gather*}

\begin{remark}\label{max_lem}
\rm 
Let $\cl{S}_{i}, i = 1, 2, 3$ be operator systems and 
$(D_{n})_{n=1}^{\infty}$ be a matrix ordering, 
with $D_{n} \subseteq M_{n}(\cl{S}_{1}\otimes \cl{S}_{2}\otimes \cl{S}_{3})$, 
satisfying property (T2) from the start of Subsection \ref{s_multivar}.
Then the compatibility condition implies that $D_{n}^{\rm max} \subseteq D_{n}$ for all $n \in \bb{N}$. 
\end{remark}

It is straightforward to verify that $1_{\cl{S}_{1}}\otimes 1_{\cl{S}_{2}}\otimes 1_{\cl{S}_{3}}$ is a matrix order unit
for the matrix ordering $(D_{n}^{\rm max})_{n=1}^{\infty}$; 
we let $(C_{n}^{\rm max})_{n=1}^{\infty}$ be its Archimedeanisation \cite[Section 3.1]{ptt}.

\begin{proposition}\label{p_coinc}
Let $\cl S_{i}, i = 1, 2, 3$, be operator systems. Then 
$$C_{n}^{\rm max} = M_n\left((\cl S_{1}\otimes_{\rm max} \cl S_{2}) \otimes_{\rm max} \cl{S}_{3}\right)^+, 
\ \ \ n \in \bb{N}.$$
\end{proposition}

\begin{proof}
Set $C_n = M_n\left((\cl S_{1}\otimes_{\rm max} \cl S_{2}) \otimes_{\rm max} \cl{S}_{3}\right)^+$, $n\in \bb{N}$. 
By Remark \ref{max_lem}, $C_{n}^{\rm max} \subseteq C_{n}$. 
Let $Q \in M_{n}((\cl{S}_{1}\otimes_{\rm max} \cl{S}_{2}) \otimes_{\rm max} \cl{S}_{3})^{+}$;
without loss of generality, we may assume that 
$Q = \alpha(Q_{1}\otimes R_{1})\alpha^{*}$ for $Q_{1} \in M_{k}(\cl{S}_{1}\otimes_{\rm max} \cl{S}_{2})^{+}$, 
$R_{1} \in M_{m}(\cl{S}_{3})^{+}$ and $\alpha\in M_{n,km}$. 
As $Q_{1} \in M_{k}(\cl{S}_{1}\otimes_{\rm max} \cl{S}_{2})^{+}$, again without loss of generality we may assume 
that 
$Q_{1} = \beta(P_{1}\otimes P_{2})\beta^{*}$, where $P_{1} \in M_{\ell}(\cl{S}_{1})^{+}, P_{2} \in M_{p}(\cl{S}_{2})^{+}$, 
and $\beta \in M_{k, \ell p}$. 
Write 
$\beta = \sum\limits_{i=1}^{r}C_{i}\otimes D_{i}$ for some 
$C_{i} \in M_{n_{1}, \ell}$ and $D_{i} \in M_{n_{2}, p}$ with $n_{1}n_{2} = k$. We see that 
\begin{gather*}
	\beta(P_{1}\otimes P_{2})\beta^{*} = \sum\limits_{i, j=1}^{r}C_{i}P_{1}C_{j}^{*}\otimes D_{i}P_{2}D_{j}^{*}.
\end{gather*}
Let $P_{1}' = (C_{i}P_{1}C_{j}^{*})_{i, j}^{r}$ and $P_{2}' = (D_{i}P_{2}D_{j}^{*})_{i, j}^{r}$; 
thus, $P_{1}' \in M_{rn_{1}}(\cl{S}_{1})^{+}$ and $P_{2}' \in M_{rn_{2}}(\cl{S}_{2})^{+}$, and
we view $P_{1}'\otimes P_{2}'$ as a $r^{2}k\times r^{2}k$ matrix. 

Let $\beta' = (\beta_{1}, \hdots, \beta_{r^{2}})$, where $\beta_{j} \in M_{k}$ for $1 \leq j \leq r^{2}$ and
\begin{gather*}
	\beta_{1} = \beta_{r+2} = \beta_{2r+3} = \cdots = \beta_{r^{2}} = I_{k},
\end{gather*}
\noindent with $\beta_{j} = 0$ otherwise. 
We have that $\beta'\in M_{k,kr^2}$; thus, 
$\gamma := \beta'\otimes I_{m} \in M_{km, r^{2}km}$, and
\begin{gather*}
	\gamma(P_{1}'\otimes P_{2}' \otimes R_{1})\gamma^{*} 
	= (\beta'\otimes I_{m})((P_{1}'\otimes P_{2}') \otimes R_{1})(\beta' \otimes I_{m})^{*} 
\\ = \bigg(\sum\limits_{i, j=1}^{r}C_{i}P_{1}C_{j}^{*}\otimes D_{i}P_{2}D_{j}^{*}\bigg) \otimes R_{1} = \beta(P_{1}\otimes P_{2})\beta^{*}\otimes R_{1}.
\end{gather*}
\noindent Thus, 
\begin{eqnarray*}
	\alpha(\beta(P_{1}\otimes P_{2})\beta^{*}\otimes R_{1})\alpha^{*} 
	& = & 
	\alpha(\gamma(P_{1}'\otimes P_{2}'\otimes R_{1})\gamma^{*})\alpha^{*}\\ 
	& = & (\alpha \gamma)(P_{1}'\otimes P_{2}' \otimes R_{1})(\alpha \gamma)^{*},
\end{eqnarray*}
\noindent where $\alpha \gamma \in M_{n, r^{3}m}$. 
This shows $Q \in D_{n}^{\rm max}$, and so $C_{n} = C_{n}^{\rm max}$ for each $n \in \bb{N}$. 
\end{proof}

\begin{corollary}\label{c_assocmax}
We have that $M_n(\max\hspace{-0.05cm}\mbox{-}\hspace{-0.09cm}\otimes_{j=1}^k \cl S_j)^+ = C_n^{\max}$, 
$n\in \bb{N}$.
\end{corollary}

\begin{proposition}\label{p_cstarsame}
If $\cl A_j$, $j\in [k]$, are unital C*-algebras then ${\max}$-$\otimes_{j=1}^k \cl A_j$ is completely order isomorphic to the image of $\otimes_{j=1}^{k} \cl A_{j}$ inside the maximal C*-algebra tensor product
${\rm C^*max}$-$\otimes_{j=1}^k \cl A_j$.
\end{proposition}

\begin{proof}
The proof relies on the ideas in the proof of \cite[Theorem 5.12]{kptt}; 
we only consider the case $k = 3$. 
Let $\cl{C} = \cl{A}_{1}\otimes_{\rm C^*max}\cl{A}_{2}\otimes_{\rm C^*max}\cl{A}_{3}$ denote the maximal ${\rm C^*}$-algebraic tensor product of $\cl{A}_{1}, \cl{A}_{2}$ and $\cl{A}_{3}$. The faithful inclusion of $\cl{A}_{1}\otimes \cl{A}_{2}\otimes \cl{A}_{3} \subseteq \cl{C}$ endows $\cl{A}_{1}\otimes \cl{A}_{2}\otimes \cl{A}_{3}$ with an operator system structure; 
let $\tau$-$\otimes_{j=1}^{3}\cl{A}_{j}$ denote this operator system, 
and let $C_{n}^{\rm max}$, $n\in \bb{N}$, 
be the matricial cones of ${\rm max}$-$\otimes_{j=1}^{3}\cl{A}_{j}$. 
For $n \in \bb{N}$, let $D_{n} = M_{n}(\tau$-$\otimes_{j=1}^{3}\cl{A}_{j})^{+}$. 
By maximality, $C_{n}^{\rm max} \subseteq D_{n}$.
For the converse inclusion, note that, if 
$X = \sum\limits_{i=1}^{k}a_{i}^{(1)}\otimes a_{i}^{(2)}\otimes a_{i}^{(3)}$ where $a_{i}^{(1)} \in M_{n}(\cl{A}_{1}), a_{i}^{(2)} \in \cl{A}_{2}$ and $a_{i}^{(3)} \in \cl{A}_{3}$, then
\begin{gather*}
	XX^{*} = \sum\limits_{i, j=1}^{k}a_{i}^{(1)}a_{j}^{(1)*}\otimes a_{i}^{(2)}a_{j}^{(2)*}\otimes a_{i}^{(3)}a_{j}^{(3)*}.
\end{gather*}
\noindent Let $A^{(1)} = (a_{i}^{(1)}a_{j}^{(1)*})_{i, j}$,  $A^{(2)} = (a_{i}^{(2)}a_{j}^{(2)*})_{i, j}$ and $A^{(3)} = (a_{i}^{(3)}a_{j}^{(3)*})_{i, j}$; then $A^{(1)} \in M_{k}(M_{n}(\cl{A}_{1}))^{+}, A^{(2)} \in M_{k}(\cl{A}_{2})^{+}$ and $A^{(3)} \in M_{k}(\cl{A}_{3})^{+}$. 
Similarly to the proof of Proposition \ref{p_coinc}, there exists 
$\beta \in M_{n, nk^{3}}$ such that 
\begin{gather*}
	\beta(A^{(1)}\otimes A^{(2)}\otimes A^{(3)})\beta^{*} =  \sum\limits_{i, j=1}^{k}a_{i}^{(1)}a_{j}^{(1)*}\otimes a_{i}^{(2)}a_{j}^{(2)*}\otimes a_{i}^{(3)}a_{j}^{(3)*} = XX^{*}, 
\end{gather*}
\noindent and hence $XX^{*} \in C_{n}^{\rm max}$. Since $D_n$ is generated, as a closed convex set, by 
elements of the form $X^*X$, we have that $D_n\subseteq  C_{n}^{\rm max}$. 
The proof is complete. 
\end{proof}


\subsubsection{The minimal tensor product}\label{sss_min}

According to \cite[Theorem 4.6]{kptt}, the bivariate minimal tensor product
is associative and hence one can unambiguously define 
the operator system 
$$\min\hspace{-0.05cm}\mbox{-}\hspace{-0.09cm}\otimes_{j=1}^k \cl S_j 
:= \cl S_1\otimes_{\min}\cdots \otimes_{\min}\cl S_k.$$

\begin{proposition}\label{p_minexp}
Let $n\in \bb{N}$. Then 
$$\hspace{-1cm} M_n(\min\hspace{-0.05cm}\mbox{-}\hspace{-0.09cm}\otimes_{j=1}^k \cl S_j)^+ 
= \{u\in M_n(\otimes_{j=1}^k \cl S_j) : (\otimes_{j=1}^k f_j)^{(n)}(u)\in M_{nn_1\dots n_k}^+$$
\vspace{-0.4cm}
$$\hspace{3.5cm} \mbox{for all u.c.p. maps } f_j : \cl S_j\to M_{n_j}, n_j\in \bb{N}, j\in [k]\}.$$
\end{proposition}

\begin{proof}
We consider the case $k = 3$ only. For $n \in \bb{N}$, let
\begin{gather*}
	C_{n} = \{u \in M_{n}(\otimes_{j=1}^{3}\cl{S}_{j}): \; (\otimes_{j=1}^{3}f_{j})^{(n)}(u) \in M_{nn_{1}n_{2}n_{3}}^{+} \\\text{ for all u.c.p. maps } f_{j}: \cl{S}_{j} \rightarrow M_{n_{j}}, \; n_{j} \in \bb{N}, \; j = 1,2,3\}.
\end{gather*}
We will show that $C_{n} = M_{n}((\cl{S}_{1}\otimes_{\rm min} \cl{S}_{2}) \otimes_{\rm min} \cl{S}_{3})^{+}$. 
We assume, without loss of generality, that $\cl{S}_{j}\subseteq \cl{B}(\cl{H}_{j})$, $j = 1,2,3$. 
Let $P \in C_{n}$ and 
suppose that 
$P = \sum\limits_{i=1}^{\ell}X_{i}\otimes y_{i}\otimes z_{i}$, where $X_{i} \in M_{n}(\cl{S}_{1}), 
y_{i} \in \cl{S}_{2}$ and $z_{i} \in \cl{S}_{3}$ for $i \in [l]$. 
Furthermore, let $\psi = \sum\limits_{s=1}^{k}\xi_{s}\otimes \eta_{s}\otimes \beta_{s}$ where $\xi_{s} \in \cl{H}_{1}^{(n)}, \eta_{s} \in \cl{H}_{2}$ and $\beta_{s} \in \cl{H}_{3}$, $s \in [k]$. 
Let $\Phi_{1}: M_{n}(\cl{S}_{1}) \rightarrow M_{k}$, 
$\Phi_{2}: \cl{S}_{2} \rightarrow M_{k}$ and 
$\Phi_{3}: \cl{S}_{3} \rightarrow M_{k}$ be the maps given by
\begin{gather*}
	\Phi_{1}(X) = (\langle X\xi_{s}, \xi_{t}\rangle)_{s, t}, \;\;\;\; \Phi_{2}(y) = (\langle y\eta_{s}, \eta_{t}\rangle)_{s, t}, \;\;\;\; \Phi_{3}(z) = (\langle z\beta_{s}, \beta_{t}\rangle)_{s, t}.
\end{gather*}
\noindent It is clear that 
$\Phi_{j}$ is completely positive, $j = 1,2,3$. 
Note that $(f_{1}\otimes f_{2}\otimes f_{3})^{(n)} = f_{1}^{(n)}\otimes f_{2}\otimes f_{3}$ for all 
linear maps $f_{j}: \cl{S}_{j} \rightarrow M_{m}$, $j = 1,2,3$. 
Thus, 
$(f_{1}^{(n)}\otimes f_{2}\otimes f_{3})(P) \in M_{nk^{3}}^{+}$ if $f_{j}: \cl{S}_{j}\rightarrow M_{k}$ 
is unital and completely positive, $i = 1, 2, 3$. 
By \cite[Lemma 4.2]{kptt}, $(\Phi_{1}\otimes \Phi_{2}\otimes \Phi_{3})(P) \in M_{nk^{3}}^{+}$. 
Let $e = (e_{1}\otimes e_{1}, \hdots, e_{k}\otimes e_{k}) \in \bb{C}^{k^{3}}$ where $\{e_{i}\}_{i=1}^{k}$ is the standard orthonormal basis of $\bb{C}^{k}$. We have 
\begin{eqnarray*}
\langle P\psi, \psi\rangle 
& = &
\sum\limits_{i=1}^{\ell}\sum\limits_{s, t=1}^{k}\langle X_{i}\xi_{s}, \xi_{t}\rangle\langle y_{i}\eta_{s}, \eta_{t}\rangle\langle z_{i}\beta_{s}, \beta_{t}\rangle\\
& = &
\sum\limits_{i=1}^{\ell}\langle (\Phi_{1}(X_{i})\otimes \Phi_{2}(y_{i})\otimes \Phi_{3}(z_{i}))e, e\rangle \\
& = &
\langle (\Phi_{1}\otimes \Phi_{2}\otimes \Phi_{3})(P)e, e\rangle.
\end{eqnarray*}
\noindent It follows that $P \in \cl{B}((\otimes_{j=1}^{3}\cl{H}_{j})^{n})^{+}$ and hence 
$C_{n} \subseteq D_{n}$. 
The reverse inclusion follows easily using the functoriality and injectivity of the minimal 
${\rm C^*}$-algebraic tensor product; the details are omitted. 
\end{proof}

\begin{remark}\label{r_funcmm}
\rm 
By their definition, and the fact that 
the bivariate minimal and maximal operator system tensor products are functorial 
\cite[Theorems 4.6 and 5.5]{kptt}, 
the minimal and the maximal multivariate operator system tensor products 
are functorial. 
Similarly, the minimal multivariate operator system tensor product is injective
(see \cite[Section 3]{kptt}). 
\end{remark}


\subsubsection{The commuting tensor product}

Let $H$ be a Hilbert space and 
$\phi_j : \cl S_j\to \cl B(H)$ be a completely positive map, $j \in [k]$. 
We call the family $(\phi_j)_{j=1}^k$ \emph{commuting} if 
$\phi_i$ and $\phi_j$ have mutually commuting ranges whenever $i\neq j$. 
Let $\Pi_{j=1}^k \phi_i : \otimes_{j=1}^k \cl S_j\to \cl B(H)$ be the linear map, given by 
$$\left(\Pi_{j=1}^k \phi_j\right)(\otimes_{j=1}^k u_j) = \Pi_{j=1}^k \phi_j(u_j), \ \ \ u_j\in \cl S_j, j \in [k].$$
For $n\in \bb{N}$, let 
$$\hspace{-2cm} C_n^{\rm c} = \{u\in M_n(\otimes_{j=1}^k \cl S_j) : (\Pi_{j=1}^k \phi_j)^{(n)}(u) \in M_n(\cl B(H))^+,$$ 
\vspace{-0.4cm}
$$ \hspace{1cm} \mbox{ for all commuting families } (\phi_j)_{j=1}^k
\mbox{ and all Hilbert spaces } H\}.$$

\begin{lemma}\label{th_commult}
Let $\cl S_1,\dots,\cl S_k$ be operator systems and $\cl S = \otimes_{j=1}^k \cl S_i$.
Then $C_n^{\rm c}$ is a cone in $M_n(\cl S)$ and $\cl S$, equipped with the family $(C_n^{\rm c})_{n\in \bb{N}}$ and 
the element $1 := 1_{\cl S_1}\otimes\cdots \otimes 1_{\cl S_k}$ as an 
Archimedean matrix order unit, is an operator system. 
\end{lemma}

\begin{proof}
It is straightforward to verify that $C_{n}^{\rm c}$ is a cone, $n\in \bb{N}$. 
Let $n, m \in \bb{N}$, $\alpha \in M_{n, m}$ and $u \in C_{m}^{\rm c}$. 
We have 
\begin{gather*}
	\bigg(\Pi_{j=1}^{k}\phi_{j}\bigg)^{(n)}(\alpha u\alpha^{*}) 
	= \alpha\left[\bigg(\Pi_{j=1}^{k}\phi_{j}\bigg)^{(m)}\hspace{-0.1cm}(u)\right]\alpha^{*}
	\in M_{n}(\cl B(H))^{+};
\end{gather*}
thus, the family $(C_{n}^{\rm c})_{n=1}^{\infty}$ is compatible. 

Let $\phi_{j} : \cl{S}_{j} \to M_{\ell_j}$ be a unital completely positive map, $j \in [k]$, 
and $\tilde{\phi_{j}}: \cl{S}_{j} \rightarrow M_{\ell_{1}\cdots\ell_{k}}$ be the map, given by 
\begin{gather*}
\tilde{\phi_{j}}(u)
= 1_{\cl{S}_{1}}\otimes \cdots \otimes \phi_{j}(u) \otimes \cdots \otimes 1_{\cl{S}_{k}}, \ \ \ u\in \cl S_j.
\end{gather*}
\noindent Then $\tilde{\phi_{j}}$ is completely positive for each $j \in [k]$, and $(\tilde{\phi_{j}})_{j=1}^{k}$ is a commuting family. Therefore
\begin{gather*}
	 (\phi_{1}\otimes \cdots \otimes \phi_{k})^{(n)}(u) = \bigg(\Pi_{j=1}^{k}\tilde{\phi_{j}}\bigg)^{(n)}\hspace{-0.1cm}(u) \geq 0, 
	\ \ \ u \in C_{n}^{\rm c}.
\end{gather*}
By Proposition \ref{p_minexp}, 
$C_{n}^{\rm c} \subseteq M_n(\min\hspace{-0.05cm}\mbox{-}\hspace{-0.09cm}\otimes_{j=1}^k \cl S_j)^+$, $n \in \bb{N}$. 
In particular, $C_{n}^{\rm c} \cap (-C_{n}^{\rm c}) = \{0\}$ for each $n \in \bb{N}$. 
On the other hand, by Remark \ref{max_lem}, 
$M_n(\max\hspace{-0.05cm}\mbox{-}\hspace{-0.09cm}\otimes_{j=1}^k \cl S_j)^+ \subseteq  C_n^{\rm c}$, $n \in \bb{N}$.
In particular, 
$1_{\cl{S}_{1}}\otimes \cdots \otimes 1_{\cl{S}_{k}}$ a matrix order unit; 
it is straightforward to verify that it is Archimedean. 
\end{proof}

We denote the operator system $\left(\cl S, (C_n^{\rm c})_{n\in \bb{N}}, 1\right)$ from Theorem \ref{th_commult} by 
${\rm c}$-$\otimes_{j=1}^k \cl S_j$.

\begin{theorem}\label{th_comte}
The map $(\cl S_1\dots,\cl S_k)\to 
{\rm c}$-$\otimes_{j=1}^k \cl S_j$ is a functorial operator system tensor product. 
\end{theorem}

\begin{proof}
Lemma \ref{th_commult} and its proof establish properties (T1) and (T3), while Remark \ref{max_lem} implies 
property (T2). 
To show functoriality, let $\cl S_j$ and $\cl T_j$ be operator systems, $\phi_j : \cl S_j\to \cl T_j$ be
completely positive maps, $j\in [k]$, and $u\in M_n({\rm c}$-$\otimes_{j=1}^k \cl S_j)^+$. 
If $\psi_j : \cl T_j\to \cl B(H)$, $j\in [k]$, are completely positive maps such that the family $(\psi_j)_{j=1}^k$ is 
commuting then the family $(\psi_j\circ\phi_j)_{j=1}^k$ is commuting, and hence 
$$(\Pi_{j=1}^{k} \psi_j)((\otimes_{j=1}^k \phi_j)(u)) = (\Pi_{j=1}^{k} (\psi_j \circ \phi_j))(u) \in M_n(\cl B(H))^+.$$
It follows that 
$\otimes_{j=1}^k \phi_j : {\rm c}$-$\otimes_{j=1}^k \cl S_j\to {\rm c}$-$\otimes_{j=1}^k \cl T_j$ is 
completely positive.
\end{proof}

\begin{proposition}\label{c_star_equiv}
If $\cl A_{1}, \hdots, \cl A_{k}$ are unital $C^{*}$-algebras, then $\cl A_{1} \otimes_{\rm max} \cdots \otimes_{\rm max} \cl A_{k} = \cl{A}_{1} \otimes_{\rm c} \cdots \otimes_{\rm c} \cl{A}_{k}$. 
\end{proposition}
\begin{proof}
By Remark \ref{max_lem}, 
$M_{n}({\rm max}$-$\otimes_{j=1}^{k}\cl A_{j})^{+} \subseteq M_{n}({\rm c}$-$\otimes_{j=1}^{k}\cl A_{j})^{+}$, $n \in \bb{N}$.
For the reverse inclusion, fix $u \in M_{n}({\rm c}$-$\otimes_{j=1}^{k} \cl A_{j})^{+}$. 
It suffices to show that $\phi^{(n)}(u) \geq 0$ for all unital completely positive maps
$\phi: \cl{A}_{1} \otimes_{\rm max}\cdots \otimes_{\rm max} \cl A_{k}\rightarrow \cl B(H)$. 
By Proposition \ref{p_cstarsame}, and an application of 
Stinespring's Theorem, we may further assume $\phi$ is a $*$-homomorphism.
By associativity of the maximal tensor product of operator systems
(see \cite[Theorem 5.5]{kptt}) and the universal property of the maximal tensor product of 
C*-algebras, we write 
$\phi = \Pi_{j=1}^{k}\phi_{j}$, where $\phi_{j}: \cl A_{j}\rightarrow \cl B(H)$ are  $*$-homomorphisms with commuting ranges. By assumption, we have
\begin{gather*}
	\phi^{(n)}(u) = \bigg(\Pi_{j=1}^{k}\phi_{j}\bigg)^{(n)}(u) \geq 0,
\end{gather*}
\noindent and so $u \in M_{n}({\rm max}$-$\otimes_{j=1}^{k}\cl A_{j})^{+}$. 
The proof is complete. 
\end{proof}

For the next theorem, recall that the \emph{coproduct} $\cl S_1\oplus_1\cl S_2$ of two operator systems
is the (unique, up to a unital complete order isomorphism) operator system, containing $\cl S_1$ and $\cl S_2$ 
as operator subsystems, satisfying the following universal property: for every 
operator system $\cl R$ and unital completely positive maps $\phi_i : \cl S_i\to \cl R$, $i = 1,2$, there exists 
a unique unital completely positive map $\phi : \cl S_1\oplus_1\cl S_2\to \cl R$ such that 
$\phi|_{\cl S_i} = \phi_i$, $i = 1,2$ (see \cite{fkpt_NYJ, kavruk}). 
We denote by $\cl A\ast_1\cl B$ the C*-free product of the unital C*-algebras $\cl A$ and $\cl B$, 
amalgamated over their units. 
The next statement extends \cite[Lemma 2.8]{pt} to the multivariate case. 

\begin{theorem}\label{l_kas}
Let $\cl A_i^{(j)}$ be a unital C*-algebra, $i \in [n_j]$, $j \in [k]$, and set 
$\cl S_j = \cl A_1^{(j)}\oplus_1\cdots\oplus_1\cl A_{n_j}^{(j)}$ and 
$\cl A_j = \cl A_1^{(j)}\ast_1\cdots\ast_1\cl A_{n_j}^{(j)}$, $j\in [k]$. 
Then 
$${\rm c}\mbox{-}\hspace{-0.07cm}\otimes_{j=1}^k \hspace{-0.07cm}\cl S_j 
\subseteq_{\rm c.o.i.} \cl A_1 \otimes_{\max} \cdots \otimes_{\max}\cl A_k.$$
\end{theorem}

\begin{proof}
Set $\cl S = {\rm c}\mbox{-}\hspace{-0.07cm}\otimes_{j=1}^k \hspace{-0.07cm}\cl S_j$ for brevity. 
By Theorem \ref{th_comte} and Proposition \ref{c_star_equiv}, the inclusion map 
$\iota : {\rm c}\mbox{-}\hspace{-0.07cm}\otimes_{j=1}^k \hspace{-0.07cm}\cl S_i
\to {\max}\mbox{-}\hspace{-0.07cm}\otimes_{j=1}^k \hspace{-0.07cm}\cl A_i$
is completely positive. 
Let $u \in M_{n}(\cl{S}) \cap M_{n}({\rm max}$-$\otimes_{j=1}^{k}\cl{A}_{j})^{+}$; 
we will show that $u \in M_{n}(\cl{S})^{+}$. 
For $j\in [k]$, let $\phi_{j}: \cl{S}_{j} \rightarrow \cl B(H)$ 
be a completely positive map, $j\in [k]$, such that the family $(\phi_j)_{j=1}^k$ is commuting. 
Assume first that $\phi_{j}$ is unital for each $j\in [k]$.  
Using \cite[Theorem 5.2]{fkpt_NYJ}, 
identify $\cl{S}_{j}$ with the linear span of $\cl{A}_{1}^{(j)}, \hdots, \cl{A}_{n_{j}}^{(j)}$ 
inside $\cl{A}_{j}$, $j \in [k]$. Each map $\phi_{j}$ is determined by a family $(\phi_{\ell}^{j})_{\ell=1}^{n_{j}}$ of maps 
$\phi_{\ell}^{j}: \cl{A}_{\ell}^{(j)}\rightarrow \cl B(H)$, which are unital and completely positive, in that
\begin{gather*}
	\phi_{j}(a_{1}^{j}+\cdots +a_{n_{j}}^{j}) = \phi_{1}^{j}(a_{1}^{j})+\cdots+\phi_{n_{j}}^{j}(a_{n_{j}}^{j}),
\end{gather*}
\noindent $a_{\ell}^{j} \in \cl{A}_{\ell}^{(j)}$, $\ell = 1, \hdots, n_{j}$. 
As shown in \cite{boca}, if $j\in [k]$, there exists a unital completely positive map 
$\tilde{\phi_{j}}: \cl{A}_{j} \rightarrow \cl B(H)$, given by
\begin{gather*}
	\tilde{\phi_{j}}(a_{\ell_{1}}^{j}\cdots a_{\ell_{m}}^{j}) = \phi_{\ell_{1}}^{j}(a_{\ell_{1}}^{j})\cdots\phi_{\ell_{m}}^{j}(a_{\ell_{m}}^{j}),
\end{gather*}
\noindent where $a_{\ell_{p}}^{j} \in \cl{A}_{\ell_{p}}^{j}$ for each $p$, and $\ell_{1} \neq \cdots \neq \ell_{m}$.
Since ${\rm ran}\phi_{\ell}^{j} \subseteq {\rm ran}\phi_{j}$, $j\in [k]$, 
we have that $\phi_{\ell}^{i}(\cl{A}_{\ell}^{(i)})$ and $\phi_{m}^{j}(\cl{A}_{m}^{(j)})$ mutually commute
whenever $i\neq j$. 
As $\cl{A}_{j}$ is generated, as a C*-algebra, by the operator subsystems $\cl A_i^{(j)}$, $i\in [n_j]$, 
we have that $\tilde{\phi_{i}}(\cl{A}_{i})$ and $\tilde{\phi_{j}}(\cl{A}_{j})$ commute whenever $i \neq j$. 
By Proposition \ref{c_star_equiv}, the map $\Pi_{j=1}^k \tilde{\phi_{j}}$ is completely positive
on $\cl A_1\otimes_{\max}\cdots\otimes_{\max}\cl A_k$, and thus
\begin{gather*}
	\bigg(\Pi_{j=1}^{k}\phi_{j}\bigg)^{(n)}(u) = (\tilde{\phi_{1}}\cdots\tilde{\phi_{k}})^{(n)}(u) \geq 0,
\end{gather*}
\noindent showing $u \in M_{n}(\cl {S})^{+}$. 

Now relax the assumption on the unitality of the maps $\phi_{j}$. 
Without loss of generality, assume that $\|\phi_{j}\|\leq 1$, so that the operator 
$T_{j} := \phi_{j}(1_{\cl{S}_{j}})$ is a positive contraction. 
Assume first that $T_{j}$ is invertible for each $j\in [k]$. 
Let $\cl M_{j}$ be the von Neumann algebra generated by $\phi_{j}(\cl{S}_{j})$; 
as $T_{i} \in {\rm ran}(\phi_{i})$, we have that $T_{i} \in \cl M_{j}'$ if $i\neq j$. 
We conclude that 
$T_{i}^{-1}, T_{i}^{1/2}$ and $T_{i}^{-1/2}$ are in $\cl M_{j}'$, whenever $i\neq j$. 
Define the maps $\hat{\phi_{j}}: \cl{S}_{j}\rightarrow \cl B(H)$ via
\begin{gather*}
	\hat{\phi_{j}}(u) = T_{j}^{-1/2}\phi_{j}(u)T_{j}^{-1/2}, \ \ \ u\in \cl S_j, \, j\in [k].
\end{gather*}
As $T_{j}^{-1/2}$ is positive, 
$\hat{\phi_{j}}$ is a (unital) completely positive map, $j\in [k]$. 
We have that $T_{i} \in \cl M_{j}'$ if $i \neq j$; in particular, $T_{i}T_{j} = T_{j}T_{i}$ when $i \neq j$, and hence 
\begin{gather*}
	T_{i}^{-1/2}T_{j}^{-1/2} = T_{j}^{-1/2}T_{i}^{-1/2} \ \ \mbox{ and } \ \ T_{i}^{1/2}T_{j}^{1/2} = T_{j}^{1/2}T_{i}^{1/2}, \ \ \ \ i\neq j.
\end{gather*}
If $u \in \cl{S}_{i}$, $v \in \cl{S}_{j}$ and $i\neq j$, then 
\begin{eqnarray*}
	\hat{\phi_{i}}(u)\hat{\phi_{j}}(v) 
	& = &  
	T_{i}^{-1/2}\phi_{i}(u)T_{i}^{-1/2}T_{j}^{-1/2}\phi_{j}(v)T_{j}^{-1/2}\\ 
	& = &  
	T_{j}^{-1/2}T_{i}^{-1/2}\phi_{i}(u)T_{i}^{-1/2}\phi_{j}(v)T_{j}^{-1/2}\\
	& = &  
	T_{j}^{-1/2}\phi_{j}(v)T_{i}^{-1/2}\phi_{i}(u)T_{i}^{-1/2}T_{j}^{-1/2}\\
	& = &  
	T_{j}^{-1/2}\phi_{j}(v)T_{j}^{-1/2}T_{i}^{-1/2}\phi_{i}(u)T_{i}^{-1/2}
	= 
        \hat{\phi_{j}}(v)\hat{\phi_{i}}(u).
\end{eqnarray*}
Thus, the family $(\hat{\phi_{j}})_{j=1}^k$ is commuting. 
Using the previous paragraph, 
we see that $\Pi_{j=1}^k \hat{\phi}_{j}(u) \geq 0$.
Let $T = T_{1}^{1/2}\cdots T_{k}^{1/2}$; then $T$ is a positive operator and thus
\begin{gather*}
\bigg(\Pi_{j=1}^{k}\phi_{j}\bigg)^{(n)}(u) = (T\otimes I_n)((\Pi_{j=1}^k \hat{\phi}_{j})(u))(T\otimes I_n)   \geq 0,
\end{gather*}
\noindent implying $u \in M_{n}(\cl S)^{+}$. 

Finally, relax the assumption on the invertibility of the operators $T_{j}$, $j\in [k]$. 
Let $\epsilon > 0$; then $T_{j}+\epsilon I$ (positive and) invertible, $j\in [k]$. 
Let $f_{j}: \cl{S}_{j} \rightarrow \bb{C}$ be a state, and define the map $\hat{\phi_{j}}: \cl{S}_{j} \rightarrow \cl B(H)$ by letting
\begin{gather*}
	\hat{\phi_{j}}(u) =\phi_{j}(u)+\epsilon f_{j}(u)I, \ \ \ u\in \cl{S}_{j}, \, j\in [k]. 
\end{gather*}
We have that $(\hat{\phi_{j}})_{j=1}^{k}$ is a commuting family
of completely positive maps. 
Furthermore, 
$\hat{\phi_{j}}(1_{\cl{S}_{j}}) = T_{j}+\epsilon I$, 
and so $\hat{\phi_{j}}(1_{\cl{S}_{j}})$ is invertible, $j \in [k]$. 
By the previous paragraph, $(\Pi_{j=1}^k \hat{\phi}_{j})^{(n)}(u)\in M_n(\cl B(H))^+$. 
Letting $\epsilon\to 0$, we obtain 
$(\Pi_{j=1}^k \phi_{j})^{(n)}(u)\in M_n(\cl B(H))^+$; 
thus $u \in M_{n}(\cl{S})^{+}$ and the proof is complete.
\end{proof}


\subsection{Representations of correlations via operator systems}\label{ss_reptensor}

In this subsection, we describe the correlations from the classes 
$\cl C_{\rm sns}$, $\cl C_{\rm sqc}$ and $\cl C_{\rm sqa}$ in terms of states on 
operator system tensor products. 
Recall that $\tilde{e}_{x,y}$ are the canonical generators of the operator system $\cl S_{X,Y}$ and, 
for a linear functional 
$$s : \cl S_{X_2,X_1}\otimes \cl S_{Y_2,Y_1}\otimes \cl S_{A_1,A_2}\otimes \cl S_{B_1,B_2} \to \bb{C},$$
let 
$\Gamma : \cl D_{X_2Y_2}\otimes \cl D_{A_1 B_1} \to \cl D_{X_1Y_1}\otimes \cl D_{A_2 B_2}$
be the linear map, given by 
$$\Gamma_s\left(x_1y_1, a_2b_2 | x_2y_2, a_1b_1\right) := 
s\left(\tilde{e}_{x_2,x_1}\otimes \tilde{e}_{y_2,y_1}\otimes \tilde{e}_{a_1,b_1} \otimes \tilde{e}_{a_2,b_2}\right).$$

\begin{theorem}\label{th_snsqc}
The map $s\to \Gamma_s$ is an affine isomorphism from 
\begin{itemize}
\item[(i)]
the state space of 
$\cl S_{X_2,X_1}\otimes_{\max} \cl S_{Y_2,Y_1}\otimes_{\max} \cl S_{A_1,A_2}\otimes_{\max} \cl S_{B_1,B_2}$
onto $\cl C_{\rm sns}$;

\item[(ii)]
the state space of 
$\cl S_{X_2,X_1}\otimes_{\rm c} \cl S_{Y_2,Y_1}\otimes_{\rm c} \cl S_{A_1,A_2}\otimes_{\rm c} \cl S_{B_1,B_2}$
onto $\cl C_{\rm sqc}$;

\item[(iii)]
the state space of 
$\cl S_{X_2,X_1}\otimes_{\min} \cl S_{Y_2,Y_1}\otimes_{\min} \cl S_{A_1,A_2}\otimes_{\min} \cl S_{B_1,B_2}$
onto $\cl C_{\rm sqa}$. 
\end{itemize}
\end{theorem}

\begin{proof}
(i) Let $\Gamma\in \cl C_{\rm sns}$, and consider it as an element of 
$\cl D_{X_2X_1}\otimes_{\min}\cl D_{Y_2Y_1}\otimes_{\min}\cl D_{A_1A_2}\otimes_{\min}\cl D_{B_1B_2}$. 
For an element $\omega\in \cl D_{Y_2Y_1}\otimes_{\min}\cl D_{B_1B_2}$, let 
$$L_{\omega} : 
\cl D_{X_2X_1}\otimes_{\min}\cl D_{Y_2Y_1}\otimes_{\min}\cl D_{A_1A_2}\otimes_{\min}\cl D_{B_1B_2}
\to \cl D_{X_2X_1}\otimes_{\min}\cl D_{A_1A_2}$$ 
be the corresponding slice map. 
Based on the remarks preceeding Definition \ref{d_snsc}, the strongly no-signalling conditions imply that, 
if $\omega = \epsilon_{y_2,y_2}\otimes \epsilon_{y_1,y_1}\otimes \epsilon_{b_1,b_1}\otimes \epsilon_{b_2,b_2}$ then 
$L_{\omega}(\Gamma)\in \cl C_{\rm ns}$. 
By \cite[Theorem 3.1]{lmprsstw}, 
$$L_{\omega}(\Gamma)\in \cl R_{X_2X_1}\otimes_{\min}\cl R_{A_1A_2}, 
\ \ \ \omega\in \cl D_{Y_2Y_1}\otimes_{\min}\cl D_{B_1B_2}.$$
By symmetry, 
$$L_{\omega'}(\Gamma)\in \cl R_{Y_2Y_1}\otimes_{\min}\cl R_{B_1B_2}, 
\ \ \ \omega'\in \cl D_{X_2X_1}\otimes_{\min}\cl D_{A_1A_2}.$$
It follows that 
$$\Gamma\in \cl R_{X_2,X_1}\otimes_{\min}\cl R_{Y_2,Y_1}\otimes_{\min}
\cl R_{A_1,A_2}\otimes_{\min}\cl R_{B_1,B_2}.$$
On the other hand, relation (\ref{eq_SdR}) implies that 
$$\hspace{-2.5cm} \left(\cl S_{X_2,X_1}\otimes_{\max} \cl S_{Y_2,Y_1}\otimes_{\max} 
\cl S_{A_1,A_2}\otimes_{\max} \cl S_{B_1,B_2}\right)^{\rm d} 
\cong_{\rm c.o.i.}$$ 
$$\hspace{3.5cm} \cl R_{X_2,X_1}\otimes_{\min}\cl R_{Y_2,Y_1}\otimes_{\min}
\cl R_{A_1,A_2}\otimes_{\min}\cl R_{B_1,B_2}.$$
The claim is proved. 

\smallskip

(ii) 
Suppose that 
$$s : \cl S_{X_2,X_1}\otimes_{\rm c} \cl S_{Y_2,Y_1}\otimes_{\rm c} \cl S_{A_1,A_2}\otimes_{\rm c} \cl S_{B_1,B_2}
\to \bb{C}$$ 
is a state. By Theorem \ref{l_kas}, we may assume that $s$ is the restriction of  a state 
$$\tilde{s} : \cl A_{X_2,X_1}\otimes_{\max} \cl A_{Y_2,Y_1}\otimes_{\max} \cl A_{A_1,A_2}\otimes_{\max} \cl A_{B_1,B_2}
\to \bb{C}.$$
The GNS construction applied to $s$ 
produces a Hilbert space $H$, a unit vector $\xi\in H$ and mutually commuting projection-valued measures
$(E_{x_2,x_1})_{x_1\in X_1}$, $(E^{y_2,y_1})_{y_1\in Y_1}$, 
$(F_{a_1,a_2})_{a_2\in A_2}$ and $(F^{b_1,b_2})_{b_2\in B_2}$ on $H$
such that 
$$\Gamma_s\left(x_1y_1, a_2b_2 | x_2y_2, a_1b_1\right)
= \left\langle E_{x_2,x_1}E^{y_2,y_1}F_{a_1,a_2}F^{b_1,b_2}\xi,\xi\right\rangle,$$
for all $x_i,y_i,a_i,b_i$, $i = 1,2$. 
Setting $P_{x_2y_2,x_1y_1} = E_{x_2,x_1}E^{y_2,y_1}$ and 
$Q_{a_1b_1,a_2b_2} = F_{a_1,a_2}F^{b_1,b_2}$, we have that
the NS operator matrices $P = (P_{x_2y_2,x_1y_1})_{x_2,x_1,y_2,y_1}$ 
and $Q = (Q_{a_1a_2,b_1b_2})_{a_1,a_2,b_1,b_2}$ are dilatable and have mutually commuting entries; thus, 
$\Gamma_s\in \cl C_{\rm sqc}$. 

Conversely, suppose that $\Gamma\in \cl C_{\rm sqc}$ and use Lemma \ref{l_comlift} to 
write $\Gamma$ in the form
(\ref{eq_qcsns2}) for some mutually commuting PVM's 
$P_X = (P_{x_2,x_1})_{x_1\in X_1}$, $P_Y = (P^{y_2,y_1})_{y_1\in Y_1}$, 
$P_A = (P_{a_1,a_2})_{a_2\in A_2}$, 
$P_B = (P^{b_1,a_2})_{a_2\in A_2}$, $x_2\in X_2$, $y_2\in Y_2$, $a_1\in A_1$, $b_1\in B_1$. 
Let $\pi_X$, $\pi_Y$, $\pi_A$ and $\pi_B$ be the unital *-representations of  
$\cl A_{X_2,X_1}$, $\cl A_{Y_2,Y_1}$, $\cl A_{A_1,A_2}$ and $\cl A_{B_1,B_2}$, arising 
canonically from 
$P_X$, $P_Y$, $P_A$ and $P_B$, respectively.
Then $\pi := \pi_X\otimes \pi_Y\otimes \pi_A\otimes \pi_B$ is a unital *-representation 
of the C*-algebra 
$\cl A := \cl A_{X_2,X_1}\otimes_{\max} \cl A_{Y_2,Y_1}\otimes_{\max} \cl A_{A_1,A_2}\otimes_{\max} \cl A_{B_1,B_2}$. 
Using Theorem \ref{l_kas}, let $s$ be the restriction to 
$\cl S_{X_2,X_1}\otimes_{\rm c} \cl S_{Y_2,Y_1}\otimes_{\rm c} \cl S_{A_1,A_2}\otimes_{\rm c} \cl S_{B_1,B_2}$
of the state on $\cl A$, given by $w\to \langle \pi(w)\eta,\eta\rangle$.
We have that $\Gamma = \Gamma_s$.

\smallskip

(iii) 
Let $\Gamma$ be a quantum SNS correlation; without loss of generality, we may thus assume 
that 
$$\Gamma(x_1y_1,a_2b_2 | x_2y_2,a_1b_1) = 
\left\langle \left(M_{x_2,x_1}\otimes M^{a_1,a_2}\otimes N_{y_2,y_1}\otimes N^{b_1,b_2}\right)\eta,\eta\right\rangle,$$
for all $x_i\in X_i$, $y_i\in Y_i$, $a_i\in A_i$ and $b_i\in B_i$, $i = 1,2$, 
where the families 
$M_X = (M_{x_2,x_1})_{x_1\in X_1}$, $M_A = (M^{a_1,a_2})_{a_2\in A_2}$, 
$N_Y = (N_{y_2,y_1})_{y_1\in Y_1}$ and 
$N_B = (N^{b_1,b_2})_{b_2\in B_2}$ are POVM's acting on finite dimensional Hilbert spaces 
$H_X$, $H_A$, $H_Y$ and $H_B$, respectively, and $\eta\in H_X\otimes H_A\otimes H_Y\otimes H_B$ is a unit vector.
After an application of Naimark's Theorem, we can further assume that 
$M_X$, $M_A$, $N_Y$ and $N_B$ are PVM's. 
Let $\pi_X : \cl A_{X_2,X_1}\to \cl B(H_X)$, 
$\pi_Y : \cl A_{Y_2,Y_1}\to \cl B(H_Y)$, 
$\pi_A : \cl A_{A_1,A_2}\to \cl B(H_A)$
and 
$\pi_B : \cl A_{B_1,B_2}\to \cl B(H_B)$ 
be the unital *-representations, canonically arising from $M_X$, $M_A$, $N_Y$ and $N_B$, respectively. 
Letting $\pi = \pi_X\otimes\pi_Y\otimes \pi_A\otimes \pi_B$, $\tilde{\eta}\in H_X\otimes H_Y\otimes H_A\otimes H_B$
be the vector obtained from $\eta$ after applying the canonical shuffle, 
$\tilde{s}$ be the state on 
$\cl A_{X_2,X_1}\otimes_{\min}\cl A_{Y_2,Y_1}\otimes_{\min}\cl A_{A_1,A_2}\otimes_{\min}\cl A_{B_1,B_2}$
given by 
$\tilde{s}(w) = \langle \pi(w)\tilde{\eta},\tilde{\eta}\rangle$, 
and $s$ be its restriction to 
$\cl S_{X_2,X_1}\otimes_{\min}\cl S_{Y_2,Y_1}\otimes_{\min}\cl S_{A_1,A_2}\otimes_{\min}\cl S_{B_1,B_2}$, 
we obtain that $\Gamma = \Gamma_s$. 

Suppose that $\Gamma\in \cl C_{\rm sqa}$, and let $(\Gamma_n)_{n\in \bb{N}}$ be a sequence of 
quantum SNS correlations with $\Gamma_n\to_{n\to \infty} \Gamma$. Using the previous paragraph, 
choose a state $s_n$ of $\cl S_{X_2,X_1}\otimes_{\min}\cl S_{Y_2,Y_1}\otimes_{\min}\cl S_{A_1,A_2}\otimes_{\min}\cl S_{B_1,B_2}$ such that $\Gamma_n = \Gamma_{s_n}$, $n\in \bb{N}$. 
Letting $s$ be a cluster point of the sequence $(s_n)_{n\in \bb{N}}$ in the weak* topology, we have that 
$\Gamma = \Gamma_s$. 

Conversely, let 
$s : \cl S_{X_2,X_1}\otimes_{\min}\cl S_{Y_2,Y_1}\otimes_{\min}\cl S_{A_1,A_2}\otimes_{\min}\cl S_{B_1,B_2}\to \bb{C}$
be a state, and, using the injectivity of the minimal operator system tensor product
(pointed out in Remark \ref{r_funcmm}), let 
$\tilde{s} : \cl A_{X_2,X_1}\otimes_{\min}\cl A_{Y_2,Y_1}\otimes_{\min}\cl A_{A_1,A_2}\otimes_{\min}\cl A_{B_1,B_2}\to \bb{C}$
be an extension of $s$. 
By \cite[Corollary 4.3.10]{kadison-ringrose}, $s$ is in the weak* closure of the convex hull of 
vector functionals on 
$\pi_X(\cl A_{X_2,X_1})\otimes_{\min} \pi_Y(\cl A_{Y_2,Y_1})\otimes_{\min}
\pi_A(\cl A_{A_1,A_2})\otimes_{\min}\pi_B(\cl A_{B_1,B_2})$ for some unital *-representations 
$\pi_X$, $\pi_Y$, $\pi_A$ and $\pi_B$. The argument given in the proof of \cite[Theorem 2.10]{pt}
can now be used to show that $\Gamma$ is a limit of quantum SNS correlations. 
\end{proof}


\section{Synchronous games}\label{s_special}

A \emph{synchronous game} \cite{psstw} over a quadruple $(X,X,A,A)$
is a non-local game $E\subseteq XX\times AA$ such that 
$$(x,x,a,b)\in E \ \ \Longrightarrow \ \ a = b.$$
A perfect strategy for a synchronous game $E\subseteq XX\times AA$
is called a \emph{synchronous correlation} over $(X,X,A,A)$. 

Assume that $Y_i = X_i$ and $B_i = A_i$, $i = 1,2$.
In this section, we restrict our attention to the case where the games 
$E_1\subseteq X_1X_1\times A_1A_1$ and 
$E_2\subseteq X_2X_2\times A_2A_2$ participating in 
a quasi-homorphism game $E_1\leadsto E_2$ are synchronous.
We achieve tracial representations of SNS quantum commuting and approximately quantum correlations, 
which lead to necessary conditions for the existence of an 
isomorphism between two synchronous games in terms of the corresponding game algebras. 
As in the paragraph before Definition \ref{d_snsc}, we have that, if 
$\Gamma$ is an SNS correlation over the quadruple $(X_2 X_2, A_1A_1, X_1 X_1, A_2A_2)$, 
then the linear maps $\Gamma_{X_2X_2\to X_1X_1} : \cl D_{X_2X_2}\to \cl D_{X_1X_1}$ and 
$\Gamma^{A_2A_2\to A_1A_1} : \cl D_{A_2A_2}\to \cl D_{A_1A_1}$, given by 
\begin{equation}\label{eq_X2X2X1X1}
\Gamma_{X_2X_2\to X_1X_1}(x_1,y_1|x_2,y_2) = \sum_{a_2,b_2\in A_2} \Gamma(x_1y_1, a_2b_2 | x_2 y_2, a_1b_1)
\end{equation}
and 
\begin{equation}\label{eq_A2A2A1A1}
\Gamma^{A_1A_1\to A_2A_2}(a_2,b_2|a_1,b_1) = \sum_{x_1,y_1\in X_1} \Gamma(x_1y_1, a_2b_2 | x_2 y_2, a_1b_1),
\end{equation}
are NS correlations over $(X_2,X_2,X_1,X_1)$ and $(A_1,A_1,A_2,A_2)$, respectively.


\begin{definition}\label{d_stsyn}
An SNS correlation $\Gamma$ over $(X_2 X_2, A_1A_1, X_1 X_1, A_2A_2)$ is called 
\emph{jointly synchronous} if $\Gamma_{X_2X_2\to X_1X_1}$ and 
$\Gamma^{A_1A_1\to A_2A_2}$ are synchronous correlations. 
\end{definition}

\begin{remark}\label{r_stsyn}
\rm 
Since the terms in (\ref{eq_X2X2X1X1}) and (\ref{eq_A2A2A1A1}) are non-negative, 
an element $\Gamma\in \cl C_{\rm sns}$ is jointly synchronous precisely when 
$$\Gamma(x_1y_1, a_2b_2 | x_2 x_2, a_1b_1) \neq 0  \ \ \Longrightarrow \ \ x_1 = y_1$$
for all $x_2\in X_2$, $x_1,y_1\in X_1$, $a_1,b_1\in A_1$ and $a_2,b_2\in A_2$, and
$$\Gamma(x_1y_1, a_2b_2 | x_2 y_2, a_1a_1) \neq 0  \ \ \Longrightarrow \ \ a_2 = b_2.$$
for all $x_2,y_2\in X_2$, $x_1,y_1\in X_1$, $a_1\in A_1$ and $a_2,b_2\in A_2$.
\end{remark}

For a linear functional $\mathsf{T}$ on $\cl A_{X_2,X_1}\otimes \cl A_{A_1,A_2}$, set
$$\Gamma_{\mathsf{T}} (x_1y_1, a_2b_2 | x_2 y_2, a_1b_1) = 
\mathsf{T}(e_{x_2,x_1}e_{y_2,y_1}\otimes e_{a_1,a_2}e_{b_1,b_2}),$$
where $x_i,y_i\in X_i, a_i, b_i \in A_i, i = 1,2$.
In the sequel, we will use the terms \lq\lq trace'' and \lq\lq tracial state'' interchangeangly.

\begin{theorem}\label{th_snsssyn}
The following hold:
\begin{itemize}
\item[(i)]
If $\mathsf{T}$ is a trace on $\cl A_{X_2,X_1}\otimes_{\max}\cl A_{A_1,A_2}$ then 
$\Gamma_{\mathsf{T}}$ is a jointly synchronous and quantum commuting correlation.

\item[(ii)]
If $\Gamma\in \cl C_{\rm sqc}$ is jointly synchronous then there exists a trace $\mathsf{T}$ on 
$\cl A_{X_2,X_1}\otimes_{\max}\cl A_{A_1,A_2}$ such that $\Gamma = \Gamma_{\mathsf{T}}$.

\item[(iii)]
A jointly synchronous SNS correlation $\Gamma$ is approximately quantum if and only if 
there exists an amenable trace $\mathsf{T}$ on $\cl A_{X_2,X_1}\otimes_{\min}\cl A_{A_1,A_2}$ such that $\Gamma = \Gamma_{\mathsf{T}}$ such that $\Gamma = \Gamma_{\mathsf{T}}$.
\end{itemize}
\end{theorem}

\begin{proof}
(i) 
That $\Gamma_{\mathsf{T}}$ is strongly no-signalling is straightforward. 
The GNS construction, applied to $\mathsf{T}$, yields a Hilbert space $H$, a unit vector $\xi\in H$ and 
a representation $\pi : \cl A_{X_2,X_1}\otimes_{\max}\cl A_{A_1,A_2}\to \cl B(H)$ such that 
$\mathsf{T}(w) = \langle \pi(w)\xi,\xi\rangle$, $w\in \cl A_{X_2,X_1}\otimes_{\max}\cl A_{A_1,A_2}$. 
Set 
$$P_{x_2y_2,x_1y_1} = \pi(e_{x_2,x_1}e_{y_2,y_1}\otimes 1) \ \mbox{ and } 
\ Q_{a_1b_1,a_2b_2} = \pi(1 \otimes e_{a_1,a_2}e_{b_1,b_2}),$$
for all $x_i,y_i\in X_i$ and $a_i,b_i\in A_i$, $i = 1,2$. 
Proposition \ref{p_dedil}, applied to the map 
$\phi : \cl S_{X_2,X_1}\otimes_{\rm c} \cl S_{X_2,X_1}\to \cl B(H)$, given by 
$$\phi(e_{x_2,x_1}\otimes e_{y_2,y_1}) = P_{x_2y_2,x_1y_1}, \ \ \ x_i,y_i\in X_i, i = 1,2,$$
shows that $(P_{x_2y_2,x_1y_1})$ is a dilatable operator matrix; by symmetry, so is $(Q_{a_1b_1,a_2b_2})$. 
Thus, $\Gamma_{\mathsf{T}}$ is a quantum commuting SNS correlation. 
Finally, the strong synchronicity is immediate from Remark \ref{r_stsyn} and the fact that 
$(e_{x_2,x_1})_{x_1\in X_1}$ and $(e_{a_1,a_2})_{a_2\in A_2}$ are PVM's. 

\smallskip

(ii) 
For brevity, write 
$$\frak{B} = \cl A_{X_2,X_1}\otimes_{\max} \cl A_{X_2,X_1}\otimes_{\max} \cl A_{A_1,A_2}\otimes_{\max} \cl A_{A_1,A_2}$$
and 
$$\frak{A} = \cl A_{X_2,X_1}\otimes_{\max} \cl A_{A_1,A_2}.$$
Note that, up to a flip of the tensor terms, $\frak{B} = \frak{A}\otimes_{\max} \frak{A}$; 
in the rest of the proof, the latter identification is used without explicit mention. 

By the (proof of) Theorem \ref{th_snsqc}, there exists a state
$\tilde{s} : \frak{B}\to \bb{C}$ such that $\Gamma = \Gamma_{\tilde{s}}$. 
If $w_1,w_2\in \frak{A}\otimes_{\max}\frak{A}$, write 
$w_1\sim w_2$ if $\tilde{s}(w_1 - w_2) = 0$. 
For $x_i\in X_i$ and $a_i,b_i\in A_i$, $i = 1,2$, we have 
\begin{eqnarray*}
\tilde{s}(e_{x_2,x_1}\otimes 1\otimes e_{a_1,a_2}\otimes e_{b_1,b_2}) 
& = & 
\sum_{y_1\in X_1}
\tilde{s}(e_{x_2,x_1}\otimes e_{x_2,y_1}\otimes e_{a_1,a_2}\otimes e_{b_1,b_2})\\
& = & 
\tilde{s}(e_{x_2,x_1}\otimes e_{x_2,x_1}\otimes e_{a_1,a_2}\otimes e_{b_1,b_2})\\
& = & 
\Gamma(x_1x_1, a_2b_2 | x_2 x_2, a_1b_1)\\
& = & 
\sum_{y_1\in X_1}
\tilde{s}(e_{x_2,y_1}\otimes e_{x_2,x_1}\otimes e_{a_1,a_2}\otimes e_{b_1,b_2})\\
& = & 
\tilde{s}(1 \otimes e_{x_2,x_1}\otimes e_{a_1,a_2}\otimes e_{b_1,b_2});
\end{eqnarray*}
thus, 
\begin{eqnarray*}
e_{x_2,x_1}\otimes 1\otimes e_{a_1,a_2}\otimes e_{b_1,b_2}
& \sim & e_{x_2,x_1}\otimes e_{x_2,x_1}\otimes e_{a_1,a_2}\otimes e_{b_1,b_2}\\
& \sim & 1 \otimes e_{x_2,x_1}\otimes e_{a_1,a_2}\otimes e_{b_1,b_2}.
\end{eqnarray*}
It follows that 
$$e_{x_2,x_1}\otimes 1\otimes 1\otimes 1
\sim e_{x_2,x_1}\otimes e_{x_2,x_1}\otimes 1 \otimes 1 
\sim 1 \otimes e_{x_2,x_1}\otimes 1 \otimes 1,$$
for $x_i\in X_i, i = 1,2$. 
Write 
$$h = e_{x_2,x_1}\otimes 1\otimes 1\otimes 1 - 1 \otimes e_{x_2,x_1}\otimes 1 \otimes 1,$$
and note that 
\begin{eqnarray*}
h^2 & = & 
e_{x_2,x_1}\otimes 1\otimes 1\otimes 1 - e_{x_2,x_1} \otimes e_{x_2,x_1}\otimes 1 \otimes 1\\ 
& & 
\hspace{2cm} - e_{x_2,x_1} \otimes e_{x_2,x_1}\otimes 1 \otimes 1 + 1 \otimes e_{x_2,x_1}\otimes 1 \otimes 1,
\end{eqnarray*}
implying $h^2 \sim 0$. 
An application of the Cauchy-Schwarz inequality now shows that 
$w h \sim 0$ and $h w \sim 0$ whenever  $w\in \frak{B}$.
In particular, for all $x_i\in X_i$, $i = 1,2$, we have 
\begin{equation}\label{eq_zeez}
ze_{x_2,x_1}\otimes 1 \otimes v \sim z\otimes e_{x_2,x_1} \otimes v 
\sim ze_{x_2,x_1}\otimes 1 \otimes v, \ \ z\in \cl A_{X_2,X_1}, v\in \frak{A}.
\end{equation}
Similarly, 
\begin{equation}\label{eq_zeez2}
u \otimes z'e_{a_1,a_2} \otimes 1 \sim u\otimes z'\otimes e_{a_1,a_2}
\sim u \otimes e_{a_1,a_2} z' \otimes 1, \ \ z'\in \cl A_{A_1,A_2}, u\in \frak{A}.
\end{equation}
Equations (\ref{eq_zeez}) and (\ref{eq_zeez2}) imply 
\begin{equation}\label{eq_zeez3}
ze_{x_2,x_1}\otimes 1 \otimes z'e_{a_1,a_2}\otimes 1
\sim 
e_{x_2,x_1}\otimes z \otimes e_{a_1,a_2}\otimes z'
\sim e_{x_2,x_1}z\otimes 1 \otimes e_{a_1,a_2}z' \otimes 1
\end{equation}
for all $z\in \cl A_{X_2,X_1}$, all $z'\in \cl A_{A_1,A_2}$ and all 
$x_i\in X_i$, $a_i\in A_i$, $i = 1,2$. 
An induction by the length of the words 
$w$ on $\{e_{x_2,x_1} : x_i\in X_i, i = 1,2\}$ and 
$w'$ on $\{e_{a_1,a_2} : a_i\in A_i, i = 1,2\}$, whose base step is provided by (\ref{eq_zeez3})
shows that 
$$zw\otimes 1 \otimes z'w'\otimes 1 \sim wz\otimes 1 \otimes w' z' \otimes 1, \ \ \ 
z,z'\in \cl A_{X_2,X_1}, w,w'\in \cl A_{A_1,A_2}$$
(see the proof of \cite[Theorem 6.1]{lmprsstw}).
We conclude that the functional 
$\mathsf{T}$ on $\cl A_{X_2,X_1}\otimes_{\max}\cl A_{A_1,A_2}$, given by 
$$\mathsf{T}(u\otimes v) = \tilde{s}(u\otimes 1\otimes v\otimes 1), \ \ \ u\in \cl A_{X_2,X_1}, v\in \cl A_{A_1,A_2},$$
is a tracial state. 
The fact that $\Gamma = \Gamma_{\mathsf{T}}$ follows from (\ref{eq_zeez3}). 

\smallskip

(iii) 
Suppose that $\Gamma$ is an approximately quantum jointly synchronous SNS correlation. 
By Theorem \ref{th_snsqc}, there exists a state $s$ on 
$\cl A_{X_2,X_1}\otimes_{\min}\cl A_{X_2,X_1}\otimes_{\min}\cl A_{A_1,A_2}\otimes_{\min}\cl A_{A_1,A_2}$
such that $\Gamma = \Gamma_s$. 
Using (i), let $\mathsf{T} : \cl A_{X_2,X_1}\otimes_{\max}\cl A_{A_1,A_2}\to \bb{C}$ be a tracial state 
such that $\Gamma = \Gamma_{\mathsf{T}}$. 
By \cite[Lemma III.3]{kps}, there exists a *-isomorphism $\partial: \cl A_{X_2,X_1}\to \cl A_{X_2,X_1}^{\rm op}$ 
such that $\partial(e_{x_2,x_1}) = e_{x_2,x_1}^{\rm op}$. 
Let $q : \cl A_{X_2,X_1}\otimes_{\max}\cl A_{A_1,A_2}\to \cl A_{X_2,X_1}\otimes_{\min}\cl A_{A_1,A_2}$ be the 
quotient map, 
$$\hspace{-3cm}\cl F : 
\cl A_{X_2,X_1}\otimes_{\min}\cl A_{A_1,A_2}\otimes_{\min}\cl A_{X_2,X_1}\otimes_{\min}\cl A_{A_1,A_2}$$
$$\hspace{4cm} \longrightarrow 
\cl A_{X_2,X_1}\otimes_{\min}\cl A_{X_2,X_1}\otimes_{\min}\cl A_{A_1,A_2}\otimes_{\min}\cl A_{A_1,A_2}$$
be the flip operation, 
and 
$$\mu : (\cl A_{X_2,X_1}\otimes_{\max}\cl A_{A_1,A_2})\otimes_{\min} 
(\cl A_{X_2,X_1}\otimes_{\max}\cl A_{A_1,A_2})^{\rm op} \to \bb{C}$$
be the linear functional, defined by letting
$$\mu = s\circ \cl F\circ \left(q\otimes \left(q\circ (\partial^{-1}\otimes\partial^{-1})\right)\right).$$
It is then straightforward to check that 
$$\mu(u\otimes v^{\rm op}) = \mathsf{T}(uv), \ \ \ u,v \in \cl A_{X_2,X_1}\otimes_{\max}\cl A_{A_1,A_2}.$$
By \cite[Theorem 6.2.7]{bo}, $\mathsf{T}$ is amenable.
The converse direction follows by reversing these steps. 
\end{proof}

\begin{proposition}\label{l_tracetotrace}
If $\Gamma$ is a jointly synchronous SNS correlation then 
$\Gamma[\cl E]$ is synchronous whenever $\cl E$ is such. 
\end{proposition}

\begin{proof}
Let $\cl E$ be a synchronous no-signalling correlation over $(X_1,X_1,A_1,A_1)$. 
Suppose that $x_2\in X_2$ and $a_2,b_2\in A_2$. We have 
\begin{eqnarray*}
\Gamma[\cl E] (a_2,b_2|x_2,x_2)
& = & 
\hspace{-0.5cm}
\sum_{x_1,y_1\in X_1} \sum_{a_1,b_1\in A_1}
\Gamma(x_1y_1, a_2b_2 | x_2x_2, a_1b_1)\cl E(a_1,b_1|x_1,y_1)\\
& = & 
\hspace{-0.5cm}
\sum_{x_1\in X_1} \sum_{a_1,b_1\in A_1}
\Gamma(x_1x_1, a_2b_2 | x_2x_2, a_1b_1)\cl E(a_1,b_1|x_1,x_1)\\
& = & 
\hspace{-0.5cm}
\sum_{x_1\in X_1} \sum_{a_1\in A_1}
\Gamma(x_1x_1, a_2b_2 | x_2x_2, a_1a_1)\cl E(a_1,a_1|x_1,x_1).
\end{eqnarray*}
Since $\Gamma$ is jointly synchronous, 
$\Gamma(x_1x_1, a_2b_2 | x_2x_2, a_1a_1) = 0$ whenever $a_2\neq b_2$. 
Thus, 
$\Gamma[\cl E] (a_2,b_2|x_2,x_2) = 0$ whenever $a_2\neq b_2$, that is, 
$\Gamma[\cl E]$ is synchronous. 
\end{proof}

\begin{remark}\label{r_Ttau}
\rm 
Let $\mathsf{T}$ be a tracial state on $\cl A_{X_2,X_1}\otimes_{\max}\cl A_{A_1,A_2}$ and 
$\tau$ be a tracial state on $\cl A_{X_1,A_1}$.
Let $\cl E_{\tau} : \cl D_{X_1X_1}\to \cl D_{A_1A_1}$ be the quantum commuting 
no-signalling correlation \cite{psstw}, given by 
\begin{equation}\label{eq_Etau}
\cl E_{\tau}(a,b | x,y) = \tau(e_{x,a}e_{y,b}), \ \ \ x,y\in X_1, a,b\in A_1.
\end{equation}
By Theorem \ref{th_pres}, $\Gamma_{\mathsf{T}}[\cl E_{\tau}]$ is a quantum commuting NS correlation while, by 
Proposition \ref{l_tracetotrace}, it is synchronous. By \cite[Theorem 5.5]{psstw}, there exists a tracial state
$\mathsf{T}[\tau] : \cl A_{X_2,A_2} \to \bb{C}$ such that 
\begin{equation}\label{eq_Ttau}
\Gamma_{\mathsf{T}}[\cl E_{\tau}] = \cl E_{\mathsf{T}[\tau]}.
\end{equation}
Let, on the other hand, 
$$\mathsf{T}\odot \tau : \cl A_{X_2,X_1}\otimes_{\max}\cl A_{X_1,A_1} \otimes_{\max} \cl A_{A_1,A_2}\to \bb{C}$$
be the tracial state, given by 
$$(\mathsf{T}\odot \tau)(u\otimes v \otimes w) = \mathsf{T}(u \otimes w) \tau(v), \ \ 
u\in \cl A_{X_2,X_1}, v\in \cl A_{X_1,A_1}, w\in \cl A_{A_1,A_2}.$$
Equation (\ref{eq_Ttau}) implies 
$$\mathsf{T}[\tau](e_{x_2,a_2}e_{y_2,b_2})
= \hspace{-0.2cm} \sum_{x_1,y_1}\sum_{a_1,b_1}
(\mathsf{T}\odot \tau)(e_{x_2,x_1}e_{y_2,y_1}
\otimes e_{x_1,a_1} e_{y_1,b_1}\otimes 
e_{a_1,a_2} e_{b_1,b_2}).$$
\end{remark}

We recall synchronous game $E\subseteq XX\times AA$ gives rise to the C*-algebra \cite{hmps}
$\cl A(E) = \cl A_{X,A}/\cl J_E$, where $\cl J_E$ is the closed ideal of $\cl A_{X,A}$, generated by the set 
$\{e_{x,a}e_{y,b} : (x,y,a,b)\not\in E\}$
(the C*-algebra $\cl A(E)$ is known as the \emph{game C*-algebra} of $E$).
Write $q_E$ for the quotient map from $\cl A_{X,A}$ onto $\cl A(E)$. 
We denote by ${\rm T}_2(\cl A(E))$ the convex set of all restrictions of 
tracial states on a C*-algebra $\cl A(E)$ to the subspace 
$$\cl A_{(2)}(E) := {\rm span}\{q(e_{x,a}e_{y,b}) : x,y\in X, a,b\in A\}.$$
By \cite[Theorem 3.2]{hmps}, 
the perfect quantum commuting strategies for $E$ are in correspondence with the elements of 
${\rm T}(\cl A(E))$, by associating with every $\tau\in {\rm T}(\cl A(E))$ the correlation 
defined in (\ref{eq_Etau}).

\begin{corollary}\label{c_syngame}
Let $E_i\subseteq X_iX_i\times A_iA_i$ be a synchronous game, $i = 1,2$. 
If $E_1\leadsto_{\rm qc} E_2$ then $\tau\to \Gamma[\tau]$ is an affine map from 
${\rm T}_2(\cl A(E_1))$ into ${\rm T}_2(\cl A(E_2))$.
\end{corollary}

\begin{proof}
For a tracial state $\tau$ on $\cl A(E)$, let $\tau_{(2)}$ be its restriction to the subspace $\cl A^{(2)}(E)$ of 
$\cl A(E)$. 
By Proposition \ref{l_tracetotrace} and Remark \ref{r_Ttau}, it suffices to show that the map 
$${\rm T}_2(\cl A(E_1)) \to {\rm T}_2(\cl A(E_2)); \ \ \ \tau_{(2)}\to \mathsf{T}[\tau]_{(2)},$$
is well-defined. 
This is immediate from Theorem \ref{th_pres} (ii) and the discussion in Remark \ref{r_Ttau}. 
\end{proof}

\medskip

\noindent 
{\bf Acknowledgements. }
The authors are grateful to Martti Karvonen for useful discussions on the topic of this paper and for 
bringing up to their attention the work \cite{bkm}. 
IT acknowledges the support of the NSF through grant DMS-2154459.



\begin{thebibliography}{99}

\bibitem{afls}
{\sc A.\,Ac\'in, T.\,Fritz, Tobias, A.\,Leverrier and A.\,B.\,Sainz},
{\it A combinatorial approach to nonlocality and contextuality},
{\rm Comm. Math. Phys. 334 (2015), no. 2, 533-628}.

\bibitem{art}
{\sc R.\,Araiza, T.\,Russell and M.\,Tomforde},
{\it A universal representation for quantum commuting correlations},
{\rm preprint (2021), arXiv:2102.05827}.

\bibitem{art2}
{\sc R.\,Araiza, T.\,Russell and M.\,Tomforde},
{\it Matricial Archimedean order unit spaces and quantum correlations},
{\rm preprint (2021), arXiv:2109.11671}.

\bibitem{amrssv}
{\sc A.\,Atserias, L.\,Man\v{c}inska, D.\,E.\,Roberson, D.\,E.\,\v{S}\'amal, S.\,Severini and A.\,Varvitsiotis},
{\it Quantum and non-signalling graph isomorphisms}, 
{\rm J. Combin. Theory Ser. B 136 (2019), 289-328}.

\bibitem{bkm}
{\sc R. S. Barbosa, M. Karvonen and S. Mansfield},
{\it Closing Bell: Boxing black box simulations in the resource theory of contextuality},
{\rm preprint (2021), arXiv:2104.11241}.



\bibitem{bell}
{\sc J.\,S.\,Bell}, 
{\it  On the Einstein Podolsky Rosen paradox}, 
{\rm Phys. Phys. Fiz. 1 (1964), no. 3, 195-200}.

\bibitem{blecher}
{\sc D.\,P.\,Blecher},
{\it The Shilov boundary of an operator space and the characterization theorems},
{\rm J. Funct. Anal. 182 (2001), no. 2, 280-343}.

\bibitem{boca} 
{\sc F.\,Boca},
{\it Free products of completely positive maps and spectral sets},
{\rm J. Funct. Anal. 97 (1991), 251-263}.


\bibitem{bhtt}
{\sc M.\,Brannan, S.\,Harris, I.\,G.\,Todorov and L.\,Turowska},
{\it Synchronicity for quantum non-local games},
{\rm preprint (2021), arXiv:2106.11489}.

\bibitem{bhtt2}
{\sc M.\,Brannan, S.\,Harris, I.\,G.\,Todorov and L.\,Turowska},
{\it Bicorrelations and quantum graph isomorphisms},
{\rm preprint}.

\bibitem{bo}
{\sc N.\,P.\,Brown and N.\,Ozawa},
{\it C*-algebras and finite-dimensional approximations},
{\rm American Mathematical Society, 2008}.

\bibitem{cmnsw}
{\sc P.\,J.\,Cameron, A.\,Montanaro, M.\,W.\,Newman, S.\,Severini and A.\,Winter},
{\it On the quantum chromatic number of a graph}, 
{\rm Electron. J. Combin. 14 (2007), no. 1}.

\bibitem{CE2} 
{\sc M.\,D.\,Choi and E.\,G.\,Effros}, 
{\it Injectivity and operator spaces,} 
{\rm J. Funct. Anal. 24 (1977), 156-209}.

\bibitem{chtw} 
{\sc R.\,Cleve, O.\,H\o yer, B.\,Toner and J.\,Watrous},
{\it Consequences and limits of nonlocal strategies}, 
{\rm Proceedings of the 19th Annual IEEE Conference on Computational Complexity (2004), 236-249}.


\bibitem{cleve-mittal}
{\sc R.\,Cleve and R.\,Mittal}, 
{\it Characterization of binary constraint system games}, 
{\rm Automata, Languages, and Programming, Lecture Notes in Computer Science, no. 8572, Springer  
(2014), 320-331}.

\bibitem{clmw}
{\sc T.\,S.\,Cubitt, D.\,Leung, W.\,Matthews and A.\,Winter},
{\it Zero-error channel capacity and simulation assisted via non-local correlation}, 
{\rm IEEE Trans. Inf. Theory 57 (2011), no. 8, 5509-5523}.


\bibitem{ddn}
{\sc G.\,De las Cuevas, T.\,Drescher and T.\,Netzer}, 
{\it Quantum magic squares: dilations and their limitations},
{\rm J. Math. Phys. 61 (2020), no. 11, 111704, 15 pp.}

\bibitem{dsw}
{\sc R.\,Duan, S.\,Severini and A.\,Winter}, 
{\it Zero-error communication via quantum channels, non-commutative graphs and a quantum Lov\'{a}sz $\theta$ function},
{\rm IEEE Trans. Inf. Theory 59 (2013), no. 2, 1164-1174}.

\bibitem{dw}
{\sc R.\,Duan and A.\,Winter}, 
{\it No-signalling-assisted zero-error capacity of quantum channels and an information theoretic interpretation of the Lov\'{a}sz number},
{\rm IEEE Trans. Inform. Theory 62 (2016), no. 2, 891-914}.


\bibitem{dp}
{\sc K.\,Dykema and V.\,I.\,Paulsen},
{\it Synchronous correlation matrices and Connes' embedding conjecture},
{\rm J. Math. Phys. 57 (2016), no. 1, 015214, 12 pp.}
 
 
\bibitem{dpp1}
{\sc K.\,Dykema, V.\,I.\,Paulsen and J.\,Prakash},
{\it The Delta game}, 
{\rm Quantum Inf. Comput. 18 (2018), no. 7-8, 599-616}.

\bibitem{dpp}
{\sc K.\,Dykema, V.\,I.\,Paulsen and J.\,Prakash},
{\it Non-closure of the set of quantum correlations via graphs},
{\rm Comm. Math. Phys. 365 (2019), no. 3, 1125-1142}.



 

\bibitem{fekete}
{\sc M.\,Fekete},
{\it Über die Verteilung der Wurzeln bei gewissen algebraischen Gleichungen mit ganzzahligen Koeffizienten}, 
{\rm Mathematische Zeitschrift 17 (1) (1923), 228-249}.

\bibitem{fkpt_NYJ}
{\sc D.\,Farenick, A.\,Kavruk, V.\,I.\,Paulsen and I.\,G.\,Todorov},
{\it Characterizations of the weak expectation property}, 
{\rm New York J. Math. 24A (2018), 107-135}.


\bibitem{fp} 
{\sc D.\,Farenick and V.\,I.\,Paulsen},
{\it Operator system quotients of matrix algebras and their tensor products},
{\rm Math. Scand. 111 (2012), 210-243}.

\bibitem{fritz} 
{\sc T. Fritz}, 
{\it Tsirelson's problem and Kirchberg's conjecture}, 
{\rm Rev. Math. Phys. 24 (2012), no. 5, 1250012, 67 pp.}

\bibitem{GR} 
{\sc C.\,Godsil and G.\,Royle}, 
{\it Algebraic graph theory}, 
{\rm Springer, 2001}.


\bibitem{hmps}
{\sc J.\,W.\,Helton, K.\,P.\,Meyer, V.\,I.\,Paulsen and M.\,Satriano},
{\it Algebras, synchronous games, and chromatic numbers of graphs},
{\rm New York J. Math. 25 (2019), 328-361}. 

\bibitem{jnvwy}
{\sc Z.\,Ji, A.\,Natarajan, T.\,Vidick, J.\,Wright and H.\,Yuen},
{\it MIP*=RE},
{\rm preprint (2020), arXiv:2001.04383}.

\bibitem{jnpp}
{\sc M.\,Junge, M.\,Navascues, C.\,Palazuelos, D.\,Perez-Garcia, V.\,Scholz and R.\,F.\,Werner}, 
{\it Connes' emnedding problem and Tsirelson's problem}, 
{\rm J. Math. Phys. 52, 012102 (2011), 12 pages}.


\bibitem{kadison-ringrose}
{\sc R.\,V.\,Kadison and J.\,R.\,Ringrose},
{\it Fundamentals of the theory of operator algebras}, 
{\rm American Mathematical Society, 1997}.

\bibitem{kt}
{\sc A. Katavolos and I. G. Todorov},
{\it Normalizers of operator algebras and reflexivity},
{\rm Proc. London Math. Soc. (3) 86 (2003), no. 2, 463-484}.


\bibitem{kavruk}
{\sc A.\,S.\,Kavruk},
{\it Nuclearity related properties in operator systems}, 
{\rm J. Operator Theory 71 (2014), no. 1, 95-156}.
 

\bibitem{kptt}
{\sc A.\,S.\,Kavruk, V.\,I.\,Paulsen, I.\,G.\,Todorov and M.\,Tomforde},
{\it Tensor products of operator systems},
{\rm J. Funct. Anal. 261 (2011), no. 2, 267-299}.



\bibitem{kps}
{\sc S.-J.\,Kim, V.\,I.\,Paulsen and C.\,Schafhauser},
{\it A synchronous game for binary constraint systems}, 
{\rm J. Math. Phys. 59 (2018), no. 3, 032201, 17 pp.}

\bibitem{lo}  
{\sc L.\,Lov\'asz}, {\it On the Shannon capacity of a graph}, 
{\rm IEEE Trans. Inf. Theory 25 (1979), no. 1, 1-7}.

\bibitem{lmprsstw}
{\sc M.\,Lupini, L.\,Man\v{c}inska, V.\,I.\,Paulsen, D.\,E.\,Roberson, G.\,Scarpa, S.\,Severini, I.\,G.\,Todorov and A.\,Winter},
{\it Perfect strategies for non-signalling games},
{\rm Math. Phys. Anal. Geom. 23 (2020), 7}.

\bibitem{lmr}
{\sc M.\,Lupini, L.\,Man\v{c}inska and D.\,E.\,Roberson},
{\it Nonlocal games and quantum permutation groups},
{\rm J. Funct. Anal. 279 (2020), no. 5, 108592, 44 pp.}

\bibitem{mptw}
{\sc L.\,Man\v{c}inska, V.\,I.\,Paulsen, I.\,G.\,Todorov and A.\,Winter},
{\it Products of synchronous games},
{\rm preprint (2021), arXiv:2109.12039}.


\bibitem{mr}
{\sc L.\,Man\v{c}inska and D.\,E.\,Roberson},
{\it Quantum homomorphisms}, 
{\rm J. Combin. Theory Ser. B 118 (2016), 228-267}.

\bibitem{mermin}
{\sc D.\,N.\,Mermin},
{\it Simple unified form for the major no-hidden-variables theorems},
{\rm Phys. Rev. Lett. 65 (1990), no. 27, 3373-3376}.

\bibitem{oz} 
{\sc N.\,Ozawa}, 
{\it About the Connes' embedding problem -- algebraic approaches},
{\rm Japan. J.  Math. 8 (2013), no. 1, 147-183}.

\bibitem{Pa} 
{\sc V.\,I.\,Paulsen}, 
{\it Completely bounded maps and operator algebras}, 
{\rm Cambridge University Press, 2002}.

\bibitem{pr}
{\sc V.\,I.\,Paulsen and M.\,Rahaman},
{\it Bisynchronous games and factorizable maps},
{\rm Ann. Henri Poincar\'{e} 22 (2021), no. 2, 593-614}.

\bibitem{psstw}
{\sc V.\,I.\,Paulsen, S.\,Severini, D.\,Stahlke, I.\,G.\,Todorov and A.\,Winter},
{\it Estimating quantum chromatic numbers},
{\rm J. Funct. Anal. 270 (2016), no. 6, 2188-2222}.

\bibitem{pt}
{\sc V.\,I.\,Paulsen and I.\,G.\,Todorov},
{\it Quantum chromatic numbers via operator systems},
{\rm Q. J. Math. 66 (2015), no. 2, 677-692}.

\bibitem{ptt}
{\sc V.\,I.\,Paulsen, I.\,G.\,Todorov and M.\,Tomforde},
{\it Operator system structures on ordered spaces}, 
{\rm Proc. Lond. Math. Soc. (3) 102 (2011), no. 1, 25-49}.

\bibitem{Pi} 
{\sc G.\,Pisier}, 
{\it Introduction to operator space theory},
{\rm Cambridge University Press, 2003}.

\bibitem{raz}
{\sc R. Raz}, 
{\it A parallel repetition theorem},
{\rm SIAM J. Comput. 27 (1998), 763-803}.

\bibitem{rieffel}
{\sc M.\,A.\,Rieffel},
{\it Morita equivalence for C*-algebras and W*-algebras},
{\rm J. Pure Appl. Algebra 5 (1974), 51-96}.

\bibitem{rs}
{\sc D.\, E.\, Roberson and S.\, Schmidt},
{\it Quantum symmetry vs. nonlocal symmetry},
{\rm preprint (2021), arxiv:2012.13328}.

\bibitem{shannon}
{\sc C. E. Shannon},
{\it The zero error capacity of a noisy channel},
{\rm IRE Trans. Inf. Theory 2 (1956), no. 3, 8-19}.


\bibitem{slofstra}
{\sc W.\,Slofstra},
{\it The set of quantum correlations is not closed},
{\rm Forum Math. Pi 7 (2019), E1}.

\bibitem{slofstra_JAMS} 
{\sc W.\,Slofstra},
{\it Tsirelson's problem and an embedding theorem for groups arising from non-local games}, 
{\rm J. Amer. Math. Soc. 33 (2020), no. 1, 1-56}.

\bibitem{tt}
{\sc I.\,G.\,Todorov and L.\,Turowska},
{\it Quantum no-signalling correlations and non-local games},
{\rm preprint (2020), arXiv: 2009.07016}.

\bibitem{whitney}
{\sc H.\,Whitney}, 
{\it Congruent graphs and the connectivity of graphs}, 
{\rm American J. Math. 54 (1932), 150-168}.

\end{thebibliography}
\end{document}